\documentclass[11pt]{article}

\usepackage[a4paper]{geometry}
\usepackage[english]{babel} 

\usepackage{mathrsfs}

\usepackage{amsmath, amssymb, amsthm, mathtools, dsfont} 
\usepackage[dvipsnames]{xcolor} 
\usepackage{graphicx, booktabs} 
\usepackage[bookmarks=false]{hyperref} 
\usepackage[normalem]{ulem}
\usepackage{comment}
\usepackage{enumerate}
\newcommand{\R}{\mathbb{R}}
\newtheorem{theorem}{Theorem}[section]
\newtheorem{definition}[theorem]{Definition}
\newtheorem{lemma}[theorem]{Lemma}

\newtheorem{remark}[theorem]{Remark}
\newtheorem{assumption}[theorem]{Assumption}
\newtheorem{proposition}[theorem]{Proposition}
\numberwithin{equation}{section}
\numberwithin{theorem}{section}

\usepackage{accents}
\newcommand{\dbtilde}[1]{\accentset{\approx}{#1}}
\newcommand{\mythmname}[1]{\emph{\textbf{{(#1)}}}}
\usepackage{comment}
\usepackage{tikz}

\usepackage{upgreek}

\newcommand{\supp}{\operatorname{supp}}

\newcommand{\beq}{\begin{equation}}
\newcommand{\eeq}{\end{equation}}
\newcommand{\beqs}{\begin{equation*}}
\newcommand{\eeqs}{\end{equation*}}
\newcommand{\ble}{\begin{lemma}}
\newcommand{\ele}{\end{lemma}}

\providecommand{\freqk}{\textcolor{black}}
\newcommand{\cH}{\mathcal{H}}
\newcommand{\cP}{\mathcal{P}}
\newcommand{\cR}{\mathcal{R}}
\newcommand{\cT}{\mathcal{T}}

\newcommand{\ad}{\operatorname{ad}}

\renewcommand{\Re}{\mathrm{Re}}
\usepackage[show]{ed}

\usepackage{soul}

\definecolor{escol}{rgb}{0,0,0.8}
\definecolor{estcol}{rgb}{0,0.5,0}
\definecolor{esnewcol}{rgb}{0,0.5,0}

\newcommand{\tfa}{\text{ for all }}
\newcommand{\tfor}{\text{ for }}

\newcommand{\ton}{\text{ on }}

\newcommand{\tand}{\text{ and }}
\newcommand{\tr}{\rm tr}

\newcommand{\Rea}{\mathbb{R}}

\newcommand{\e}{\epsilon}
\newcommand{\Jbullet}{\,\cdot\,}

\title{
The $hp$-FEM does not suffer from the pollution effect for piecewise-smooth Helmholtz problems with Gevrey regularity at boundaries
}
\author{ Jeffrey Galkowski\footnote{Department of Mathematics, University College London, London, UK. E-mail: \href{mailto:j.galkowski@ucl.ac.uk}{j.galkowski@ucl.ac.uk}.}, \,\,
 Mostafa Meliani\footnotemark[2]\footnote{Department of Mathematical Sciences, University of Bath, Bath, UK. E-mails: \href{mailto:mm4138@bath.ac.uk}{mm4138@bath.ac.uk}, \href{mailto:eas25@bath.ac.uk}{eas25@bath.ac.uk}.}, \,\, 
 Euan A.~Spence\footnotemark[2]}
\date{}
\begin{document}
\maketitle
\begin{abstract}
We consider the $hp$-FEM applied to the Helmholtz scattering problem with wavenumber $k$, truncated with a perfectly-matched layer. The scatterer consists of a combination of Dirichlet, Neumann, and penetrable obstacles together with variable coefficients. Provided that the Helmholtz solution operator is polynomially bounded in $k$, all coefficients are piecewise smooth, all boundary surfaces are Gevrey and all coefficients restricted to boundary surfaces are Gevrey together with all their normal derivatives, we show that the $hp$-FEM is quasioptimal when $p\geq 1+\e \log k$ and $hk/p$ is sufficiently small; i.e., the $hp$-FEM does not suffer from the pollution effect.
This result generalises the analogous results in 
both \cite{bernkopf2025wavenumber} (proved for piecewise analytic coefficients and analytic boundaries) and \cite{galkowski2024hp} (proved for smooth coefficients that are analytic near analytic obstacles) to a much larger class of scatterers.
\end{abstract}

\section{Introduction}

\subsection{The scattering problem and its approximation using a PML}

We consider scattering by a bounded open obstacle $\Omega_-\subset \mathbb{R}^d$ with  connected complement, $\Omega_+:=\mathbb{R}^d\setminus \overline{\Omega_-}$: given $f\in L^2_{\rm comp}(\overline{\Omega_+})$, find $u\in H^1_{\rm loc}(\overline{\Omega_+})$ such that
\begin{equation}
\label{e:edp}
-k^{-2}\operatorname{div}( A\nabla u)-n u=f\text{ in }\Omega_+,\qquad (\mathcal{B}u)|_{\partial\Omega_+}=0,\qquad (k^{-1}\partial_r-i)u=o_{r\to \infty}(r^{\frac{1-d}{2}})
\end{equation}
where $\Gamma_s := \partial\Omega_+=\Gamma_{\textup{hard}}\sqcup \Gamma_{\textup{soft}} $ with $\Gamma_{\textup{hard}}$ and $\Gamma_{\textup{soft}}$ smooth, closed hypersurfaces and
$(\mathcal{B}u)|_{\Gamma_{\textup{soft}}}= u$ in the sound-soft case and $(\mathcal{B}u)|_{\Gamma_{\textup{hard}}}= \nu\cdot (A\nabla u)$, with $\nu$ the normal to $\Gamma_s$ in the sound-hard case.

We assume that there is $\Gamma_p\Subset \Omega_+$ a smooth, closed hypersurface such that $A\in \overline{C^\infty}(\Omega_+\setminus \Gamma_p)$ is a symmetric, positive-definite matrix with real coefficients that equals the identity outside a compact set and 
$n\in \overline{C^\infty}(\Omega_+\setminus \Gamma_p)$ is a real-valued function that equals one outside a compact set. Here, a function $u\in \overline{C^\infty}(V)$ if for all connected components $V_j$ of $V$ there exists $U_j \in C^{\infty}(\Rea^d)$ such that $U_j|_{V_j}= u|_{V_j}$ (i.e., $u$ is $C^{\infty}$ ``up to the boundary").
Under these assumptions, the solution to \eqref{e:edp} exists and is unique; see \cite[\S4.2]{galkowski2025numerical} and the references therein.

We approximate the Sommerfeld radiation condition using a radial perfectly matched layer (PML):~let $\Omega_{\tr} \subset \mathbb{R}^d$ be open, bounded, connected, with $C^{1,1}$ boundary, and contain  
a closed ball containing
$\Omega_-\cup \supp (A-I)\cup \supp (n-1)$. We truncate the problem~\eqref{e:edp} to the computational domain $\Omega:=\Omega_+\cap \Omega_{\tr}$ and apply the finite-element method to the problem:~given $f\in L^2(\Omega)$, find $u\in H^1(\Omega)$ such that
\begin{equation}
\label{e:PML1}
P_ku:=-k^{-2}\operatorname{div}(A_\theta\nabla u)+k^{-2} b_\theta\cdot\nabla u-n_\theta u=f\text{ in }\Omega,\qquad (\mathcal{B}u)|_{\partial\Omega_+}=0,\qquad u|_{\partial\Omega_{\tr}}=0,
\end{equation} 
where $A_\theta$, $b_\theta$, and $n_\theta$ are defined in, e.g., \cite[\S4.4 and Definition 4.7]{galkowski2025numerical} and, in particular, equal $A$, $0$, and $n$ (respectively) in the non-PML region. For all but finitely-many values of $k$
the solution to \eqref{e:PML1} exists and is unique.
When the data $f$ is away from the PML region, the difference between the solutions of \eqref{e:edp} and \eqref{e:PML1} (measured away from the PML region) is exponentially small in $k$ \cite{GLS2}.
 
Let 
$$
\rho(k):=\sup\Big\{ \|u\|_{L^2(\Omega)}\,:\, u\text{ solves~\eqref{e:PML1} with }\, \|f\|_{L^2(\Omega)}=1\Big\}.
$$
Recall that, with the normalisation used in~\eqref{e:edp}, for all $k_0>0$ there exists $c>0$ such that  $\rho(k)\geq ck$ for $k\geq k_0$. By~\cite[Theorem 1.6]{GLS2}, for a radial PML, there exist $C, k_1>0$ such that for $k>k_1$ and $\chi\equiv 1$ on the convex hull of $\Omega$, 
\beq\label{e:PMLcontrol}
\rho(k)\leq C\sup\Big\{ \|\chi u\|_{L^2(\Omega^+)}\,:\, u\text{ solves~\eqref{e:edp} with }\, \|\chi f\|_{L^2(\Omega^+)}=1\Big\};
\eeq
i.e. the PML solution operator is controlled by the scattering solution operator.  

\subsection{Gevrey-type regularity assumptions}

Recall the definition of the Gevrey-$\sigma$ class.

\begin{definition}[Gevrey-$\sigma$ class]\label{d:Gevrey_class}
    Let $\sigma \geq 1$. A function $g\in C^\infty(\Omega)$ is \emph{Gevrey class $G^\sigma(\Omega)$} with constant $M$ if for every compact $K \subset \Omega$ and every $\gamma \in \mathbb{N}^{d}$
    \beq\label{e:Gevrey}
    \sup_{x\in K}|\partial^\gamma g(x)| \leq M^{|\gamma|+1} (\gamma!)^\sigma.
    \eeq
\end{definition}

Note that a Gevrey-$1$ function is analytic; thus the Gevrey class of functions can be seen as an intermediate regularity class between the classes of $C^{\infty}$ functions and analytic functions.

\begin{definition}[Gevrey-$\sigma$ hypersurface]
    A hypersurface $\Gamma$ is Gevrey-$\sigma$ if for every point $x_0\in \Gamma$ there is a neighbourhood $U$ of $x_0$, an orthogonal matrix $O$, 
    and an open set $V\subset\mathbb{R}^d$ a neighbourhood of $0$ and $\psi:U\to V$ a $G^\sigma$ diffeomorphism such that 
    $$
    \psi(U\cap \Gamma)=\big\{(x',0)\,:\,x'\in\mathbb{R}^{d-1},\, (x',0)\in V\big\}. 
    $$
\end{definition}

\begin{definition}[Gevrey-$\sigma$ at a hypersurface $\Gamma$]\label{d:Gev_at_Gamma}
Let $\sigma\geq 1$, $\Gamma \subset \overline{\Omega}$ a smooth, closed hypersurface and $u\in \overline{C^\infty}(\Omega\setminus \Gamma)$. A function $u$ is Gevrey-$\sigma$ at $\Gamma$ if for all connected components, $\Omega_i$ of $\Omega\setminus \Gamma$, there is $M>0$ such that for all $\gamma\in\mathbb{N}^d$
\beqs
\sup_{x\in\Gamma}|\partial^\gamma (u|_{\Omega_i}) (x)|\leq M^{|\gamma|+1}(\gamma!)^{\sigma}.
\eeqs
\end{definition}

\subsection{The finite element method with piecewise polynomials}

 Let $a_k(\cdot,\cdot)$ be the sesquilinear form associated with~\eqref{e:PML1}. We work in the 
Hilbert space
\begin{equation*}
    \cH:= 
\big\{u\in H^1(\Omega)\,:\,u|_{\Gamma_{\textup{soft}}}=0,\, u|_{\partial\Omega_{\tr}}=0\big\}
 \end{equation*}
 with norm 
  \begin{equation}\label{e:1knorm}
     \|u\|^2_{H^1_k(\Omega)}:= k^{-2} \big\|\nabla u\big\|_{L^2(\Omega)}^2 + \|u\|^2_{L^2(\Omega)}.
 \end{equation}

\begin{definition}[Galerkin approximation]
Given  a closed subspace $V\subset \cH$, a \emph{Galerkin approximation of $u\in \cH$} is an element $u_h\in V$ such that 
\begin{equation}
\label{e:galerkinDef}
a_k( u_h,w_h) =a_k(u,w_h)\quad \tfa w_h\in V.
\end{equation}
\end{definition}
Given a triangulation $\mathcal{T}$, a reference element $\widehat{T}$, and element maps $\{\mathcal{F}\}_{T\in \mathcal{T}}$, the space of mapped polynomials of degree $p$ on triangulation $\mathcal{T}$ is defined by 
$$ V^p(\mathcal{T},\{\mathcal{F}_T\}_{T\in \mathcal{T}}):= \big\{ 
v\in H^1(\Omega) \, :\, \text{ for all }  T \in \cT, \, v|_T\circ \mathcal{F}_T \in \mathbb{P}^p(\widehat{T})
\big\}.
$$
We consider a $G^\sigma$ simplicial, affine-conforming triangulation, $\mathcal{T}$ of $\Omega$, where the notions of ``$G^\sigma$'' and ``affine-conforming" are defined below in Definitions \ref{d:Gsigma} and \ref{d:AffineConforming}, respectively. These definitions correspond, roughly speaking, to 
$\mathcal{T}$ consisting of curved elements with  Gevrey-$\sigma$ element maps from a reference simplex, and 
a compatibility condition on neighbouring element maps ensuring that there exist non-trivial globally $H^1$ mapped polynomials (i.e., ensuring that $V^p(\mathcal{T},\{\mathcal{F}_T\}_{T\in \mathcal{T}})$ is not the zero function).
Let  
$$
V_{\mathcal{T}}^p:= V(\mathcal{T},\{\mathcal{F}_T\}_{T\in \mathcal{T}})\cap \mathcal{H},
$$
and let $h(\mathcal{T})$
 to be an upper bound for the diameter of any element of the triangulation. 
Appendix \ref{sec:projection} proves new results about approximation in $V_{\mathcal{T}}^p$, with the bounds explicit in both the polynomial degree and the regularity of the approximated function.
\subsection{Statement of the main result}

\begin{theorem}[$k$-uniform quasi-optimality of both the $hp$-FEM and the $p$-FEM]\label{t:main}
Let $\sigma \geq  1$, $\Gamma_s$, $\partial \Omega_{\tr}$, and $\Gamma_p$ be Gevrey-$\sigma$, $A_\theta$ and $n_\theta$ be Gevrey-$\sigma$ at both $\Gamma_p$ and $\Gamma_s$, $K\subset (0,\infty)$ be such that 
\beq\label{e:polyBound}
\sup_{k\in K}\frac{\log \rho(k)}{\log k}<\infty,
\eeq
 and $\varepsilon, k_0,\Upsilon>0$.
Then there exist $C$, $C_\textup{qo}, h_0 > 0$ such that
for all $k \in K \cap [k_0,\infty)$, 
and all $G^\sigma$ simplicial,  affine-conforming triangulations $\cT$ with constant $\Upsilon$, if 
\beq\label{e:thresholds}
0<h(\cT)<h_0, \qquad \qquad  \frac{h(\cT)k}{p} \leq C, \qquad\qquad p\geq 1 + \varepsilon\log k,
 \eeq
 and 
 \[
 \partial \Omega\cup \Gamma_p \subset \bigcup_{T\in\cT} \partial T
 \]
 then for all $u\in H^1_k(\Omega)$ the Galerkin approximation of $u$ in $V_{\mathcal{T}}^p$ exists and satisfies
 \beq\label{e:qo}
 \|u-u_h\|_{H^1_k(\Omega)} \leq C_\textup{qo} \min_{w_h \in \mathcal{V}^p_{\cT}}\|u-w_h\|_{H^1_k(\Omega)}.
 \eeq
 \end{theorem}

 \begin{remark}
If the triangulation is quasi-uniform 
then the dimension of $V_\mathcal{T}^p$ is proportional to $\big(p/h(\mathcal{T})\big)^d$. Thus Theorem \ref{t:main} shows that there exist piecewise-polynomial spaces with dimension $\sim k^d$ for which the Galerkin approximation is $k$-uniformly quasi-optimal.
 \end{remark}

\begin{remark}[Polynomial boundedness of the solution operator occurs at ``most" frequencies]\label{r:polyBound}
The bound \eqref{e:polyBound} is equivalent to $\rho(k)$ being polynomially bounded in $k$ for $k\in K$. Furthermore, 
given $\delta,k_0>0$, there exists $K\subset [k_0,\infty)$ with the measure of $[k_0,\infty)\setminus K$ less than $\delta$ such that \eqref{e:polyBound} holds for $k\in K$. 
To see this, recall that 
(i) the solution of the PML problem \eqref{e:PML1} exists and is unique for all but finitely-many values of $k$ -- this is a consequence of well-posedness for sufficiently-large $k$ \cite{GLS2} and Fredholm theory (showing well-posedness for all but a discrete set of $k$s),
(ii) by \eqref{e:PMLcontrol}, the PML solution operator inherits the bound on the scattering solution operator for sufficiently-large $k$, and 
(iii) the  scattering solution operator is polynomially bounded for ``most" $k$ (in the sense above) 
by \cite{lafontaine2021most}.
\end{remark}
 
\subsection{Discussion of Theorem \ref{t:main}}\label{s:discussion}

\paragraph{The context and novelty of Theorem~\ref{t:main}}
It has been known since the seminal works of Ihlenburg and Babu\v{s}ka~\cite{ihlenburg1995finite,ihlenburg1997finite}  that, despite the fact that a fixed number of points per wavelength, $hk\sim 1$, is enough to accurately approximate waves with wavenumber $k$, the $h$-version of the FEM requires $hk$ to tend to 0 as $k\to \infty$ to maintain accuracy. This phenomenon is known as \emph{the pollution effect} \cite{babuska2000pollution}. 

The pollution effect becomes less pronounced as $p$ increases. Indeed, 
for general Helmholtz problems the Galerkin solution is $k$-uniformly quasi-optimal if $(hk)^p \rho(k)$ is sufficiently small \cite{ChNi:20} and has controllably-small relative error (for $k$-oscillating data) if $(hk)^{2p}\rho(k)$ is sufficiently small \cite{galkowski2025sharp}.  It is therefore natural to consider the $p$- or $hp$-FEM, with $p$ increasing as $k$ increases.

It was first shown in \cite{melenk2010convergence, melenk2011wavenumber} (see also \cite{spence2023simple}) that, if
\begin{itemize}
    \item[(i)] 
    the solution operator is polynomially bounded in $k$ (which holds for ``most'' $k$ by \cite{lafontaine2021most}; see Remark \ref{r:polyBound}), and 
    \item[(ii)] the domain is analytic, and the PDE coefficients are constant,
\end{itemize}
    then both the $hp$-FEM and $p$-FEM do not suffer from the pollution effect -- i.e., the Galerkin solution is $k$-uniformly quasioptimal -- if 
    \begin{equation}\label{e:threshold}
       h\leq c,\quad \frac{hk}{p}\leq c_\e\quad \text{ and }\quad  p\geq 1+\e \log k.
    \end{equation}
    These results were then generalised to wider classes of Helmholtz problems in \cite{lafontaine2022wavenumber,galkowski2023decompositions,galkowski2024hp, bernkopf2025wavenumber}.

The current state-of-the-art analyses of the $hp$-FEM and $p$-FEM yielding $k$-uniform quasioptimality require regularity assumptions that are substantially weaker than (ii) above. Indeed, 
under \eqref{e:threshold},  $k$-uniform quasioptimality of the $hp$-FEM and $p$-FEM is proved 
for the variable-coefficient Helmholtz Dirichlet problem when \emph{either}
\begin{itemize}
    \item[(a)] $\Omega_-$ is analytic and the coefficients, $A$ and $n$ are smooth in $\Omega$ and analytic in a neighbourhood of $\partial\Omega_-$~\cite{lafontaine2022wavenumber, galkowski2023decompositions, galkowski2024hp}, \emph{or}
    \item[(b)]  $\Omega_-$ is analytic and the coefficients $A$, and $n$ are piecewise analytic with respect to an analytic interface $\Gamma_p$~\cite{bernkopf2025wavenumber}. 
\end{itemize}
In comparison, Theorem~\ref{t:main} requires
\begin{itemize}
    \item[(c)] $\Omega_-$ is Gevrey-$\sigma$ and the coefficients $A$ and $n$ are piecewise smooth with respect to a Gevrey-$\sigma$ interface $\Gamma_p$, and $A$ and $n$ are Gevrey-$\sigma$ at $\partial\Omega_-$ and at $\Gamma_p$ (in the sense that the restriction from either side of $\Gamma_p$ is Gevrey-$\sigma$; see Definition~\ref{d:Gev_at_Gamma}). 
\end{itemize}
That is, Theorem~\ref{t:main} weakens the regularity requirements of~\cite{lafontaine2022wavenumber,galkowski2023decompositions,galkowski2024hp,bernkopf2025wavenumber} in two substantial ways: 1) analyticity requirements are replaced by Gevrey-type requirements and 2) all regularity requirements beyond piecewise smoothness are imposed only \emph{at} interfaces and boundaries rather than in neighbourhoods thereof.

\paragraph{The proof of Theorem~\ref{t:main}: high- and low-frequency splitting}

The analysis of the $hp$-FEM  via splitting the solution into ``well-behaved" and ``poorly-behaved" parts has been known since the work of Babu\v{s}ka and Guo \cite{BabuskaGuo1988RegularityI, BabuskaGuo1989RegularityII, GuoBabuska1986Part1, GuoBabuska1986Part2} on domains with corner singularities, where the poorly-behaved part arises as a result of failure of full elliptic regularity at the corner; this approach was then further developed in the books \cite{Schwab1998, melenk2004hp}.

In the context of $k$-explicit analyses of the Helmholtz equation, this splitting is one into low and high frequencies:~low frequencies should be well approximated by polynomials but may be large in $k$, while high frequencies have less good approximation, but should be well controlled in $k$ because the Helmholtz operator is, in some sense, uniformly elliptic on high frequencies (in the absence of boundaries, this can be understood via the notion of \emph{semiclassical ellipticity} \cite[Theorem E.33]{DyZw:19}).

The key technical ingredient in the analysis of the Helmholtz $hp$-FEM is then the method for splitting the solution $u$ to~\eqref{e:PML1} into high- and low-frequencies:
$$
u=u_{\mathcal{L}}+u_{\mathcal{H}}
$$
so that $u_{\mathcal{L}}$ is well approximated by piecewise polynomials and $u_{\mathcal{H}}$ is controlled in $k$. 

In the approach of 
\cite{melenk2010convergence, melenk2011wavenumber, bernkopf2025wavenumber}
(also used in 
\cite{melenk2021wavenumber, melenk2024wavenumber,melenk2026wavenumber} for the time-harmonic Maxwell equations), suitable frequency cutoffs are applied to the data $f$ in \eqref{e:PML1} -- on $\mathbb{R}^d$ these cutoffs are defined via the Fourier transform, and on subsets of $\mathbb{R}^d$ these are defined via first extending to $\mathbb{R}^d$ and then using the Fourier transform. The function $u_{\mathcal{L}}$ is then defined 
as the Helmholtz solution with right-hand side given by the low-frequencies of the data. Since these 
low frequencies are automatically analytic,
if the domain and coefficients are piecewise analytic, then  $u_{\mathcal L}$ is piecewise analytic
by analytic elliptic regularity.
The proof that $u-u_{\mathcal L}=: u_{\mathcal H}$ is ``well-behaved" with respect to $k$ relies crucially on the fact that, for high-frequency data, the Helmholtz solution is ``close" (in an appropriate sense) to the solution of Helmholtz problem with a sign flipped (so the PDE is coercive, uniformly in $k$); i.e., semiclassical ellipticity on high frequencies.  

In the approach of \cite{lafontaine2022wavenumber,galkowski2023decompositions, galkowski2024hp, galkowski2025numerical}, frequency cut-offs are applied direct to the Helmholtz solution $u$ in \eqref{e:PML1}. 
There are two basic ideas that have been used to achieve frequency splittings in this way (see also the discussion in~\cite[Section 7.1]{galkowski2025numerical}):
\begin{enumerate}
    \item use the Fourier transform  
    \cite{lafontaine2022wavenumber}, or 
    \item apply a function of a self-adjoint operator, $\mathcal{P}$,~\cite{galkowski2023decompositions, galkowski2024hp,galkowski2025numerical} that agrees with the Helmholtz operator $P_k$ near the scatterer.
\end{enumerate}
By using the Fourier transform-- specifically defining $u_{\mathcal{L}}:=1_{[-2k,2k]}(|D|)u$-- one can, essentially always, split $u$ into a low-frequency part that is well approximated and a high-frequency part. The difficulty is then to show that the high frequency part is controlled by the data $f$ without growth in $k$. One of the surprising features of~\cite{lafontaine2022wavenumber} was that, in contrast to 
the approach in \cite{melenk2010convergence, melenk2011wavenumber, bernkopf2025wavenumber},
it required no analyticity assumptions, instead showing that, despite the fact $A$ and $n$ may not be better than smooth, the high frequencies (as measured by the Fourier transform) are nevertheless well controlled.

The papers~\cite{galkowski2023decompositions,galkowski2024hp,galkowski2025numerical} then used this insight to remove analyticity assumptions strictly away from boundaries, but still require analyticity in neighbourhoods thereof. This analyticity is used to guarantee that the eigenfunctions of $\mathcal{P}$ are approximated well. Up to now, however, there was no conceptual explanation for why substantial additional regularity assumptions (beyond smoothness) might be needed near boundaries, but not in the interior. 

In this paper, we provide such an explanation using a new method for frequency splitting. Specifically, we show that it is enough to construct an auxiliary self-adjoint, elliptic, second-order differential operator $\widetilde{\mathcal{P}}$, such that, for any $N$, the domain of $\widetilde{\mathcal{P}}^N$ agrees with that of $ P_k^N$ (away from the PML boundary  $\partial\Omega_{\tr}$) and such that the eigenfunctions of $\widetilde{\mathcal{P}}$ are approximated well by piecewise polynomials. This new approach (i) explains why, in the absence of boundaries and interfaces, one can use Fourier multipliers that are based on the operator $-\Delta$ and (ii) enables us to impose  regularity assumptions (beyond smoothness) only \emph{at} boundaries and interfaces. Specifically, we construct $\widetilde{\mathcal{P}}$ by finding $\widetilde{a}$ and $\tilde{n}$ that, together with all derivatives, agree with $A$ and $n$ at $\partial\Omega_+$ and $\Gamma_p$, are Gevrey-$\sigma$ in a neighbourhood of these interfaces, and are piecewise smooth. The operator defined by sesquilinear form
$$
\langle \widetilde{\mathcal{P}}u,v\rangle := \langle \widetilde{a}k^{-1}\nabla u,k^{-1}\nabla v\rangle -\langle \tilde{n}u,v\rangle 
$$
acting on the same space as the sesquilinear form associated to $P_k$ then has the required properties. 

\paragraph{Potential extensions to Theorem \ref{t:main}.}
The proof of Theorem~\ref{t:main} suggests that one can weaken the regularity requirements for the triangulation, $\Gamma_p$, and $\partial\Omega_+$ as well as those for $A$ and $n$ at $\Gamma_p$ and $\partial\Omega_+$ to the following Denjoy--Carleman class of functions
\begin{equation}
\label{e:denjoyCarleman}
\sup_{x\in K}|\partial^\gamma g(x)|\leq M^{|\gamma|+1}e^{\sigma |\gamma|^2}
\end{equation}
(compare to \eqref{e:Gevrey}); this is discussed further in Remark \ref{r:epsquared} below. 

The only missing component to proving Theorem~\ref{t:main} under this assumption is the relevant elliptic regularity theory. Although we are not aware of such results for the classes in~\eqref{e:denjoyCarleman}, we do not believe there are any fundamental difficulties in proving such elliptic-regularity results. We have chosen not to pursue this here since Gevrey classes of functions are much more standard in the numerical-analysis literature (see, e.g., \cite{FeSc:20}).

A natural question is whether one can entirely remove assumptions -- on either $\partial\Omega_-$ and $\Gamma_p$ or the coefficients -- at the boundary stronger than just smoothness. 

We do not expect 
removing assumptions on
$\partial\Omega_-$ or $\Gamma_p$ 
for general (i.e., not constructed using knowledge of both the coefficients and the geometry) $C^\infty$ triangulations 
to be possible for the following reason:~in a conforming triangulation, there will always be a simplex, $T$, one of whose faces is contained in the boundary. Because of this, the element map $\mathcal{F}_T$ cannot be more regular than the boundary and hence, for a typical smooth $u$, $u\circ \mathcal{F}_T$ cannot be more regular than the boundary itself and therefore, the quality of the approximation by a polynomial will not be well controlled as the degree $p$ tends to infinity. Specifically, if the $\ell^{\text{th}}$ derivative of the boundary is of size $M_\ell$, one expect an approximation no better than
$$
\bigg(\frac{Chk}{p}\bigg)^p\max(1,k^{-p}M_p)
$$
and hence, if $M_p$ grows arbitrarily fast and $Chk/p$ is bounded below, it will not be possible to make this error superpolynomially small in $k$. (The exception to this would be to construct the triangulation and element maps using knowledge of $u$ so that $u\circ \mathcal{F}_T$ has better regularity than either $u$ or $\mathcal{F}_T$.)

The question of whether the regularity of $A$ and $n$ on the boundaries can be relaxed is somewhat less clear, but an indication that it may not be possible is the following:~all splittings of $u$ into high and low frequencies to date require that $u_{\mathcal{L}}$ be in the domain of any fixed power of the operator $P$. 
This implies that $u_{\mathcal{L}}$ cannot be more regular than the coefficients (at the boundaries); thus, 
when coefficients do not have additional regularity at the boundary $u_{\mathcal{L}}$ is poorly approximated as $p\to \infty$. 

\subsection{Organization of the paper}
Section~\ref{sec:pseudo_reg}
describes the operator $\widetilde{\cP}$ used to define the low-frequency cut-offs and proves related pseudolocality results. 
Section \ref{s:abstract} proves a bound on the adjoint-approximability constant modulo controlling the low-frequencies. 
Section \ref{s:construction} constructs $\widetilde{\cP}$ under Gevrey-type regularity assumptions. 
Section \ref{s:proof} proves Theorem \ref{t:main} using the results of Sections \ref{s:abstract} and  \ref{s:construction}. 
Appendix~\ref{sec:Schatz} recaps the Schatz argument for proving quasi-optimality.
Appendix~\ref{sec:projection} 
proves $p$-explicit polynomial approximation results tailored to Gevrey-$\sigma$ functions, and  Appendix~\ref{sec:elliptic_reg_theory} 
proves elliptic regularity results under Gevrey-$\sigma$ regularity.

\section{
Low-frequency cutoffs defined by functional calculus and related 
pseudolocality results}\label{sec:pseudo_reg}

\subsection{Notation}

We first define the Helmholtz PML problem under less regularity than in Theorem \ref{t:main}. 

\begin{definition}[$H^\ell$ Helmholtz PML problem]
\label{d:helmholtzProblem}
The PML problem~\eqref{e:PML1} is a \emph{$H^{\ell}$ Helmholtz PML problem} if $\Omega_{\textup{tr}}$ and $\Omega_-$ and $C^{\ell-1,1}$, and there is $\Gamma_{ p}\Subset \Omega_+$ a $C^{\ell-1,1}$ closed, embedded hypersurface such that $A_\theta,b_\theta,n_\theta\in \overline{C^{\ell-2,1}}(\Omega_+\setminus \Gamma_{ p})$.  We then call the operator $P_k$ in~\eqref{e:PML1} an $H^\ell$ Helmholtz PML operator.
The problem \eqref{e:edp} is an $H^\infty$ Helmholtz problem if it is an 
$H^\ell$ Helmholtz problem for all $\ell$ and we call $P_k$ an $H^\infty$ Helmholtz PML operator.
\end{definition}

In a similar way to how the space $\overline{C^\infty}(V)$ was defined, here, a function $u\in \overline{C^\ell}(V)$ if for all connected components $V_j$ of $V$ there exists $U_j \in C^{\ell}(\Rea^d)$ such that $U_j|_{V_j}= u|_{V_j}$.

\begin{definition}[Semiclassical differential operator]\label{d:semicl_diff_op}
A \emph{semiclassical differential operator} of order $m$ on $\Omega\setminus \Gamma_p$ is an operator of the form
\begin{equation}
\label{e:semiclassicalDifferential}
 \sum_{|\alpha| \leq m} a_\alpha(x,k)\,(k^{-1} D_x)^\alpha,
\end{equation}
where $D_x = \frac{1}{i}\partial_x$ and $a_\alpha\in S(\Omega)$ where  
$$
S(\Omega\setminus \Gamma_p):=\Big\{ a\in C^0( (0,\infty);\overline{C^\infty}(\Omega\setminus \Gamma_p))\,:\,  
\forall \, k_0>0,\gamma \in \mathbb{N}^d
\sup_{k\in[k_0,\infty)}
\|\partial_x^\gamma a(x,k)\|_{L^\infty(\Omega\setminus \Gamma_p)}<\infty\Big\}.
$$
We say that a semiclassical differential operator of order $m$ is elliptic if for all $k_0>0$ there is $c>0$ such that
$$
\inf_{k\in[k_0,\infty),\,x\in\Omega\setminus \Gamma_p}\Big|\sum_{|\alpha|=m}a_{\alpha}(x,k)\xi^\alpha\Big|\geq c |\xi|^m,\qquad \xi \in\mathbb{R}^d,
$$
where $a_{\alpha}$ are the coefficients appearing in~\eqref{e:semiclassicalDifferential}. 

We say that a semiclassical differential operator of order $m$ has Gevrey-$\sigma$ coefficients if for all $k_0>0$ there is $C>0$ such that for all $\alpha\in\mathbb{N}^d$ with $|\alpha|\leq m$ and all  $\gamma \in \mathbb{N}^d$, 
$$
\sup_{k\in [k_0,\infty)}\|\partial_x^\gamma a_{\alpha}(x,k)\|_{L^\infty(\Omega\setminus \Gamma_p)}\leq C^{|\gamma|+1}(\gamma !)^{\sigma}
$$
where $a_{\alpha}$ are the coefficients appearing in~\eqref{e:semiclassicalDifferential}. 
\end{definition}

The following notation is used to state the pseudolocality results. 
\begin{definition}
Let $X,Y$ be Banach spaces, $f:(0,\infty)\to (0,\infty)$, and $B_k:X\to Y$ a $k$-dependent operator. 
$$
B_k=O(f(k))_{X\to Y}
$$
if for all $k_0>0$ there is $C>0$ such that 
$$
\|B_k\|_{X\to Y}\leq Cf(k),\qquad k>k_0.
$$
Furthermore, 
$$
B_k=O(k^{-\infty})_{X\to Y}
$$
if for all $k_0>0$ and $N>0$ there is $C>0$ such that 
$$
B_k=O(k^{-N})_{X\to Y}.
$$
\end{definition}

\subsection{The operator $S_k= \psi(\widetilde{\cP})$ and pseudolocality results}

We begin with an abstract result about two operators satisfying G\aa rding inequalities.
\begin{lemma}
\label{l:coerce}
Let $\mathcal{H}\subset H^1_k(\Omega)$ be a Hilbert space with $$
\|u\|_{\mathcal{H}}=\|u\|_{H^1_k(\Omega)},\qquad for u\in \mathcal{H},
$$
$c_{\rm G}, C_{\rm G}>0$, and $P, \tilde{\mathcal{P}}:\mathcal{H}\to \mathcal{H}^*$ satisfy 
\begin{gather*}
\langle \widetilde{\cP}u,u\rangle\geq c_{\rm G}\|u\|_{H^1_k}^2-C_{\rm G}\|u\|_{L^2}^2 \quad\tfa u \in \mathcal{H},\\
\Re\langle Pu,u\rangle\geq c_{\rm G}\|u\|_{H^1_k}^2-C_{\rm G}\|u\|_{L^2}^2 \quad\tfa u \in \mathcal{H},
\end{gather*}
and $\widetilde{\mathcal{P}}$ selfadjoint.  Then there is $\psi\in C_c^\infty(\mathbb{R})$ such that, with $S_k:=\psi(\widetilde{\mathcal{P}})$, 
$$
\Re \big\langle \big(P_k+S_k\big)v,v\big\rangle \geq c\|v\|_{H_k^1}^2\quad\textrm{ for all }  v\in \mathcal{H}.
$$
\end{lemma}

\begin{proof}
There is $\varepsilon>0$ and $c_1, c_2$ such that 
\[
\Re \langle (P-\varepsilon \widetilde{\mathcal{P}}) v, v\rangle \geq c_1 \|v\|^2_{H_k^1}- c_2 \|v\|^2_{L^2}.
\]
Next, let $M$ be large enough such that $\varepsilon M> c_2$. Since $\widetilde{\cP}$ is self-adjoint and  satisfies a G\aa rding inequality, we can define functions of this operator $\widetilde{\cP}$; see, e.g., 
\cite[Theorem 2.37]{mclean2000strongly}, \cite[Ch. 5]{schmudgen2012unbounded}. In particular, there exist $\psi_0 \in  C_c^\infty(\R)$ such that $\widetilde{\cP} + \psi_0(\widetilde{\mathcal{P}}) \geq M$; see, e.g., \cite[Lemma 2.1]{galkowski2025sharp} and \cite[\S 5.3]{AGS2}. 

Thus a simple computation gives us the desired result for $\psi := \varepsilon \psi_0$
\[
\Re \big\langle \big(P+\varepsilon \psi_0(\widetilde{\mathcal{P}})\big) v, v\big\rangle =\Re \big\langle \big(P-\varepsilon \widetilde{\mathcal{P}} + \varepsilon\big(\widetilde{\mathcal{P}}+ \psi_0(\widetilde{\mathcal{P}})\big)\big) v, v\big\rangle \geq c_1 \|v\|^2_{H_k^1}.
\]
\end{proof}

We require two sets of pseudolocality results -- spatial and in frequency.
The spatial ones needed here are proved in 
\cite[Theorems 5.27 and 6.2]{AGS2} and can be applied directly to the semiclassical operator $\widetilde{\cP}$. We recall them here without proof.

\begin{lemma}\mythmname{Spatial Pseudolocality \cite[Theorems 5.27 and 6.2]{AGS2}}
\label{l:spatialPseudolocalFinal}
Let $P_k$ be an $H^\infty$ Helmholtz PML operator and $\widetilde{\cP}$ be a self adjoint realisation of an elliptic, semiclassical differential operator of order 2.
Suppose that $\psi\in \mathcal{S}(\mathbb{R})$, $\chi_1,\chi_2 \in \overline{C^\infty}(\Omega\setminus\Gamma_p)$  with $\supp \chi_1\cap \supp \chi_2=\emptyset$. Then for all $N>0$, 
\begin{align*}
\chi_1\psi(\widetilde{\mathcal{P}})\chi_2&=O(k^{-\infty})_{(H_k^N(\Omega\setminus \Gamma_{ p}))^*\to H_k^N(\Omega\setminus \Gamma_{ p})}.
\end{align*}
Moreover, for $\psi\in C_c^\infty(\mathbb{R})$ constructed in Lemma~\ref{l:coerce} and $S_k:=\psi(\widetilde{\mathcal{P}})$,
\begin{align*}
\chi_1
(P_k+S_k)^{-1}
\chi_2&=O(k^{-\infty})
_{
 (H_k^N(\Omega\setminus \Gamma_{ p}))^*
\to H_k^N(\Omega\setminus \Gamma_{ p})},\\ 
\chi_1
(P_k^*+S_k)^{-1}
\chi_2&=O(k^{-\infty})_{( H_k^N(\Omega\setminus \Gamma_{ p}))^*\to H_k^N(\Omega\setminus \Gamma_{p})}.
\end{align*}
\end{lemma}

The following frequency pseudolocality result 
is a consequence of \cite[Theorem 5.37]{AGS2} after we check certain commutator relations between $P_k$ and $\widetilde{\cP}$
(see Lemma \ref{l:commuteP} below). When $\widetilde{\cP} = \Re P_k$, the analogous result was proved 
in \cite[Theorem 5.37 and Lemmas 6.7 and 6.9]{AGS2}.
To simplify notations below, it is useful to define the $N$-commutator:
\begin{equation}
\ad^0_g f = f, \qquad
\ad^N_g f = [g,\ad^{N-1}_gf];  
\end{equation}
with $N \geq 1$.

\begin{lemma}
\mythmname{Pseudolocality in Frequency}
\label{l:frequencyPseudolocalFinal}
Let $k_0>0$ and let $P_k$ be an $H^\infty$ Helmholtz PML operator, $\cP:=\Re P_k$, and $\widetilde{\cP}$ be a self adjoint realisation of an elliptic, semiclassical differential operator of order 2 such that 
for all $N \in \mathbb{N}$,
$\mathcal{D}(\widetilde{\cP}^N) =\mathcal{D}( \cP^N) = D_k^{2N}$ 
where
$$
D_k^{2N}:= \left\{ u \in H^{2N}(\Omega\setminus \Gamma_p) \,: \,\begin{gathered} \mathcal{B} \cP^j u|_{\Gamma} = 0,\, (\cP^ju)|_{\partial\Omega_{\tr}}=0,\\ \cP^ju\in H^1(\Omega),\, 
\partial_{\nu_{A,+}}\cP^ju|_{\Gamma_p}=\partial_{\nu_{A,-}}\cP^ju|_{\Gamma_p}\end{gathered}  \tfor j = 0, \ldots, N-1,\,  \right\}
$$
and for all $N_0>0$ there is $C>0$ such that for all $k\geq k_0$ and $u\in\mathcal{D}(\mathcal{P}^{N_0})$,  
\beq \label{e:graphNorm}
\begin{gathered}
C^{-1}\sum_{m=0}^{N_0}\|\widetilde{\mathcal{P}}^mu\|_{L^2}\leq \sum_{m=0}^{N_0}\|\mathcal{P}^mu\|_{L^2}\leq 
C\sum_{m=0}^{N_0}\|\widetilde{\mathcal{P}}^mu\|_{L^2}.
    \end{gathered}
    \end{equation}
For all $M>0$, $N>0$, $\psi_1,\psi_2\in C^\infty(\mathbb{R})$, $\chi \in \overline{C^\infty}(\Omega)$ such that 
$$
\|\langle \Jbullet\rangle^{-M}\psi_i\|_{L^\infty}<\infty,
$$
either $\psi_1\in \mathcal{S}(\mathbb{R})$ or $\psi_2\in \mathcal{S}(\mathbb{R})$,  $\supp\psi_1\cap \supp\psi_2=\emptyset$, $\supp \nabla \chi\cap (\partial \Omega_{\tr}\cup \Gamma_s\cup \Gamma_{ p}) =\emptyset$, and $\supp \chi\cap \partial \Omega_{\tr}=\emptyset$, we have 
\begin{align}
\psi_1(\widetilde{\mathcal{P}})\chi\psi_2(\widetilde{\mathcal{P}})&=O(k^{-\infty})_{\mathcal{D}(\mathcal{P}^{-N})\to \mathcal{D}(\mathcal{P}^{N})},\label{ineq:freq_pseudo}\\
\psi_1(\widetilde{\mathcal{P}})\chi 
(P_k+S_k)^{-1}
\chi\psi_2(\widetilde{\mathcal{P}})&=O(k^{-\infty})_{\mathcal{D}(\mathcal{P}^{-N})\to \mathcal{D}(\mathcal{P}^{N})},\label{ineq:freq_pseudo_2}\\
\psi_1(\widetilde{\mathcal{P}})\chi 
(P_k^*+S_k)^{-1}
\chi\psi_2(\widetilde{\mathcal{P}})&=O(k^{-\infty})_{\mathcal{D}(\mathcal{P}^{-N})\to \mathcal{D}(\mathcal{P}^{N})},\label{ineq:freq_pseudo_3}
\end{align}
where we write $\mathcal{D}(\cP^{-N}):=[\mathcal{D}(\cP^N)]^*.$
\end{lemma}
\begin{remark}
    Note that since $\mathcal{P}\geq -C$, $\mathcal{D}(\cP^{-N})=(\cP+C+1)^{N}L^2$. 
\end{remark}

We start by giving a general result on repeated commutators of semiclassical operators with matching domains.
\begin{lemma}\label{lem:adNAB_domains}
Let $A,\,B$ be two semiclassical second-order differential operator and suppose that for all $k_0>0$ and $n\in\mathbb{N}$ there is $C_0>0$ such that for all $k>k_0$,
$$
\|A\|_{\mathcal{D}(\mathcal{P}^{n})\to \mathcal{D}(\cP^{n-1})}+\|B\|_{\mathcal{D}(\mathcal{P}^{n})\to \mathcal{D}(\cP^{n-1})}\leq C_0.
$$
Then, for all $N,n \in \mathbb{N}$, $k_0>0$ there exist $C>0$ such that for $k>k_0$,
\begin{equation}\label{N_commutation_map}
\ad^N_A B: \mathcal{D}(\cP^n) \to  \mathcal{D}(\cP^{n-N/2-1})
\end{equation}
and \begin{equation}\label{k_depend}
\|\ad^N_A B\|_{\mathcal{D}(\cP^n) \to  \mathcal{D}(\cP^{n-N/2-1})} \leq C k^{-N}.
\end{equation}
\end{lemma}
\begin{proof}
By interpolation, it is sufficient to assume that $n$ is an even integer.
The proof relies on a density argument, specifically that for all $m, n, N \in \mathbb{N}$ such $m \geq n+N$
\[\overline{\mathcal{D}(\mathcal{P}^{m/2})}^{H^n_k} = \mathcal{D}(\mathcal{P}^{n/2}), \qquad \overline{\mathcal{D}(\mathcal{P}^{m/2-N-1})\cap H_k^{m-N-2}}^{H_k^{n-N-2}} = \mathcal{D}(\mathcal{P}^{n/2-N/2-1}). \]
It therefore suffices to show the intermediate result
\begin{equation}\label{N_commutation_map_inter}
\ad^N_A B: \mathcal{D}(\mathcal{P}^{m/2}) \to \mathcal{D}(\mathcal{P}^{m/2-N-1})  \cap H^{m-N-2}_k
\end{equation}
\begin{equation}\label{k_depend_inter}
\|\ad^N_A B\|_{\mathcal{D}(\mathcal{P}^{m/2}) \to \mathcal{D}(\mathcal{P}^{m/2-N-1})\cap H_k^{m-N-2}} \leq C k^{-N},
\end{equation}
for all $m\geq n+N$.
    The mapping \eqref{N_commutation_map_inter} follows from the fact that $\ad_A^N B$ is an alternating sum of applications of the operator $A$, $N$ times, and $B$, once. Since, for all $j \in \mathbb{N}$, 
    $$
    \|A^j\|_{\mathcal{D}(\mathcal{P}^{n})\to \mathcal{D}(\mathcal{P}^{n-j})}+\|B^j\|_{\mathcal{D}(\mathcal{P}^{n})\to \mathcal{D}(\mathcal{P}^{n-j})}\leq C
    $$
    $\ad^N_A B$ maps then $\mathcal{D}(\mathcal{P}^{m/2})$ to $\mathcal{D}(\mathcal{P}^{m/2-N-1})$. By the fact that 
    \[
    \mathrm{ad}_A^N B = k^{-N} C_{N+2}
    \]
    where $C_{N+2}$ is a $N+2$-th order semiclassical differential operator in the sense of Definition~\ref{d:semicl_diff_op}, then $\ad_A^N B$ also maps $\mathcal{D}(\mathcal{P}^{m/2})$ to $H_k^{m-N-2}$.
Thus, estimate \eqref{k_depend_inter} follows. The desired \eqref{N_commutation_map} and \eqref{k_depend} are then obtained by density, using the fact that $\ad^N_A B$ is a bounded (i.e., continuous) mapping from $H_k^m$ to $H_k^{m-N-2}$.
\end{proof}

\begin{lemma}
\label{l:commuteP}
Let $\widetilde{\cP}$ be as in Lemma~\ref{l:frequencyPseudolocalFinal}. Suppose that $\varphi\in C^\infty(\overline{\Omega})$, $\supp \nabla\varphi \cap (\partial\Omega\cup\Gamma_p)=\emptyset$, and $\supp \varphi\cap \partial \Omega_{\tr}=\emptyset$. Then,
\begin{equation}\label{eq:HRevil1}
\|\ad_{\varphi\widetilde{\cP}\varphi }^NP_k\|_{\mathcal{D}(\mathcal{P}^{n/2})\to \mathcal{D}(\mathcal{P}^{n/2-N/2-1})}\leq Ck^{-N}, 
\quad 
\|\ad_{\varphi\widetilde{\cP}\varphi }^N P_k^*\|_{\mathcal{D}(\mathcal{P}^{n/2})\to \mathcal{D}(\mathcal{P}^{n/2-N/2-1})}\leq Ck^{-N}.
\end{equation}
\end{lemma}
\begin{proof}[Proof of Lemma \ref{l:commuteP}]
We prove \eqref{eq:HRevil1}. We write 
\beqs
\ad_{\varphi \widetilde{\cP}_k\varphi}^N P_k= \psi_0 \ad_{\varphi \widetilde{\cP}_k\varphi}^N P_k+(1-\psi_0) \ad_{\varphi \widetilde{\cP}_k\varphi}^N P_k, 
\eeqs
where $\psi_0$ is supported in a region close to $\partial\Omega$ where $\varphi$ is constant.
More precisely, let $\psi_i\in C^\infty(\overline{\Omega})$, $i=-1,0,1$ with $\supp (1-\psi_i)\cap \partial\Omega_-=\emptyset$, $\supp\psi_i\cap \supp\nabla \varphi=\emptyset$, $\supp (1-\psi_i)\cap \supp \psi_{i-1}=\emptyset$, and $\psi_1 P_k=\psi_1 P_k^*$
 (note that such functions exist since $\supp \nabla \varphi\cap \partial\Omega_-=\emptyset$).

By locality of $\cP_k$, $P_k$, and $P_k^*$ and the fact that $\cP_k=P_k=P_k^*$ on $\supp \psi_1\supset \supp\psi_0$,
$$
\psi_0 \ad_{\varphi \widetilde{\cP}_k\varphi}^N P_k =\psi_0\ad_{\widetilde{\cP}_k}^N P_k, 
$$
and thus \[
\|\psi_0\ad_{\varphi \widetilde{\cP}_k\varphi}^N P_k\|_{\mathcal{D}(\mathcal{P}^{n/2})\to \mathcal{D}(\mathcal{P}^{n/2-N/2-1})}\leq Ck^{-N}
\]
where we used Lemma~\ref{lem:adNAB_domains} with $A =\widetilde{\cP}$ and $B=P_k$.

Now, let $\tilde{\varphi}\in C^\infty(\overline{\Omega})$ with $\supp \tilde{\varphi}\cap \partial \Omega_{\tr}=\emptyset$ and $\supp \varphi\cap \supp (1-\tilde{\varphi})=\emptyset$. Then
$$
(1-\psi_0)\ad^N_{\varphi\widetilde{\cP}_k\varphi}P_k=(1-\psi_0)\ad^N_{\varphi(1-\psi_{-1})\widetilde{\cP}_k\varphi(1-\psi_{-1})}\big[(1-\psi_{-1})\tilde{\varphi}P_k(1-\psi_{-1})\tilde{\varphi}\big].
$$
Since $\supp \tilde{\varphi}\cap \supp (1-\psi_{-1})\cap \partial\Omega=\emptyset$, 
integration by parts (with all the boundary terms vanishing) shows that $\varphi(1-\psi_{-1})\widetilde{\cP}_k\varphi(1-\psi_{-1})$ and $(1-\psi_{-1})\tilde{\varphi}P_k(1-\psi_{-1})\tilde{\varphi}$ coincide with differential operators on $\overline{C^\infty}(\Omega\setminus\Gamma_p)$. The result 
$$
\|\ad^N_{\varphi\widetilde{\cP}_k\varphi}P_k\|_{\mathcal{D}(\mathcal{P}^{n/2})\to\mathcal{D}(\mathcal{P}^{n/2-N/2-1})}\leq Ck^{-N} 
$$
then follows by direct differentiation (using the product rule) and then density of $\overline{C^\infty}(\Omega\setminus\Gamma_p)$ in $\mathcal{\cP}^n$.
The proof of the analogous bound for $P_k^*$ is identical.
\end{proof}

\begin{proof}[Proof of Lemma~\ref{l:frequencyPseudolocalFinal}]
        Estimate \eqref{ineq:freq_pseudo} follows directly from \cite[Lemma 7.5]{galkowski2025numerical}.
To prove \eqref{ineq:freq_pseudo_2}, we need to prove certain commutator estimates. Let $n,N \in \mathbb{N}$ . The result \cite[Theorem 5.37]{AGS2} combined with \cite[Theorem 5.31]{AGS2} gives the result of the lemma provided that there is $C>0$ such that for all $k\geq k_0$ 
\begin{equation}\label{eq:desired_commutator_estimate}
\|\ad^N_{\widetilde{\cP}}\chi (P_k+S_k)\chi\|_{D_k^n \to D_k^{n-N-2}} \leq C k^{-N}.
\end{equation}
 Let $\tilde{\chi}\in C^\infty(\overline{\Omega})$ with $\supp \nabla\tilde{\chi}\cap \partial\Omega=\emptyset$ and $\supp (1-\tilde{\chi})\cap \supp \chi=\emptyset$. Then, by spatial locality of $\widetilde\cP$,
\beq\label{e:mrcroc2}
\ad_{\widetilde{\cP}}^N\chi (P_k+S_k)\chi= \ad_{\tilde{\chi}\widetilde{\cP}\tilde{\chi}}^N\chi (P_k+S_k)\chi.
\eeq
We now use \cite[Lemma 6.7]{AGS2} which states that the spatial cutoff function $\chi$ is a boundary compatible operator in the sense of \cite[Definition 5.35]{AGS2}, i.e., 
\begin{equation}
   \|\ad_{\tilde{\chi}\widetilde{\cP}\tilde{\chi}}^N\chi \|_{D_k^n \to D_k^{n-N}} \leq C k^{-N}.
\end{equation}
Therefore, by repeated use of the identity $\ad_{A} BC = (\ad_A B)C + B(\ad_A C)$, to prove \eqref{eq:desired_commutator_estimate} it is enough to show that 
\[
\|\ad^N_{ \tilde{\chi} \widetilde{\cP}\tilde{\chi}} (P_k+S_k)\|_{D_k^n \to D_k^{n-N-2}} \leq C k^{-N}.
\]
Recalling that $S_k = \psi(\widetilde{\cP})$, we know by \cite[Proposition 5.33]{AGS2} that
\[
\|\ad^N_{ \tilde{\chi} \widetilde{\cP}\tilde{\chi}} S_k\|_{D_k^n \to D_k^{n-N-2}} \leq C k^{-N}.
\]
Finally, by Lemma~\ref{l:commuteP} \[
\|\ad^N_{ \tilde{\chi} \widetilde{\cP}\tilde{\chi}} (P_k)\|_{D_k^n \to D_k^{n-N-2}} \leq C k^{-N}.
\]
which finishes the proof of \eqref{ineq:freq_pseudo_2}. The proof of \eqref{ineq:freq_pseudo_3} follows along similar lines and is omitted.
\end{proof}

\section{Bound on the high frequencies of the solution operator}\label{s:HF}

We now use the pseudolocality results of the previous subsection (i.e., Lemmas ~\ref{l:spatialPseudolocalFinal} and~\ref{l:frequencyPseudolocalFinal}) to 
show that 
the (adjoint) solution operator behaves well on high frequencies as measured by $(1-\psi(\widetilde{\mathcal{P}}))$, with $\widetilde{\cP}$ as in Lemma \ref{l:frequencyPseudolocalFinal}.

\begin{lemma}[High frequency elliptic-type estimate]\label{l:trainLate1}
Let $P_k$ be an $H^\infty$ Helmholtz PML operator (in the sense of Definition \ref{d:helmholtzProblem}) and let $\widetilde{\cP}$ satisfy the assumptions of Lemma \ref{l:frequencyPseudolocalFinal}.
Given $\psi_0 \in C_c^\infty$ constructed in Lemma~\ref{l:coerce}, 
let $\psi\in C_c^\infty(\mathbb{R})$ be such that $\supp (1-\psi)\cap \supp \psi_0=\emptyset$ and $\chi \in \overline{C^\infty}(\Omega)$ such that $\supp \chi\cap \partial \Omega_{\tr}=\emptyset$, $\supp \nabla \chi \cap (\Gamma_{ p}\cup \Gamma_s
)=\emptyset$. Then, for all $N>0$ $k_0>0$ there is $C>0$ such that for all $k>k_0$
\begin{equation}
\label{e:highElliptic1}
\|(1-\psi(\widetilde{\mathcal{P}}))\chi\mathcal{R}^*\|_{L^2\to H_k^2(\Omega\setminus \Gamma_{ p})\cap H_k^1(\Omega)}\leq C\big(1+k^{-N}\|\mathcal{R}^*\|_{L^2\to L^2}\big).
\end{equation}
\end{lemma}
\begin{proof}
The proof follows from emulating the steps in \cite[Lemma 7.7]{galkowski2025numerical}.  
Recalling that we denote by $\mathcal{R}^*$ and $\widetilde{\mathcal{R}}_{S_k}^*$ the right, resp., left inverse of $P^*_\theta$, resp., $P_k^*+S_k$, we find that
$$
\mathcal{R}^*=\widetilde{\mathcal{R}}_{S_k}^*+\widetilde{\mathcal{R}}_{S_k}^*\psi_0(\widetilde{\mathcal{P}})\mathcal{R}^*,
$$
so that 
\begin{equation}\label{e:slough0}
(1-\psi(\widetilde{\mathcal{P}}))\chi\mathcal{R}^*= (1-\psi(\widetilde{\mathcal{P}}))\chi \mathcal{R}_{S_k}^*+(1-\psi(\widetilde{\mathcal{P}}))\chi\mathcal{R}_{S_k}^*\psi_0(\widetilde{\mathcal{P}})\mathcal{R}^*.
\end{equation}
Let $\tilde{\chi}\in \overline{C^\infty}(\Omega)$ with $\supp \nabla \tilde{\chi}\cap (\partial \Omega_{\tr} \cup\Gamma_{p}\cup \Gamma_s)=\emptyset$, $\supp (1-\tilde{\chi})\cap \supp \chi=\emptyset$, and $\supp \tilde{\chi}\cap \partial \Omega_{\tr}=\emptyset$. Then, by Lemma \ref{l:spatialPseudolocalFinal}, 
\begin{align}\nonumber
&(1-\psi(\widetilde{\mathcal{P}}))\chi\mathcal{R}_{S_k}^*\psi_0(\widetilde{\mathcal{P}})\mathcal{R}^*\\
&
=(1-\psi(\widetilde{\mathcal{P}}))\chi\mathcal{R}_{S_k}^*\tilde{\chi}\psi_0(\widetilde{\mathcal{P}})\mathcal{R}^* +(1-\psi(\widetilde{\mathcal{P}}))O(k^{-\infty})_{L^2\to H_k^N(\Omega\setminus \Gamma_{ p})}
\psi_0(\widetilde{\mathcal{P}})\mathcal{R}^*.
\label{e:slough1}
\end{align}
Now let $\psi_1\in C_c^\infty(\mathbb{R})$ be such that $\supp \psi_0\cap \supp(1-\psi_1)=\emptyset$ and $\supp \psi_1\cap \supp(1-\psi)=\emptyset$. 
By Lemma \ref{l:frequencyPseudolocalFinal} and then Lemma \ref{l:spatialPseudolocalFinal},
\begin{align}\nonumber
(1-\psi(\widetilde{\mathcal{P}}))\chi\mathcal{R}_{S_k}^*\tilde{\chi}\psi_0(\widetilde{\mathcal{P}})\mathcal{R}^*
&=
(1-\psi(\widetilde{\mathcal{P}}))\chi(1-\psi_1(\widetilde{\mathcal{P}}))\mathcal{R}_{S_k}^*\tilde{\chi}\psi_0(\widetilde{\mathcal{P}})\mathcal{R}^* \\ \nonumber
&\qquad
+O(k^{-\infty})_{\mathcal{D}(\mathcal{P}^{-N})\to \mathcal{D}(\mathcal{P}^N)}
\mathcal{R}_{S_k}^*\tilde{\chi}\psi_0(\widetilde{\mathcal{P}})\mathcal{R}^*\\ \nonumber
&=(1-\psi(\widetilde{\mathcal{P}}))\chi(1-\psi_1(\widetilde{\mathcal{P}}))\tilde{\chi}\mathcal{R}_{S_k}^*\tilde{\chi}\psi_0(\widetilde{\mathcal{P}})\mathcal{R}^* \\
&\,\quad
+O(k^{-\infty})_{L^2\to H_k^N(\Omega\setminus \Gamma_{ p})}
\mathcal{R}_{S_k}^*\tilde{\chi}\psi_0(\widetilde{\mathcal{P}})\mathcal{R}^*
\label{e:slough2}
\end{align}
(where we have combined the remainder terms).
Finally, by Lemma \ref{l:frequencyPseudolocalFinal},
\begin{align}\nonumber
&(1-\psi(\widetilde{\mathcal{P}}))\chi(1-\psi_1(\widetilde{\mathcal{P}}))\tilde{\chi}\mathcal{R}_{S_k}^*\tilde{\chi}\psi_0(\widetilde{\mathcal{P}})\mathcal{R}^*\\
&\qquad\qquad=(1-\psi(\widetilde{\mathcal{P}}))\chi O(k^{-\infty})_{\mathcal{D}(\mathcal{P}^{-N})\to \mathcal{D}(\mathcal{P}^N)}\psi_0(\widetilde{\mathcal{P}})\mathcal{R}^*;
\label{e:slough3}
\end{align}
the result then follows by combining \eqref{e:slough0}, Lemma \ref{l:coerce} (together with elliptic regularity to control the $H_k^2(\Omega\setminus \Gamma_{ p})$ norm), \eqref{e:slough1}, \eqref{e:slough2}, and \eqref{e:slough3}.
\end{proof}

\section{Abstract $hp$-FEM convergence result}\label{s:abstract}

As discussed in \S\ref{s:discussion}, the main idea of the proof of Theorem \ref{t:main} is to split the solution operator into ``high'' and ``low" frequency components, with this splitting then used to bound the ``adjoint approximability constant" $\eta_{L^2\to H^1_k}(V_{\mathcal T}^p)$ in the Schatz duality argument (see Lemma \ref{l:Schatz} below). The following result -- Theorem \ref{t:thm71_abstract} -- bounds the high-frequency components (using Lemma \ref{l:trainLate1}), leaving the low-frequency components to be controlled separately -- the advantage of stating this result is that it separates out the parts of the proof of Theorem \ref{t:main} that require differentiability assumptions, such as Gevrey-$\sigma$ or analyticity.

\begin{theorem}\label{t:thm71_abstract}
Let $d\leq 3$, $P_k$ be an $H^\infty$ Helmholtz PML operator (in the sense of Definition \ref{d:helmholtzProblem}), $\widetilde{\cP}_i$, $i=1,2$ satisfy the assumptions of Lemma \ref{l:frequencyPseudolocalFinal}, and $U$ be a neighborhood of $\partial\Omega_{\tr}$ contained in PML region. Then there is $\psi\in C^\infty_c(\R)$ such that for all
$\chi_1,\chi_2\in \overline{C^\infty}(\Omega;[0,1])$ with 
$$
\{\chi_1+\chi_2<1\}\subset U, \qquad\supp \chi_i\cap \partial\Omega_{\tr}=\emptyset,\qquad\supp \big((1-\chi_i)\chi_i\big)\cap (\Gamma_p\cup \Gamma_s)=\emptyset
$$
and $N,k_0>0$ and $\Upsilon>0$ there exist $C,\, h_0 >0$ such that for all affine-conforming $C^{p+1}$ simplicial triangulations $\mathcal{T}$   with constant $\Upsilon>0$ 
(in the sense of Definition \ref{d:Cr})
such that
    \[
    \partial \Omega \cup \Gamma_p
    \subset \bigcup_{T \in \mathcal{T}} \partial T,
    \]
    $p\geq 1,\, 0 <h(\mathcal{T})<h_0,$ and $k\geq k_0$,
    \begin{align}\nonumber
         &\eta_{L^2\to H^1_k}(V_{\mathcal{T}}^p) 
       \\
       &\leq C 
            \frac{hk}{p} 
             + C \bigg[ \left(\frac{hk}{p}\right) k^{-N}
            +\|(I-\Pi_h^p)\psi(\widetilde{\mathcal{P}}_1)\chi_1\|_{L^2(\Omega)\to H_k^1(\Omega)}
            +
\|(I-\Pi_h^p)\psi(\widetilde{\cP}_2)\chi_2\|_{L^2(\Omega)\to H_k^1(\Omega
)}
\Bigg]
\rho(k).
\label{est:estimate_thm71_abstract}
       \end{align}
\end{theorem}

\begin{proof}
The proof is similar to that of \cite[Theorem 7.1]{galkowski2025numerical}:~we split into high and low frequencies using (i) the Fourier transform away from all boundaries and interfaces and (ii) $\psi(\widetilde{\mathcal{P}})$ near all non-PML boundaries and interfaces. 
The contribution from the PML is dealt with using ellipticity of $P_k$ (and hence good solution-operator bounds) in the PML region (recalled from 
\cite[Theorem 4.2 and \S C.3]{AGS2}).

To localize the operator $\cR^*$ near and away from boundaries we introduce the following spatial cut-off functions:~let $\varphi_{P}\in \overline{C^\infty}(\Omega)$ with $\supp \varphi_P$ contained inside the PML and $\supp (1-\varphi_P)\cap U=\emptyset$.
Let also $\psi\in C_c^\infty(\mathbb{R})$ with $\supp(1-\psi)\cap [-M,M]=\emptyset$ for some $M>0$ to be chosen large enough.  We then write 
\begin{align*}
\mathcal{R}^*&=\varphi_P\mathcal{R}^* + (1-\varphi_P)\mathcal{R}^*\\
&=\varphi_P\mathcal{R}^*+\chi_1(1-\varphi_P)\mathcal{R}^*+(1-\chi_1)(1-\varphi_P)\mathcal{R}^*\\
&=
\underbrace{
\varphi_P\mathcal{R}^*+(1-\psi(\widetilde{\mathcal{P}}_1)\big)\chi_1 (1-\varphi_P)\mathcal{R}^*+(1-\psi(\widetilde{\mathcal{P}}_2)\big)(1-\chi_1)(1-\varphi_P) \mathcal{R}^*}_{\mathcal{H}\mathcal{R}^*} \\
&\qquad +\underbrace{\psi(\widetilde{\mathcal{P}}_1)\chi_1(1-\varphi_P)\mathcal{R}^*+ \psi(\widetilde{\mathcal{P}}_2)(1-\chi_1)(1-\varphi_P)\mathcal{R}^*}_{\mathcal{L}\mathcal{R}^*}.
\end{align*}
This decomposition of $\mathcal{R}^*$ into $\mathcal{H}\mathcal{R}^*$ and $\mathcal{L}\mathcal{R}^*$ is one into a well-controlled, in principle high-frequency part and a low frequency, potentially growing with $k$ part. In particular, we claim that 
\begin{equation} 
\label{e:highFreq}
\|(I-\Pi_h^p)\mathcal{H}\mathcal{R}^*\|_{L^2\to H_k^1}\leq C\frac{hk}{p}+C_Nk^{-N}\rho(k)\Big(\frac{hk}{p}\Big),
\end{equation}
and 
\begin{multline}
\label{e:lowFreq}
\|(I-\Pi_h^p)\mathcal{L}\mathcal{R}^*\|_{L^2\to H_k^1}\\
\leq \Big(\|(I-\Pi_h^p)\psi(\widetilde{\mathcal{P}}_1)\chi_1\|_{L^2(\Omega)\to H_k^1(\Omega)}
            +
\|(I-\Pi_h^p)\psi(\widetilde{\cP}_2)\chi_2\|_{L^2(\Omega)\to H_k^1(\Omega)}\Big)\rho(k)
\end{multline}
The estimate~\eqref{e:lowFreq} follows directly from the definition of $\rho(k)$ and the fact that $|(1-\varphi_P)(1-\chi_1)/\chi_2|<C$ and $|(1-\varphi_P)|<C$. Therefore, it remains to prove~\eqref{e:highFreq}.

We start with $(I-\Pi_h^p)\varphi_P\mathcal{R}^*$.  
By ellipticity of $P_k$ in the PML region (see e.g.~\cite[Theorem 4.2 and \S C.3]{AGS2}), 
\begin{equation}
\label{e:PMLEstimate}
\|\varphi_P\mathcal{R}^*\|_{L^2\to H_k^2}\leq C.
\end{equation}
By the analogue of Theorem \ref{t:polyApprox} for $C^{p+1}$ triangulations (which holds by, e.g., \cite[Theorem A.11]{galkowski2025numerical} with $\mathcal{T}_1=\Omega$ and $\mathcal{T}_2=\emptyset$),
\beq\label{e:Cpplus1}
\|(I-\Pi_h^p)\|_{H_k^{M+1}\to H_k^1}\leq \bigg(\frac{Chk}{p}\bigg)^M.
\eeq
Therefore,
\begin{equation}
\label{e:PMLApprox}
\|(I-\Pi_h^p)\varphi_P\mathcal{R}^*\|_{L^2\to H_k^1}\leq C\frac{hk}{p}.
\end{equation}

Next, we estimate $(1-\psi(\widetilde{P}_1))\chi_1(1-\varphi_P)\mathcal{R}^*$. 
By~\eqref{e:highElliptic1} with $M$ large enough,
\begin{equation}
\label{e:highElliptic}
\|(1-\psi(\widetilde{\cP}_1))(1-\varphi_P)\chi_1\mathcal{R}^*\|_{L^2\to H_k^2(\Omega\setminus \Gamma_{ p})\cap H_k^1(\Omega)}\leq C(1+k^{-N}\rho(k)),
\end{equation}
where we have used that $\supp(1-\varphi_P)\chi_1\cap \partial \Omega_{\tr}=\emptyset$ and $(1-\varphi_P)\chi_1$ is constant on $\Gamma_{ p}\cup \Gamma_s$.
Therefore, by \eqref{e:Cpplus1},
\begin{equation}
\label{e:highPart}
\|(I-\Pi_h^p)(1-\psi(\widetilde{\cP}_1))(1-\varphi_P)\chi_1\mathcal{R}^*\|_{L^2\to H_k^1}\leq C\frac{hk}{p}(1+k^{-N}\rho(k)).
\end{equation}
The estimate for $(I-\Pi_h^p)(1-\psi(\widetilde{\cP}_2))(1-\varphi_P)(1-\chi_1)\mathcal{R}^*$ is identical and hence the lemma follows. 
\end{proof}

\section{Construction of an operator $\widetilde{\cP}$ under Gevrey-type regularity assumptions}
\label{s:construction}

\subsection{Construction of $\widetilde{\cP}$}

In this section, we construct of an operator $\widetilde{\cP}$, such that for all $n \in \mathbb{N}$, $\mathcal{D}(\widetilde{\cP}^n) =\mathcal{D}( \cP^n) $. 

The crucial point is that given a function, $f$, that is Gevrey-$\sigma$ at a hypersurface, $\Gamma$ (in the sense of Definition \ref{d:Gev_at_Gamma}), there exists a function $\tilde{f}$ that is Gevrey-$\sigma$ in a neighborhood of $\Gamma$ and agrees with $f$ to infinite order at $\Gamma$. This is the content of the next lemma.

\begin{lemma}\label{eq:neighborhood_extension}
    Suppose that $\widetilde{\Gamma}\subset \overline{\Omega}$ is a $G^\sigma$ closed, embedded hypersurface and $a\in \overline{C^\infty}(\Omega\setminus\ \widetilde{\Gamma})$ is $G^\sigma$ at $\widetilde{\Gamma}$. Then, there is a neighborhood, $V$ of $\widetilde{\Gamma}$ and $\widetilde{a}\in G^\sigma(V\setminus \widetilde{\Gamma})$ such that for all $N>0$, there is $C>0$ such that 
    $$
    |\widetilde{a}(x)-a(x)|\leq C d(x,\widetilde{\Gamma})^N,\qquad x\in V\setminus \widetilde{\Gamma}.
    $$
\end{lemma}
\begin{proof}
We treat the case $\sigma=1$ and $\sigma>1$ separately. In both cases, we find two functions $\tilde{a}_1$ and $\tilde{a}_2$ that are $G^{\sigma}$ in $V$ and each of which matches all derivatives of $a$ when restricted from one side of $\tilde{\Gamma}$. The extension $\tilde{a}$ is then given by defining $\tilde{a}$ by $\tilde{a}_1$ on one side of $\tilde{\Gamma}$ and $\tilde{a}_2$ on the other.

We now construct one such extension, denoted below by  $\tilde{a}$.

    We first treat $\sigma=1$. To do this, observe that at every point $x_0\in \widetilde{\Gamma}$, the Taylor series of $a$ at $x_0$ has a positive radius of convergence $R_{x_0}$. Denote by $a_{x_0}$ the (convergent) Taylor series in $B(x_0,R_{x_0})$. Notice that $a_{x_0}$ is then analytic in $B(x_0,R_{x_0})$ and, since $\widetilde{\Gamma}$ is analytic as are $\partial_{\nu}^ja|_{\widetilde\Gamma}$, $\partial_{\nu}^ja_{x_0}|_{\widetilde{\Gamma}\cap B(x_0,R_{x_0})}=\partial_{\nu}^ja|_{\widetilde{\Gamma}\cap B(x_0,R_{x_0})}$. By compactness, we may cover $\widetilde{\Gamma}$ by finitely many such balls. 
    
    Now, consider the overlap of $V_{x_1,x_0}:=B(x_1,R_{x_1})\cap B(x_2,R_{x_2})$. Then, on $\widetilde{\Gamma}$, $a_{x_0}$ and $a_{x_1}$ agree together with all derivatives and hence, since they are both analytic in $V_{x_1,x_0}$, the $a_{x_0}|_{V_{x_1,x_0}}=a_{x_1}|_{V_{x_1,x_0}}$. Hence, we may set 
    $$
    \widetilde{a}(x)=a_{x_0}(x),\quad x\in B(x_0,R_{x_0}).
    $$

When $\sigma>1$, the existence of an extension $\widetilde{a}$ follows from \cite{bruna1980extension}. More precisely, 
forming the Whitney jet associated with $a$ on $\Gamma$ (i.e. the collection of all derivatives of $a$ at $\Gamma$) and apply \cite[Theorem 3.1]{bruna1980extension} to obtain
$\tilde a$ such that \eqref{eq:neighborhood_extension} holds.
\end{proof}

\begin{lemma}\label{l:construction}
    Suppose that $\cP$ is a self-adjoint second-order, elliptic, semiclassical differential operator with coefficients that are $G^\sigma$ at $\Gamma_s\cup \Gamma_p $. Then there is a neighborhood $V$ of $\Gamma_s\cup \Gamma_p$ and $\widetilde{\cP}$ a self-adjoint second-order, elliptic, semiclassical differential operator such that its coefficients are $G^\sigma$ on $V$ and for all $n\in \mathbb{N}$, $\mathcal{D}(\widetilde{\cP}^n) =\mathcal{D}( \cP^n)$.
\end{lemma}
\begin{proof}
Since $\cP$ is a self-adjoint, second-order, elliptic, semiclassical differential operator there exist $\{a_\alpha(\cdot, \cdot)\}_{|\alpha|\leq 2}$, $A$, and $n$ such that
    \[\cP u= \sum_{|\alpha| \leq 2} a_\alpha(x,k)\,(k^{-1} D_x)^\alpha u=-k^{-2}\nabla \cdot(A\nabla u) -n u.\]
    Without loss of generality, we assume that 
    $$   \sum_{|\alpha|=2}A^{ij} (x)\xi_{i}\xi_j>|\xi|^2>0.
    $$
        We look for an extension of the form 
    \[\widetilde{\cP}u = -k^{-2}\nabla\cdot( \tilde{A}\nabla u) -\tilde{n}u.\]
    We choose $M$ large enough so that $\mathcal{P}+M\geq 1$ and define the self-adjoint operator $\widetilde{\cP}$ as the Friedrichs extension of the sesquilinear form $\widetilde{Q}:\mathcal{D}([\mathcal{P}+M]^{1/2})\times \mathcal{D}([\mathcal{P}+M]^{1/2})\to \mathbb{C}$,
    \begin{gather*}
    \widetilde{Q}(u,v):=\langle \tilde{A}k^{-1}\nabla u,k^{-1}\nabla v\rangle-\langle \widetilde{n} u,v\rangle.
    \end{gather*}
    Observe that, using Green's formula, to ensure  $\mathcal{D}(\widetilde{\cP}^n) =\mathcal{D}( \cP^n)$ for all $n\in \mathbb{N}$, it suffices that for all multiindices $\alpha$ with $|\alpha|\leq 2$, and all $m\geq 0$,
    \beq
\label{eq:matching_dom_cond}\frac{\partial^m}{\partial \nu^m} A^{ij} = \frac{\partial^m}{\partial \nu^m} \tilde{A}^{ij},\qquad \frac{\partial^m}{\partial \nu^m} n= \frac{\partial^m}{\partial \nu^m} \tilde{n}  \quad \ton \Gamma_s\cup \Gamma_p.
    \eeq
For the equivalence of norms,~\eqref{e:graphNorm}, it is then sufficient that $\widetilde{\mathcal{P}}$ is an elliptic semiclassical differential operator of order 2.

Lemma~\ref{eq:neighborhood_extension} ensures the existence of a neighborhood $\tilde V$ of $\Gamma_s\cup \Gamma_p$ and $\dbtilde a_\alpha \in G^\sigma(\tilde V)$ such that \eqref{eq:matching_dom_cond} holds. (Strictly speaking, we use two applications for $\Gamma_p$, one on either side of the interface.)

Let $\chi \in \overline{C^\infty} (\Omega;[0,1])$ with $\supp (1-\chi)\cap (\Gamma_s\cup \Gamma_p)=\emptyset$ and $\supp \chi\subset \widetilde{V}$. Define 
$$
\widetilde{A}^{ij}= \chi\dbtilde{A^{ij}}+ (1-\chi)A^{ij},\qquad \widetilde{n}=\chi \dbtilde{n}+(1-\chi)n.
$$
To see that $\widetilde{A}$ remains elliptic, observe that, shrinking $\widetilde{V}$ if necessary, we may assume that
$$
\inf_{x\in V,|\xi|=1,\xi\in\mathbb{R}^d}\dbtilde{A}^{ij}(x)\xi_i\xi_j>0.
$$
Therefore, since $\mathcal{P}$ is elliptic,  $\supp \chi \subset \widetilde{V}$, and $0\leq \chi\leq 1$, there is $c>0$ such that
$$
\Re\sum_{|\alpha|=2}[\chi \dbtilde A^{ij}+(1-\chi)  A^{ij}]\xi_i\xi_j\geq \chi (x)c|\xi|^2+(1-\chi(x))c|\xi|^2=c|\xi|^2, 
$$
and hence $\widetilde{\mathcal{P}}$ is elliptic. Choosing $V\subset \widetilde{V}$ a neighborhood of $(\Gamma_s\cup\Gamma_p)$ with $\supp (1-\chi)\cap V=\emptyset$ then finishes the proof.
\end{proof}

\subsection{Bound on functions of $\widetilde{\cP}$}

\begin{theorem}\label{thm:psi_low_freq_est}
    Let $\sigma\ge1$. Let $k_0>0$, $\Gamma_s$, $\Gamma_p$ be Gevrey-$\sigma$ and let  $\psi \in C_c^\infty(\R)$ be a compactly supported function. Suppose $\widetilde{\cP}$ is a self-adjoint, elliptic semiclassical second-order differential operator with Gevrey-$\sigma$ coefficients in a neighborhood $V$ of $\Gamma_s \cup \Gamma_p$. Let a smaller neighbourhood $\widetilde V \Subset V $. There exists a constant $C_0\geq 1$ such that for all $\alpha\in\mathbb{N}^d$ and all $k\geq k_0$
    \beq\label{est:psi_low_freq_est}
    \| \partial^\alpha  \psi(\widetilde{\cP})\|_{L^2(\Omega)\to L^2(\widetilde{V}\setminus\Gamma_p)} \leq C_0^{|\alpha|+1} \max\{k, |\alpha|^\sigma\}^{|\alpha|}.
    \eeq
\end{theorem}

To prove Theorem~\ref{thm:psi_low_freq_est}, we pay special attention to how fast the neighbourhoods  $\widetilde V \Subset V $ shrink to avoid degeneracy of the estimate. To do this,
we use the Gevrey elliptic estimates in Appendix~\ref{sec:elliptic_reg_theory}. We first start with a result on half balls $V_R:= B_R \cap \{x_n>0\}$, which allows us to obtain control on the derivatives on shrinking domains with the derivative order. 
\begin{lemma}[Flattened boundary]
\label{p:psi_low_freq_est_flattened_bndry}
    Let $\sigma\geq1$. Let $R_*,k_0>0$, let  $\psi \in C_c^\infty(\R)$ be a compactly supported function. Suppose $\widetilde{\cP}$ is a self-adjoint, elliptic semiclassical second-order differential operator on $\Omega$ with Gevrey-$\sigma$ coefficients in a half-ball $V_{R_*}$. There exists a constant $C_0\geq 1$ such that for all $\alpha\in\mathbb{N}^d$, for all $R \in (0,R_*]$, 
    and all $k\geq k_0$
    \beq\label{est:psi_low_freq_est2_bndry}
      \| \partial^\alpha  \psi(\widetilde{\cP})\|_{L^2(\Omega)\to L^2(V_{R/2})} \leq C_0^{|\alpha|+1} \max\big\{k,|\alpha|^\sigma\big\}^{|\alpha|}.
    \eeq
\end{lemma}
\begin{proof}
We actually prove the stronger estimate that, for $|\alpha|\geq 0$,
\beq\label{est:psi_low_freq_est2_bndryb}
   \max_{0\leq \rho\leq \frac{R}{2\max(|\alpha|,1)}} \rho^{|\alpha|\sigma}   \| \partial^\alpha  \psi(\widetilde{\cP})\|_{L^2(\Omega)\to L^2(V_{R-|\alpha|\rho})} \leq C_0^{|\alpha|+1} \max\big\{\frac{k}{\max(|\alpha|^\sigma,1)},1\big\}^{|\alpha|};
    \eeq
the result \eqref{est:psi_low_freq_est2_bndry} then follows by setting $\rho=\frac{R}{2|\alpha|}$ in \eqref{est:psi_low_freq_est2_bndryb}.

To prove \eqref{est:psi_low_freq_est2_bndryb}, we proceed by induction over the order $|\alpha|$.
    For $|\alpha| \in \{0,1\}$, \eqref{est:psi_low_freq_est2_bndry} is satisfied provided $C_0$ is large enough. We additionally choose $C_0\geq A:= B_0 B\geq 1$, with $B_0, B$ as in Proposition~\ref{prop:isotropic}. Assume by induction that~\eqref{est:psi_low_freq_est2_bndry} holds for all $0\leq |\alpha|\leq a-1$ and some $a\geq2$. 
  
   Using Proposition~\ref{prop:isotropic} with boundary data $h=g=0$,  
   we write 
    \begin{equation*}
       \begin{aligned}
        &\max_{0\leq \rho\leq \frac{R}{2|\alpha|}} \rho^{|\alpha|\sigma}  \| \partial^\alpha \psi(\widetilde{\cP}) u\|_{L^2(V_{R-|\alpha|\rho})} \\
        &\leq  A^{a}\max \Big\{
        \frac{k}{ a^\sigma},1\Big\}
        ^{a}\Big(  \max_{0 \leq |\beta|\leq a-2} 
        \max\Big\{\frac{k}{(|\beta|+2)^\sigma},1\Big\}^{-|\beta|} R_*^{2\sigma} A^{-|\beta|-1} \max_{0\leq \rho\leq \frac{R}{2(|\beta|+1)}} \rho^{|\beta|\sigma} 
        \|\partial^{\beta}  \widetilde{\cP}  \psi(\widetilde{\cP}) u\|_{L^2(V_{R-|\beta|\rho})} \\& \hspace{5cm}
        + \| \psi(\widetilde{\cP})u\|_{H^1_k(V_{R})}\Big).
        \end{aligned}
        \end{equation*}

        Let $\tilde \psi(\widetilde{\cP})\in C_c^\infty (\R)$ $\supp(1-\tilde \psi) \cap \supp(\psi) = \emptyset$, we write 
        $\widetilde{\cP} \psi(\widetilde{\cP}) = (\psi \tilde \psi) (\widetilde{\cP}) \widetilde{\cP}\tilde{\psi}(\widetilde{\cP})$, and use the induction hypothesis to find
          \begin{equation*}
       \begin{aligned}
        & \max_{0\leq \rho\leq \frac{R}{2|\alpha|}} \rho^{|\alpha|\sigma}  \| \partial^\alpha \psi(\widetilde{\cP}) u\|_{L^2(V_{R-|\alpha|\rho})} \\
        & 
        \leq  
    A^{a}\max \Big\{
        \frac{k}{ a^\sigma},1\Big\}\left(e^{2\sigma} R_*^{2\sigma}( A^{-1}C_0)^{a-1}
        \|\tilde \psi(\widetilde{\cP}) \widetilde{\mathcal{P}}\tilde{\psi}(\widetilde{\cP}) u\|_{L^2(\Omega)} + \|\psi(\widetilde{\cP}) u\|_{H^1_k(V_R)}\right)\\
        & 
        \leq  e^{2\sigma} R_*^{2\sigma}A^{a}( A^{-1}C_0)^{a-1}  \Big\{
        \frac{k}{ a^\sigma},1\Big\} \left( \|\tilde \psi(\widetilde{\cP}) \widetilde{\mathcal{P}} \tilde{\psi}(\widetilde{\cP})u\|_{L^2(\Omega)} + e^{-2\sigma } R_*^{-2\sigma}(A^{-1}C_0)^{1-a}\|\psi(\widetilde{\cP}) u\|_{H^1_k(V_R)}\right)
        \\
        & 
        \leq  e^{2\sigma} R_*^{2\sigma} A {C_0}^{a-1}  \Big\{
        \frac{k}{ a^\sigma},1\Big\} \left( \|\tilde \psi(\widetilde{\cP}) \widetilde{\mathcal{P}} \tilde{\psi}(\widetilde{\cP})u\|_{L^2(\Omega )} + e^{-2\sigma } R_*^{-2\sigma}(A^{-1}C_0)^{1-a}\|\psi(\widetilde{\cP}) u\|_{H^1_k(V_R)}\right),
        \end{aligned} 
\end{equation*}
        where we have used that for all $\beta \in \mathbb{N}^d$
        \[
        \Bigg(\frac{(|\beta|+2)^\sigma}{\max\{|\beta|^\sigma,1\}}\Bigg)^{|\beta|} \leq e^{2\sigma}.
        \]
            Since $\widetilde{\psi}$ has compact support, 
        \[\|\widetilde \psi(\widetilde{\cP}) \widetilde{\mathcal{P}} \tilde{\psi}(\widetilde{\cP})u\|_{L^2(\Omega)} \leq C_1 \|u\|_{L^2(\Omega)},
        \]
        and by the induction hypothesis 
        \[
        \|\psi (\widetilde{\cP}) u\|_{H^1_k(V_R)} \leq \Big(C_0 + R_*^{-\sigma} C_0^2 \max\Big\{\frac{1}{k_0},1\Big\}\Big)\|u\|_{L^2(\Omega)} \leq c_{k_0,R_*} C_0^2\|u\|_{L^2(\Omega)},
        \]
        for some constant  $c_{k_0,R_*}>0$. Thus
         \begin{equation*}
       \begin{aligned}
        \max_{0\leq \rho\leq \frac{R}{2|\alpha|}} \rho^{|\alpha|\sigma} \| \partial^\alpha \psi (\widetilde{\cP})u\|_{L^2(V_{R-|\alpha|\rho})} 
        &\leq   e^{2\sigma} R_*^{2\sigma} A {C_0}^{a-1}  \Big\{
        \frac{k}{ a^\sigma},1\Big\} \left(  C_1 + c_{k_0,R_*}  A^a  C_0^2\right)\|u\|_{L^2(\Omega)}\\
        &
        \leq  \left( e^{2\sigma} R_*^{2\sigma} A {C_0}^{a-1}C_1   +  c_{k_0,R_*} A^a C_0^2\right) \max \left(
        k, a^\sigma\right)^{a}\|u\|_{L^2(\Omega)}\\
        &
        \leq C_0^{a+1}   \Big\{
        \frac{k}{ a^\sigma},1\Big\}  \|u\|_{L^2(\Omega)}
        \end{aligned} 
\end{equation*}
provided $C_0$ is chosen such that
\[
C_0 \geq \max\big\{\sqrt{2e^{2\sigma} R_*^{2\sigma} A C_1}, 2 c_{k_0,R_*}  A^2\big\}.
\]
\end{proof}

Next we consider estimates on flattened interfaces. To simplify presentation, we use the notation that $I_R = B_R \cap \{x_n=0\}$ and $(B\setminus I)_R = B_R\setminus I_R$.

\begin{lemma}[Flattened interface]
\label{p:psi_low_freq_est_flattened_interface}
    Let $\sigma\geq 1$. Let $R_*,k_0>0$, let  $\psi \in C_c^\infty(\R)$ be a compactly supported function. Suppose $\widetilde{\cP}$ is a self-adjoint, elliptic semiclassical second-order differential operator on $\Omega$ with Gevrey-$\sigma$ coefficients in a ball $(B\setminus I)_{R_*}$. There exists a constant $C_0\geq 1$ such that for all $\alpha\in\mathbb{N}^d$, for all $R \in (0,R_*]$,
    and all $k\geq k_0$
\beq\label{est:psi_low_freq_est2_interface}
     \| \partial^\alpha  \psi(\widetilde{\cP})\|_{L^2(\Omega)\to L^2((B\setminus I)_{R/2})} \leq C_0^{|\alpha|+1} \max\big\{k,|\alpha|^\sigma\big\}^{|\alpha|}
    \eeq
\end{lemma}
\begin{proof}
The proof follows similarly to that of Proposition~\ref{p:psi_low_freq_est_flattened_bndry} by using as initial building block the estimate of Proposition~\ref{prop:transmission_isotropic}    
\end{proof}

The final building block toward proving Theorem~\ref{thm:psi_low_freq_est}, is to have interior estimates on shrinking balls with the derivative order.

\begin{lemma}[Interior estimate]
\label{p:psi_low_freq_est_flattened_interior}
    Let $\sigma\geq1$. Let $R_*,k_0>0$, let  $\psi \in C_c^\infty(\R)$ be a compactly supported function. Suppose $\widetilde{\cP}$ is a self-adjoint, elliptic semiclassical second-order differential operator on $\Omega$ with Gevrey-$\sigma$ coefficients in a ball $B_{R_*}$. There exists a constant $C_0\geq 1$ such that for all $\alpha\in\mathbb{N}^d$, for all $R \in (0,R_*]$, 
    and all $k\geq k_0$
    \beq\label{est:psi_low_freq_est2_interior}
     \| \partial^\alpha  \psi(\widetilde{\cP})\|_{L^2(\Omega)\to L^2(B_{R/2})} \leq C_0^{|\alpha|+1} \max\big\{k,|\alpha|^{\sigma}
    \big \}^{|\alpha|}.
    \eeq
\end{lemma}
\begin{proof}
The proof follows similarly to that of Proposition~\ref{p:psi_low_freq_est_flattened_bndry} by using as initial building block the estimate of Proposition~\ref{prop:transmission_isotropic}    
\end{proof}

\begin{proof}[Proof of Theorem~\ref{thm:psi_low_freq_est}]
The proof follows similarly to Theorems\ref{thm:int_Gev_reg},~\ref{thm:uttb_Gev_reg}, and~\ref{thm:transmission}. 
In particular, for each $x_0 \in \overline{ \widetilde V}$, we have three cases
\begin{itemize}
    \item $x_0 \in \overline{\widetilde{V}}\setminus(\Gamma_s\cup\Gamma_p)$ is an interior point: there exists a neighbourhood $\mathcal{U}_2(x_0) \subset V$ of $x_0$ which maps to the ball $B_{R}$ through a Gevrey-$\sigma$ map $\phi$. In $\mathcal{U}_2(x_0)$, the mapped operator $\widetilde \cP$ becomes
    \[\phi^{-1} \circ \breve{\cP} \circ \phi = \widetilde{\cP}. \]
    The operator $\breve \cP$, has Gevrey-$\sigma$ coefficients on the ball $B_{R_*}$, such that the estimate established in Lemma~\ref{p:psi_low_freq_est_flattened_interior} can be applied to $\breve \cP$ with constant $C_0(x_0)$.

     We denote the neighbourhood $\mathcal{U}_1(x_0):=\phi^{-1}(B_{R/2})$.
     \item $x_0 \in \Gamma_s$ is boundary point: there exists a neighbourhood $\mathcal{U}_2(x_0) \subset V$ of $x_0$ which maps to the half-ball $V_{R}$ through a Gevrey-$\sigma$ map $\phi$. In $\mathcal{U}_2(x_0)$, the mapped operator $\breve \cP$, has Gevrey-$\sigma$ coefficients on the ball $B_{R}$, such that the estimate established in Lemma~\ref{p:psi_low_freq_est_flattened_bndry} can be applied with constant $C_0(x_0)$.
      We denote the neighbourhood $\mathcal{U}_1(x_0):=\phi^{-1}(V_{R/2})$.
     \item $x_0 \in \Gamma_p$ is an interface point: there exists a neighbourhood $\mathcal{U}_2(x_0) \subset V\setminus \Gamma_p$ of $x_0$ which maps to the ball $B_{R}\setminus I_{R}$ via a Gevrey-$\sigma$ map $\phi$, such that the estimate established in Lemma~\ref{p:psi_low_freq_est_flattened_interface} can be applied with constant $C_0(x_0)$.
      We denote the neighbourhood $\mathcal{U}_1(x_0):=\phi^{-1}((B\setminus I)_{R/2})$.
\end{itemize}
 Finally, by varying $x_0 \in \widetilde V$, we can extract a finite covering of the compact set $\widetilde V$ by open sets $\mathcal{U}_1(x_0)$. The proof is then complete by combining a finite number of estimates \eqref{est:psi_low_freq_est2_interior}, \eqref{est:psi_low_freq_est2_bndry}, and \eqref{est:psi_low_freq_est2_interface}.
\end{proof}

\color{black}

\section{Proof of Theorem \ref{t:main}}\label{s:proof}

We now control the low-frequencies in Theorem \ref{t:thm71_abstract} using Theorem \ref{thm:psi_low_freq_est}, i.e., using 
the Gevrey-$\sigma$ regularity assumptions made on the boundaries and on the coefficicient at the boundaries.

We now incorporate the low frequencies in the approximability result for $\eta_{L^2\to H^1_k}(V_{\mathcal{T}}^p)$. The low-frequency cut-off ensures ultra-differentiability; the goal is then to show that the growth of the higher-order norms grow in a controlled way with respect to the wave number $k$ and polynomial order $p$. To achieve this result, we rely here on the Gevrey-$\sigma$ regularity assumption made on the boundaries and on the coefficicient at the boundaries.

\begin{theorem}\label{t:thm71}
    Let the assumptions of Theorem~\ref{t:main} hold. Given $M,N,k_0>0$ and $\Upsilon>0$ there exist $C,\, h_0 >0$ such that for all affine-conforming $G^\sigma$ simplicial triangulations $\mathcal{T}$ with constant $\Upsilon>0$ (in the sense of Definition~\ref{d:Gsigma}) such that
    \[
    \partial \Omega \cup \Gamma_p 
    \subset \bigcup_{T \in \mathcal{T}} \partial T,
    \]
    $p\geq m\geq 1,\, 0 <h(\mathcal{T})<h_0,$ and $k\geq k_0$,
    \begin{align}\nonumber
         &\eta_{L^2\to H^1_k}(V_{\mathcal{T}}^p) 
       \\
       &\quad\leq C 
            \frac{hk}{p} 
             + C \bigg[                         \left(\frac{hk}{p}\right)k^{-N}
            +C^{m+1}(hk)^{m}k^{-m-1}
(p+1)^{-m}
\max\Big\{k,(m+1)^{\sigma}\Big\}^{m+1}
\Bigg]
\rho(k).
\label{est:estimate_thm71}
       \end{align}
\end{theorem}

\begin{remark}
    A result analogous to Theorem \ref{t:thm71} when $m=p$ and $\sigma=1$ was proved in 
\cite[Theorem 7.1]{galkowski2025numerical}; the difference is that Theorem \ref{t:thm71} includes the 
bound on the low frequencies (obtained via 
Gevery-$\sigma$ regularity) whereas \cite[Theorem 7.1]{galkowski2025numerical} includes this low frequency bound abstractly, as in Theorem \ref{t:thm71_abstract}.
\end{remark}

\begin{proof}[Proof of Theorem~\ref{t:main} using Theorem \ref{t:thm71}]
We apply the Schatz argument (Lemma \ref{l:Schatz}) to the sesquilinear form $a(\cdot,\cdot)$ associated with \eqref{e:PML1} and the norm \eqref{e:1knorm}. This sesquilinear form has norm independent of $k$, and satisfies Assumption \ref{a:Gaarding1} with $c_G$ and $C_G$ independent of $k$ by \cite[Lemma 2.3]{galkowski2024hp} and \cite[Lemma A.6]{GLS2} (see also \cite[Lemmas 4.1 and 4.2]{galkowski2025numerical}).
It is therefore sufficient to prove that 
$\eta_{L^2\to H^1_k}(V_{\mathcal{T}}^p)$ can be made arbitrarily small. 
By \eqref{e:polyBound}, $\rho(k)$ is polynomially bounded in $k$; thus the first term in square brackets on the right-hand side of \eqref{est:estimate_thm71} can be made small by choosing $N$ sufficiently large. 
Observe that the second-term in square brackets on the right-hand side of \eqref{est:estimate_thm71} contains the free parameter $1\leq m\leq p$. We now choose $m$ depending on $p, k$, and $\sigma$ in the following two different ways -- note that both choices are designed so that the maximum on the right-hand side of \eqref{est:estimate_thm71} is proportional to $k$.
When $k_0\leq k\leq 2$ the result follows by taking $h$ sufficiently small. Therefore, we assume that $k_0\geq 2$.  

When $p+1 \leq  k^{1/\sigma}$, we choose $m=p$. Thus \eqref{est:estimate_thm71} becomes
\begin{equation}\label{est:pre_choice_params}
        \begin{aligned}
         &\eta_{L^2\to H^1_k}(V_{\mathcal{T}}^p) \begin{aligned}
       \leq C 
        \Bigg[ & 
            \frac{hk}{p} + \left(\frac{hk}{p}\right)^M k^{-N} \rho(k)  + \left(C\frac{hk}{p+1}\right)^p    \rho(k)
        \Bigg].
        \end{aligned}
       \end{aligned}
    \end{equation}
Given $\delta, \epsilon>0$, there exists $c>0$ such that if $hk/p\leq c$ and $p\geq 1+ \epsilon \log k$ (exactly as in \eqref{e:thresholds}) then $\eta_{L^2\to H^1_k}(V_{\mathcal{T}}^p)\leq \delta$. The quasi-optimality bound \eqref{e:qo} then follows from the Schatz argument (Lemma \ref{l:Schatz}).      
     When $k^{1/\sigma}\leq p+1$, we
      choose $m+1 =  \max(\lceil k^{1/\sigma}\rceil,2)$. Then 
\begin{equation}\label{est:pre_choice_params_case_2}
        \begin{aligned}
         &\eta_{L^2\to H^1_k}(V_{\mathcal{T}}^p) \begin{aligned}
       \leq C 
        \Bigg[ & 
            \frac{hk}{p} + \left(\frac{hk}{p}\right)^M k^{-N} \rho(k)  + \left(C \frac{hk}{p+1}\right)^{ \max(\lceil k^{1/\sigma}\rceil,2)-1}   \rho(k)
        \Bigg].
        \end{aligned}
       \end{aligned}
    \end{equation}
Taking $hk/(p+1)\leq c$, and using that $c^{\max(\lceil k^{1/\sigma}\rceil,2)-1}\ll k^{-N}$, then makes  $\eta_{L^2 \to H^1_k}(V_{\mathcal{T}}^p)$ small.
\end{proof}

\begin{remark}\label{r:epsquared}
We expect that elliptic regularity under the Denjoy--Carleman class 
\eqref{e:denjoyCarleman} would result in the following analogue of \eqref{est:estimate_thm71}
   \begin{align}\nonumber
         &\eta_{L^2\to H^1_k}(V_{\mathcal{T}}^p) 
       \\
       &\quad\leq C 
            \frac{hk}{p} 
             + C \bigg[                         \left(\frac{hk}{p}\right)^M k^{-N}
            +C^{m+1}(hk)^{m}k^{-m-1}
(p+1)^{-m}
\max\big\{k,e^{\sigma (m+1)}\big\}^{m+1}
\Bigg]
\rho(k).
\label{est:estimate_thm71new}
       \end{align}
Suppose that $\rho(k)$ grows like $k^{M}$. Since the maximum in~\eqref{est:estimate_thm71new} is larger than $k$, we require 
$$
\Big(\tfrac{hk}{p+1}\Big)^m\ll k^{-M},
$$
with $(hk)/(p+1)=c_0$, and hence $m\geq \e \log k$ for some $\e>M/\log(1/c_0)$. 
Furthermore, since the maximum is larger than $e^{\sigma m^2}$
and $m\geq \e \log k$,
we require also
$$
k^{-M}\gg k^{-1}\Big(\tfrac{h}{p+1}e^{\sigma (m+1)}\Big)^{m}\geq k^{-1}\big(c_0 k^{-1}e^{\sigma (m+1)}\big)^{m}\geq k^{-1}\big(c_0 k^{-1+\sigma \e }\big)^{m}.
$$
In particular, if $m= \varepsilon \log k$ then
\[
1 \gg k^{-1} \big(c_0 k^{-1+\sigma \e } e^{M/\varepsilon}\big)^{m}.
\]
Hence, for $k$ large (say greater than 1), we see that we need a double condition on $\varepsilon$ given by
$
M/\log (1/c_0)
<\e<\sigma^{-1}$. This implies that $c_0 < 
e^{-M\sigma}$, so that increasing $\sigma$ implies a corresponding increase in the dimension of the piecewise polynomial approximation space. This shows that our proof, when applied to Denjoy--Carleman classes that are not contained in the union over $\sigma$ of those in~\eqref{e:denjoyCarleman}, will require piecewise polynomial spaces whose dimension grows faster than $k^d$.
\end{remark}

\begin{proof}[Proof of Theorem \ref{t:thm71}]

Let $\widetilde{\mathcal{P}}_1$ be the operator constructed in Lemma \ref{l:construction}. Then there is a neighborhood $W$ of $\Gamma_s \cup \Gamma_p$ such that the coefficients of $\widetilde{\mathcal{P}}_1$ are Gevrey-$\sigma$ in $W$.  
 \begin{gather*}
  W_{\mathcal{T}}:=\bigcup_{T \in \cT\,:\,T \subset W} T.
 \end{gather*}
 Then, let $h_0>0$ be small enough such that $\Gamma_p\cup \Gamma_s \subset W_{\mathcal{T}}$ and let $\chi_1\in \overline{C^\infty}(\Omega)$ such that $\supp \chi_1\cap \partial\Omega_{\tr}=\emptyset$, $\supp (1-\chi_1)\cap (\Gamma_p\cup \Gamma_s)=\emptyset$ and for $0<h<h_0$, 
 $$
 d(\supp \chi_1, \Omega\setminus W_{\mathcal{T}})>c>0.
 $$

 By Theorem~\ref{t:polyApprox}, $m_{T}=M+1$ for $T\subset \Omega\setminus W_{\mathcal{T}}$ and $m_T=m+1$ for $T\subset W_{\mathcal{T}}$ and $s=1$, 
together with Lemma~\ref{l:spatialPseudolocalFinal}, we obtain 
\begin{equation} 
\label{e:lowFrequencyEstimates}
\begin{aligned}
&\|(I-\Pi_h^p)\psi(\widetilde{\cP}_1)\chi_1\|_{L^2\to H_k^1(\Omega)}\\
&\qquad\leq C\Big(\frac{h k}{p}\Big)^{M}k^{-M-1}\|\psi(\widetilde{\cP}_1)\chi_1\|_{L^2\to H^{M+1}(\Omega\setminus W_{\mathcal{T}})}\\
&\qquad\qquad\qquad+\big(Ch(\mathcal{T})k\big)^{m}     (p+1)^{-m}(m+1)^{\sigma (m+1)}  k^{-m-1}\|\psi(\widetilde{\cP}_1)\|_{L^2\to H^{m+1}(W_{\mathcal{T}}\setminus \Gamma_{\rm p})\cap H^1(W_{\mathcal{T}})},\\
&\qquad\leq C\Big(\frac{h k}{p}\Big)^{M}\|\psi(\widetilde{\cP}_1)\chi_1\|_{L^2\to H^{M+1}_k(\Omega\setminus W_{\mathcal{T}})}\\
&\qquad\qquad\qquad+\big(Ch(\mathcal{T})k\big)^{m}     (p+1)^{-m}(m+1)^{\sigma (m+1)}  k^{-m-1}\|\psi(\widetilde{\cP}_1)\|_{L^2\to H^{m+1}(W\setminus \Gamma_{\rm p})\cap H^1(W)},\\
&\qquad\leq C\Big(\frac{h k}{p}\Big)^{M}k^{-N}+\big(Ch(\mathcal{T})k\big)^{m}     (p+1)^{-m}(m+1)^{\sigma (m+1)}  k^{-m-1}\|\psi(\widetilde{\cP}_1)\|_{L^2\to H^{m+1}(W\setminus \Gamma_{\rm p})\cap H^1(W)},
\end{aligned}
\end{equation}
    where we work in the Sobolev norms~\eqref{d:sobolev_norms_scaled} scaled appropriately for Gevrey-$\sigma$ type estimates.

Let $U$ be a neighborhood of $\partial\Omega_{\tr}$ contained in the PML and $\chi_2\in 
\overline{C^\infty}(\Omega)$ such that  $\supp \chi_2\cap (\Gamma_p\cup \Gamma_s\cup \partial \Omega_{\tr})=\emptyset$ and 
$$
\supp(1-\chi_2)\cap \{ \chi_1<\frac{1}{2}\}=\emptyset.
$$
Finally, let $\widetilde{W}$ be a neighborhood of $\supp \chi_2$ such that $\overline {\widetilde{W}}\cap (\Gamma_p\cup \Gamma_s\cup \partial \Omega_{\tr})=\emptyset$.

Our goal now is to construct $\widetilde{\mathcal{P}}_2$ equal to $-k^{-2}\Delta-1$ on $\widetilde{W}$ and such that $\widetilde{P}_2$ agrees with $\mathcal{P}=\Re P_k$ near $\Gamma_p\cup\Gamma_s\cup \partial \Omega_{\tr}$. 
We complete the proof assuming such a
$\widetilde{\mathcal{P}}_2$ exists, and then construct such a $\widetilde{\mathcal{P}}_2$ at the end.
Let
 \begin{gather*}
  \widetilde{W}_{\mathcal{T}}:=\bigcup_{T \in \cT\,:\,T \subset \widetilde{W}} T;
 \end{gather*}
 then 
 there are $c>0$ and $h_0>0$ small enough such that 
 $$
 d(\supp \chi_2,\Omega\setminus \widetilde{W}_{\mathcal{T}})>c,
 $$
and hence the same sequence of estimates leading to~\eqref{e:lowFrequencyEstimates} with $\widetilde{W}_{\mathcal{T}}$ replacing $W_{\mathcal{T}}$ yields
\begin{equation} 
\label{e:lowFrequencyEstimates2}
\begin{aligned}
&\|(I-\Pi_h^p)\psi(\widetilde{\cP}_2)\chi_2\|_{L^2\to H_k^1(\Omega)}\\
&\qquad\leq C\Big(\frac{h k}{p}\Big)^{M}k^{-N}+\big(Ch(\mathcal{T})k\big)^{m}     (p+1)^{-m}(m+1)^{\sigma (m+1)}  k^{-m-1}\|\psi(\widetilde{\cP}_2)\|_{L^2\to H^{m+1}(\widetilde{W})}.
\end{aligned}
\end{equation}

By construction, $\widetilde{\cP}_1$ and $\widetilde{\cP}_2$ satisfy the assumptions of Theorem \ref{thm:psi_low_freq_est} with $V$ equal to, respectively, $W$ and $\widetilde{W}$, so that
\beq\label{e:football1}
\|\psi(\widetilde{\cP}_1)\|_{L^2\to H^{m+1}(W\setminus \Gamma_{ p})\cap H^1(W)}
+ \|\psi(\widetilde{\cP}_2)\|_{L^2\to H^{m+1}(\widetilde{W})}
 \leq 
 \sum_{\ell=0}^{m+1}\frac{1}{\ell!^\sigma}
 C^{\ell+1} \max\{k, \ell^\sigma\}^{\ell}.
\eeq
To bound the sum, 
first suppose that $(m+1)^{\sigma}\leq k$. Then, by Stirling's approximation, 
\begin{equation*}
\sum_{\ell=0}^{m+1}\frac{1}{\ell!^\sigma}C^\ell \max(k,\ell^\sigma)^\ell=\sum_{\ell=0}^{m+1}\frac{1}{\ell!^\sigma}C^{\ell}k^{\ell}
\leq \sum_{\ell=0}^{m+1}\Big(\frac{Ck}{\ell^\sigma}\Big)^{\ell}.
\end{equation*}
Now, for $Ck/\ell^\sigma \geq e$ (which is ensured since $k\geq (m+1)^\sigma$, as long as $C \geq e$) the function $\ell\mapsto(Ck/\ell^\sigma)^\ell$ is increasing. Therefore 
\begin{equation}\label{e:boundingsums1}
\sum_{\ell=0}^{m+1}\Big(\frac{Ck}{\ell^\sigma}\Big)^{\ell}
\leq (m+2) \Big(\frac{Ck}{(m+1)^\sigma}\Big)^{m+1} \leq 
\Big(\frac{Ck}{(m+1)^\sigma}\Big)^{m+1}
\end{equation}
(where we have increased $C$ in the last inequality). 

Next, when $(m+1)^\sigma\geq k$, we use 
\eqref{e:boundingsums1} (in the first sum) and 
Stirling's approximation (in the second sum) to obtain that
$$
\sum_{\ell=0}^{m+1}\frac{1}{\ell!^\sigma}C^\ell \max(k,\ell^\sigma)^\ell=\sum_{\ell=0}^{\lfloor k^{1/\sigma}\rfloor }\frac{1}{\ell!^\sigma}C^{\ell}k^{\ell}+
\sum_{\ell=\lfloor k^{1/\sigma}\rfloor}^{m+1}\frac{1}{\ell!^\sigma}C^{\ell}\ell^{\sigma \ell}\leq C^{ k^{1/\sigma}+1}+\sum_{\ell=\lfloor k^{1/\sigma}\rfloor}^{m+1}C^{\ell}\leq C^{m+2}.
$$
The result \eqref{est:estimate_thm71} then follows by combining these last inequalities with
\eqref{est:estimate_thm71_abstract},~\eqref{e:lowFrequencyEstimates}, \eqref{e:lowFrequencyEstimates2}, 
\eqref{e:football1}.

We now go back to the construction of $\widetilde{\cP}_2$. Let $\chi_3,\chi_4 \in \overline{C^\infty}(\Omega; [0,1])$ such that  $\supp \chi_3\cap (\Gamma_p\cup \Gamma_s\cup \partial \Omega_{\tr})=\emptyset$, $\supp (1-\chi_3) \cap \widetilde{W}= \emptyset$, and $\chi_3^2+\chi_4^2=1$ We then set 
\[\widetilde{\cP}_2 =  \chi_3(-k^{-2}\Delta-1)\chi_3 +  \chi_4\cP \chi_4.
\]
The so-constructed operator $\widetilde{\cP}_2$ verifies the assumptions of Theorem~\ref{thm:psi_low_freq_est} and that \(\widetilde{\cP}_2 =  -k^{-2}\Delta-1\) on $\widetilde{W}$.

\end{proof}

\appendix

\section{Recap of the Schatz duality argument} \label{sec:Schatz}

Let $\mathcal{H}\subset \mathcal{H}_0\subset \mathcal{H}^*$ be  Hilbert spaces with $\mathcal{H}_0$ identified with its dual and let $a:\mathcal{H}\times \mathcal{H}\to \mathbb{C}$ be a sesquilinear form. 

\begin{definition}[Galerkin approximation of $u$ for $a$]
Let $\mathcal{V}_h\subset \mathcal{H}$ and $u\in\mathcal{H}$. A \emph{Galerkin approximation of $u$ in $\mathcal{V}_h$}, is an element $u_h\in\mathcal{V}_h$ such that 
\begin{equation}
\label{e:galerkinOrthogonality}
a(u_h,v_h)=a( u,v_h)\quad\text{ for all }v_h\in\mathcal{V}_h.
\end{equation}
\end{definition}

\begin{definition}[Norm of a sesquilinear form]
\label{d:norm}
For $a:\mathcal{H}\times \mathcal{H}\to \mathbb{C}$ a sesquilinear form,
$$
\|a\|:=\sup_{\substack{u,v\in\mathcal{H}\\u,v\neq0}}\frac{|a(u,v)|}{\|u\|_{\mathcal{H}}\|v\|_{\mathcal{H}}}.
$$
\end{definition}

If $\|a\|<\infty$, we say that $a$ is \emph{bounded}.
We also need the notion of injectivity, surjectivity, and invertibility for sesquilinear forms.
\begin{definition}[Injectivity, surjectivity, and invertibility]
\label{d:injSurBij}
A sesquilinear form is \emph{injective} if there is an operator $\mathcal{R}^*:\mathcal{H}^*\to \mathcal{H}$ such that for all $v\in\mathcal{H}^*$,
\begin{equation}\label{e:R*}
a(w,\mathcal{R}^*v)=\langle w,v\rangle\quad \text{for all }w\in\mathcal{H}.
\end{equation}
$a$ is \emph{surjective} if there is an operator $\mathcal{R}:\mathcal{H}^*\to \mathcal{H}$ such that for all $v\in\mathcal{H}^*$,
$$
a(\mathcal{R}v,w)=\langle v,w\rangle\quad \text{for all }w\in\mathcal{H}.
$$
$a$ is \emph{invertible} if it is both surjective and injective. In this case, $\mathcal{R}^*=(\mathcal{R})^*$. 
\end{definition}

\begin{assumption}[G\aa rding inequality]
\label{a:Gaarding1}
There are $c_{\rm G}>0$ and $C_{\rm G}>0$ such that for all $u\in \mathcal{H}$
\beq\label{eq:Gaarding_ineq}
|a(u,u)|\geq c_{\rm G}\|u\|_{\mathcal{H}}^2-C_{\rm G}\|u\|_{\mathcal{H}_0}^2.
\eeq
\end{assumption}

The Schatz duality argument \cite{Sc:74, ScWa:96, Sa:06} then gives sufficient conditions for quasioptimality in terms of the following adjoint approximability constant.
\begin{definition}[Adjoint approximability]
Let $\mathcal{V}_h\subset \mathcal{H}$, $a$ be an injective sesquilinear form, and $\mathcal{W}$ be a Hilbert space with $\mathcal{W}\subset \mathcal{H}^*$. The \emph{adjoint approximability constant from $\mathcal{W}$ to $\mathcal{H}$} is defined by
\begin{equation}\label{e:eta}
\eta_{_{\mathcal{W}\to \mathcal{H}}}(\mathcal{V}_h):=\|(I-\Pi_h)\mathcal{R}^*\|_{\mathcal{W}\to \mathcal{H}}.
\end{equation} 
\end{definition}
\begin{lemma}[Schatz duality argument]
\label{l:Schatz}
Let $a:\mathcal{H}\times \mathcal{H}\to \mathbb{C}$ be an injective sesquilinear form satisfying Assumption~\ref{a:Gaarding1}, and let $0<\epsilon<1$.  Then for all $u\in\mathcal{H}$ and $\mathcal{V}_h\subset \mathcal{H}$ finite dimensional subspaces satisfying
\begin{equation}
\label{e:schatzCondition}
C_{\rm G}\|a\|^2\big(\eta_{_{\mathcal{H}_0\to \mathcal{H}}}(\mathcal{V}_h)\big)^2\leq (1-\epsilon)c_{\rm G},
\end{equation}
the Galerkin approximation, $u_h$, for $u$ in $\mathcal{V}_h$ exists, is unique, and satisfies
\begin{equation}
\label{e:SchatzFinal}
\|u-u_h\|_{\mathcal{H}}\leq \epsilon^{-1}c_{\rm G}^{-1}\|a\|\|(I-\Pi_h)u\|_{\mathcal{H}}.
\end{equation}
\end{lemma}
\begin{proof}
See, e.g., \cite[Lemma 6.9]{galkowski2025numerical}
\end{proof}

\section{Polynomial approximation results for Gevrey triangulations}\label{sec:projection}

In this appendix, we prove polynomial approximation results for Gevrey-$\sigma$ triangulations that are explicit in both the polynomial degree and the regularity of the approximated function. 
To do this, we define Sobolev norms with a special weighting motivated by the expected growth for Gevrey-$\sigma$ functions (compare to the treatment of the case $\sigma=1$ in \cite[\S A.5]{galkowski2025numerical}).
For an open set $U \subset \R^d$, we define 
\begin{equation}\label{d:sobolev_norms_scaled}
\begin{aligned}
\| u \|^2_{H_{k}^{p}(U)} &:= \sum_{\ell=0}^p \frac{1}{(\ell!)^{2\sigma}}|u|^2_{H_{\ell}^{p}(U)}\\
| u |^2_{H_{k}^{p}(U)} &:= \sum_{\substack{|\alpha| = p \\ \alpha \in \mathbb{N}^d}} \left(\frac{p!}{\alpha!}\right)^\sigma \|(k^{-1}\partial)^\alpha u\|^2_{L^2(U)}.
\end{aligned}
\end{equation}

The main result is as follows. 

\begin{theorem}\label{t:polyApprox}
Let $\Upsilon>0$, $\e>0$, $\sigma \geq 1$, there exists $C>0$ such that for all affine-conforming $G^\sigma$ simplicial triangulations $\mathcal{T}$ with constant $\Upsilon$ (in the sense of Definitions \ref{d:AffineConforming} and \ref{d:Gsigma} below), and all $\{m_{T}\}_{T\in\mathcal{T}}\subset\mathbb{N}$, $d/2< m_T\leq p+1$, there is $\mathcal{C}_{\mathcal{T}}^p : H^1(\Omega) \cap (\oplus_{T \in \mathcal{T}} H^{m_{T}}(T)) \to \mathcal{P}^p_{\mathcal{T}}\cap H^1(\Omega)$ such that for all $0\leq s\leq 1$, $u \in H^1(\Omega) \cap (\oplus_{T\in\mathcal{T}} H^{m_T}(T))$, and all $T\in\mathcal{T}$,

\begin{align}
    \|(I-\mathcal{C}_{\mathcal{T}}^p)u\|_{H^s_k(T)}^2 &\leq   
C    \left(\frac{h_T k}{p+1}\right)^{2(m_T-s)}\Big(\frac{Cm_T^{\sigma}}{k}\Big)^{2m_T}\|u\|^2_{H^{m_T}(T)}.\label{est:fixed_reg_interpolant_estim}
\end{align}
In addition, if $u|_{F}=0$ for $F$ a boundary face of $T$, then $\mathcal{C}^p_{\mathcal{T}}u|_{F}=0$.
As a consequence,
\[
\|(I-\Pi_{\mathcal{P}^p_{\mathcal{T}}}^{H^1})u\|_{H^1_k(\Omega)}^2 \leq  C\sum_{T \in \mathcal{T}} 
    \left(\frac{h_T k}{p+1}\right)^{2(m_T-1)} \Big(\frac{Cm_T^{\sigma}}{k}\Big)^{2m_T}\|u\|^2_{H^{m_T}(T)}, 
\]
where $\Pi_{\mathcal{P}^p_{\mathcal{T}}}^{H^1}$ is the $H^1_k$-orthogonal projection onto $\mathcal{P}^p_{\mathcal{T}} \cap H^1$ and, for any union of faces of triangles, $\Gamma$, and $u\in H^1(\Omega)\cap (\oplus_{T\in\mathcal{T}}H^{m_T}(T))$ with $u|_{\Gamma}=0$, 
\[
\|(I-\Pi_{\mathcal{P}^p_{\mathcal{T},\Gamma}}^{H^1})u\|_{H^1_k(\Omega)}^2 \leq  C\sum_{T \in \mathcal{T}} 
    \left(\frac{h_T k}{p+1}\right)^{2(m_T-1)} \Big(\frac{Cm_T^{\sigma}}{k}\Big)^{2m_T}\|u\|^2_{H^{m_T}(T)}, 
\]
where $\Pi_{\mathcal{P}^p_{\mathcal{T},\Gamma}}^{H^1}$ is the $H^1_k$-orthogonal projection onto 
$$
\big\{w\in \mathcal{P}^p_{\mathcal{T}} \cap H^1\,:\, w|_{\Gamma}=0\big\}.
$$
\end{theorem}

\subsection{The context of Theorem \ref{t:polyApprox}}

The closest approximation results in the literature to Theorem \ref{t:polyApprox} that we are aware of are those in \cite[\S4]{FeSc:20} and \cite[Appendix A]{galkowski2025numerical}.

The paper \cite[\S4]{FeSc:20} uses approximation results on tensorized polynomials 
(building on the results of \cite[\S3.3]{Schwab1998} and \cite[Appendix A]{costabel2005exponential}),
considers triangulations with simplicial elements, and constructs an interpolant of a function $u$ on an element $T$ that depends on the values of $u$ in some neighbouring elements.  (This last point comes from the fact that~\cite{FeSc:20} uses triangulations with affine maps $A_T:\widehat{T}\to T\subset  \Omega$ to the reference simplex that  can be extended to $A_T:\widehat{K}\to K\subset \Omega$ where $\widehat{K}$ is a parallelpiped; the interpolant is then constructed on $T$ by approximating $u|_{K}\circ A_T$, pulling the resulting polynomial back to $K$, and then restricting to $T$.)

In contrast, Theorem \ref{t:polyApprox} covers general polynomials on curved triangulations (with Gevery-$\sigma$ element maps), and the interpolant on an element $T$ depends only on the values of the function on $T$. The fact that we work directly with polynomials of degree $p$ in Theorem~\ref{t:polyApprox} rather than with tensorized polynomials allows us to obtain improved approximation with increasing regularity of the function $u$ all the way to $H^{p+1}$ rather only to $H^{\lfloor p/d\rfloor+1}$. 

In addition, Theorem \ref{t:polyApprox} generalises the results of~\cite[Theorem A.11]{galkowski2025numerical} in two ways: i) it consider Gevrey-$\sigma$ triangulations (as opposed to Gevrey-$1$ triangulations; i.e., those with analytic element maps) and ii) it is explicit in the regularity, $m$, of the function to be approximated for $d/2+\epsilon\leq m\leq p+1$ (rather than only for $m=p+1$).

\subsection{The precise definitions of triangulations}
We start by recalling the definition of triangulations and their properties from~\cite[Appendix B]{galkowski2025numerical} as well as adapting them slightly for the Gevrey setting.

\begin{definition}[Conforming triangulations]
\label{d:conformingMesh}
A finite collection of closed sets $\mathcal{T}$ is a \emph{conforming triangulation} of $\Omega\subset \mathbb{R}^d$ if
\begin{itemize}
\item[(i)] for each $T \in \mathcal{T}$, $T$ is closed, $\mathring{T}$ is nonempty and connected,
\item[(ii)] $\overline{\Omega}= \cup_{T\in \mathcal{T}}T$
\item[(iii)] if $T_1, T_2\in \mathcal{T}$ and $\mathring{T_1}\cap \mathring{T_2} \neq \emptyset$ then $T_1=T_2$.
\item[(iv)] for each $T \in \mathcal{T}$, $T$ has Lipschitz boundary and 
$$
\partial T=\bigcup_{j=1}^d \bigcup_{i=1}^{N_{j,d}}F_{j,i}(T)
$$
with $F_{j,i}(T)$ a smooth, open, embedded submanifold of dimension $d-j$ such that $\partial \big(F_{j,i}(T)\big)=\partial\overline{F_{j,i}(T)}$,
\item[(v)] if $T_1,T_2\in\mathcal{T}$ and $F_{j_1,i_1}(T_1)\cap F_{j_2,i_2}(T_2)\neq \emptyset$, then $F_{j_1,i_1}(T_1)=F_{j_2,i_2}(T_2)$. 
\end{itemize}
\end{definition}

\begin{definition}[$m$-simplex]
Let $m>0$. A set $T\Subset \mathbb{R}^d$ is an \emph{$m$-simplex} if there are $\{x_j\}_{j=0}^m\subset \mathbb{R}^d$ such that
$$
T=\text{convex hull}\big(\{ x_j\}_{j=0}^m\}\big)
$$
$T$ is an \emph{open $m$-simplex} if $T=(\widetilde{T})^\circ$ where $\widetilde{T}$ is an $m$ simplex  and $T\neq \emptyset$. 
\end{definition}

\begin{definition}[Reference element and element maps]
\label{d:referenceElement}
$\widehat{T}\Subset \mathbb{R}^d$ is a reference element for a triangulation $\mathcal{T}$ if $\operatorname{diam}(\widehat{T})=1$ and there exist a family of bi-Lipschitz maps 
$\{\mathcal{F}_{T}\}_{T\in \mathcal{T}}$ such that $T = \mathcal{F}_T( \widehat{T})$ for all $T\in \mathcal{T}$ and for all $(j,i)$ there is $i'$ such that  $F_{j,i}(T)=\mathcal{F}_T(F_{j,i'}(\widehat{T}))$. The triangulation $\mathcal{T}$ is \emph{simplicial} if $\widehat{T}$ is a $d$ simplex.
\end{definition}

We use the following measure of how star-shaped an open subset of $\mathbb{R}^d$ is. 
\begin{definition}[Shape-regularity constant]
Let $U\subset \mathbb{R}^d$ and define
$$
\rho_{\max}(U):=\sup \big\{ \rho>0\,:\, U\text{ is star-shaped with respect to a ball of radius }\rho\big\}.
$$
The \emph{shape-regularity constant of $U$} is defined by
\begin{equation}
\label{e:shapeRegularity}
\gamma(U):=\frac{\operatorname{diam}(U)}{\rho_{\max}(U)}.
\end{equation}
\end{definition}

\begin{definition}[$C^r$-triangulation]\label{d:Cr} Let $\sigma\geq 1$. A triangulation is $C^r$ with constant $\Upsilon>0$ if  $\gamma(\widehat{T})\leq \Upsilon$ 
and for each $T \in \mathcal{T}$ there is $h_T>0$ such that the element map $\mathcal{F}_T$ can be written as $R_T\circ A_T$ where $A_T$ is an affine map and 
\[
\begin{aligned}
& \|\partial A_T\|_{L^\infty} \leq \Upsilon h_T, \qquad \|(\partial A_T)^{-1}\|_{L^\infty} \leq \Upsilon h_T^{-1}, \qquad \|(\partial R_T)^{-1}\|_{L^\infty} \leq \Upsilon,\\
& \|\partial^\alpha R_T\|_{L^\infty} \leq \Upsilon^{1+|\alpha|} \alpha!\quad\tfa |\alpha|\leq r.
\end{aligned}
\]
\end{definition}

\begin{definition}[$G^\sigma$-triangulation]\label{d:Gsigma} Let $\sigma\geq 1$. A triangulation is $G^\sigma$ with constant $\Upsilon>0$ if  $\gamma(\widehat{T})\leq \Upsilon$ 
and
for each $T \in \mathcal{T}$ there is $h_T>0$ such that the element map $\mathcal{F}_T$ can be written as $R_T\circ A_T$ where $A_T$ is an affine map and 
\[
\begin{aligned}
& \|\partial A_T\|_{L^\infty} \leq \Upsilon h_T, \qquad \|(\partial A_T)^{-1}\|_{L^\infty} \leq \Upsilon h_T^{-1}, \qquad \|(\partial R_T)^{-1}\|_{L^\infty} \leq \Upsilon,\\
& \|\partial^\alpha R_T\|_{L^\infty} \leq \Upsilon^{1+|\alpha|} (\alpha!)^{\sigma},
\end{aligned}
\]
 for every multi-index $\alpha$ such that $|\alpha|<\infty$.   
\end{definition}

\begin{definition}[Affine-conforming triangulation]
\label{d:AffineConforming}
A $G^\sigma$ triangulation (in the sense of Definition \ref{d:Gsigma}) is \emph{affine conforming} if whenever $F_{j_1,i_1}(T_1)\cap F_{j_2,i_2}(T_2)\neq \emptyset$ there is an affine isomorphism $\kappa:\mathcal{F}_{T_1}^{-1}(F_{j_1,i_1}(T_1))\to \mathcal{F}_{T_2}^{-1}(F_{j_2,i_2}(T_2))$  such that 
\beq\label{e:isomorphism}
\mathcal{F}_{T_1}|_{
\overline{
\mathcal{F}_{T_1}^{-1}(F_{j_1,i_1}(T_1))
}}
=\mathcal{F}_{T_2}|_{
\overline{
\mathcal{F}_{T_2}^{-1}(F_{j_2,i_2}(T_2))
}}\circ \kappa.
\eeq
\end{definition}

\begin{remark}
Without a condition on the element maps such as Definition \ref{d:AffineConforming}, 
the space of mapped polynomials on two adjacent elements of the triangulation 
may not even intersect non-trivially when restricted to a common face. If the intersection is trivial, there cannot be any non-trivial, globally-$H^1$ mapped polynomials.
\end{remark}

\subsection{Sobolev spaces and coordinate changes in the Gevrey setting}

The following result generalises 
\cite[Lemma A.14]{galkowski2025numerical}.

\begin{proposition} \label{prop:projection}
    Let $\sigma\geq 1$, $d\geq 1$, $C_0>0$. Then there exists $C_1>0$ such that for all $t\geq 0$, $U$, $V \subset \R^d$ open, $\gamma: U \to V$ a bijection satisfying 
    \[
    \|\partial^\alpha \gamma\|_{L^\infty} \leq C_0^{|\alpha|} (\alpha!)^\sigma \ \textrm{for } \alpha \in \mathbb{N}^d, |\alpha| \leq \max(t,1), \ \textrm{and } \|(D\gamma)^{-1}\|_{L^\infty} \leq C_0,
    \]
and $u\in H^t(V)$, then 
    \
\begin{equation}
    \label{ineq:sigma_est_transformation}
    \frac{1}{t!^\sigma} |u \circ \gamma|_{H^t(U)} \leq C_1^t \|u\|_{H^t(V)}.
    \end{equation}
\end{proposition}

To prove Proposition~\ref{prop:projection} we use the following inequalities.
\begin{lemma}
(i) For $a\geq 0$ and $\ell \geq 1$,
\begin{equation}\label{ineq:factorial_multinomial}
a!(\ell!)^a \leq (\ell a)!.
\end{equation}
(ii) Let $n \geq 0$ be an integer and let $a_i \geq 0$ for $i=1,\dots,n$. Then 
\begin{equation}\label{ineq:factorial_sum_prod}
\prod_{i=1}^n a_i! \leq \left(\sum_{i=1}^n a_i\right)!.
\end{equation} 
and, for $\sigma \geq 1$,
\begin{equation}\label{ineq:convexity_power}
\sum_{i=1}^n a_i^\sigma \leq \left(\sum_{i=1}^n a_i\right)^\sigma.   
\end{equation}

\end{lemma}
\begin{proof}
We prove \eqref{ineq:factorial_multinomial} by induction:~assume
\[
\frac{(\ell a)!}{(\ell!)^a} \geq a!
\]
is true for all $0 \leq a \leq a'$ for some $a'\geq 0$, we now show that it holds for $a'+1$.
By the induction hypothesis 
\[
\frac{(\ell (a+1))!}{(\ell!)^{a+1}} \geq \frac{(\ell a+1)\ldots(\ell a+\ell)}{1 \ldots \ell }a!.
\]
We next notice that
\[
\frac{(\ell a+1)\ldots(\ell a+\ell)}{1 \ldots \ell } = \prod_{i=1}^\ell \frac{\ell a+i}{i} \geq\prod_{i=1}^\ell \frac{i a+i}{i}   \geq (a+1)^\ell \geq (a+1),
\]
which proves \eqref{ineq:factorial_multinomial}.
The proof of \eqref{ineq:factorial_sum_prod}
by induction 
is straightforward.
We now prove \eqref{ineq:convexity_power}. Let $S = \sum\limits_{i=1}^n a_i$. If $S=0$ then $a_i = 0$ for all $1\leq i\leq n$ and the result is trivial. We may then assume $S\neq 0$ and let 
\[\lambda_i = \frac {a_i}{S} \in [0,1] \ \ \textrm{for all  } 1\leq i\leq n, \qquad\sum_{i=1}^n \lambda_i =1. \]
Since $\sigma \geq 1$, $\lambda_i^\sigma \leq \lambda_i$, and thus $\sum_{i=1}^n \lambda_i^\sigma \leq 1$.
Therefore
\[\sum_{i=1}^n a_i^\sigma \leq S^\sigma \leq \left(\sum_{i=1}^n a_i\right)^\sigma.
\]
\end{proof}

\begin{proof}[Proof of Proposition~\ref{prop:projection}]
The proof follows closely that of \cite[Lemma A.14]{galkowski2025numerical}, namely,
using the multivariate Faa di Bruno formula and bounding carefully the resulting combinations of sums and products. 
First, we define the set 
\[
p_s(\nu, \lambda) := \left\{ 
\begin{multlined}
(\mu_1, \ldots, \mu_s; \ell_1, \ldots, \ell_s) \in (\mathbb{N}^d)^{2s}\, : \, |\mu_i|>0, 0< \ell_1 < \ldots < \ell_s, \\
\sum_{i=1}^{s} \mu_i =\lambda, \sum_{i=1}^{s} |\mu_i| \ell_i = \nu
\end{multlined}
\right\},
\]
where, following the notation in \cite{constantine1996multivariate}, the order relation $ \mu \leq \nu$ for $\mu, \, \nu \in \mathbb{N}^d$ is defined by either 
\begin{itemize}
    \item[(i)] $|\mu| \leq |\nu|$, or
    \item[(ii)] there exists $1 \leq r \leq d$, such that $\mu_i = \nu_i,$ $i = 1, \ldots, r-1$ and $\mu_r \leq \nu_r$.
\end{itemize}
From the multivariate Faa di Bruno formula~\cite[Theorem 2.1]{constantine1996multivariate}
\begin{equation*}
\begin{aligned}
&\frac{1}{t!^{2\sigma}}  |u \circ \gamma|^2_{H^t(U)} \\
=& \sum_{|\nu| = t} \int_U \left(\sum_{m=1}^t \sum_{|\lambda|=m} (\partial^\lambda u) \circ \gamma \sum_{s=1}^t \sum_{p_s(\nu,\lambda)} \frac{\nu!}{\sqrt{(\nu!)^\sigma(t!)^\sigma}} \prod_{j=1}^s \frac{[\partial^\ell_j \gamma]^{\mu_j}}{\mu_j![\ell_j!]^{|\mu_j|}} \right)^2 
\\
\leq & C_0^{2t}\sum_{|\nu| = t} \int_U \left(\sum_{m=1}^t \sum_{|\lambda|=m}  \frac{1}{(m!)^\sigma}(\partial^\lambda u) \circ \gamma \sum_{s=1}^t \sum_{p_s(\nu,\lambda)} \frac{(\nu!m!)^\sigma}{\sqrt{(\nu!)^\sigma(t!)^\sigma}} \left(\prod_{j=1}^s \frac{1}{\mu_j!} \right) \left(\frac{1}{\nu!}\prod_{j=1}^s [\ell_j!]^{|\mu_j|}\right)^{\sigma - 1} \right)^2
\\
\leq & C_0^{2t}\sum_{|\nu| = t} \int_U \left(\sum_{m=1}^t \sum_{|\lambda|=m}  \frac{1}{(m!\lambda!)^{\sigma/2}} (\partial^\lambda u) \circ \gamma \left(\frac{\lambda!}{m!}\right)^{\sigma/2}   \sum_{s=1}^t \sum_{p_s(\nu,\lambda)} \left(\frac{\nu! m!}{\sqrt{\nu!t!}}\prod_{j=1}^s \frac{1}{\mu_j!} \right)^\sigma  \left(\frac{1}{\nu!}\prod_{j=1}^s \mu_j![\ell_j!]^{|\mu_j|}\right)^{\sigma - 1} \right)^2
\\
\leq & C_0^{2t}\left[\sum_{m=1}^t \sum_{|\lambda|=m}   \frac{1}{(m!\lambda!)^{\sigma}} \int_U\left|(\partial^\lambda u) \circ \gamma  \right|^2 \right] \times \\ & \hphantom{\sum \sum} \left[\sum_{|\nu| = t}  \sum_{m=1}^t \sum_{|\lambda|=m} \left(\frac{\lambda!}{m!}\right)^\sigma  \left(\sum_{s=1}^t \sum_{p_s(\nu,\lambda)} \left(\frac{\nu! m!}{\sqrt{\nu!t!}}\prod_{j=1}^s \frac{1}{\mu_j!} \right)^\sigma  \left(\frac{1}{\nu!}\prod_{j=1}^s \mu_j![\ell_j!]^{|\mu_j|}\right)^{\sigma - 1} \right)^2\right],
\end{aligned}
\end{equation*}
where in the last line we used Cauchy--Schwarz inequality. By the definition of the norm $|\cdot|_{H^t(V)}$,
\begin{equation*}
\begin{aligned}
&\frac{1}{t!^{2\sigma}}  |u \circ \gamma|^2_{H^t(U)} \\
\leq& C_0^{2t} \|u\|_{H^t(V) }\sum_{|\nu| =t} \sum_{m=1}^t \sum_{|\lambda| = m } \left(\frac{\lambda!}{m!}\right)^\sigma \left( \sum_{s=1}^t \sum_{p_s(\nu,\lambda)} \left(\frac{\nu! m!}{\sqrt{\nu!t!}}\prod_{j=1}^s \frac{1}{\mu_j!} \right)^\sigma \left(\frac{1}{\nu!}\prod_{j=1}^s \mu_j![\ell_j!]^{|\mu_j|}\right)^{\sigma - 1}  \right)^2.
\end{aligned}
\end{equation*}
By 
\eqref{ineq:factorial_multinomial}, \eqref{ineq:factorial_sum_prod}, and the fact that 
$\sum_{i=1}^s |\mu_i|\,\ell_i = \nu$ (from the definition of $p_s(\nu,\lambda)$),
\[\prod_{j=1}^s \mu_j![\ell_j!]^{|\mu_j|} \leq \prod_{j=1}^s (|\mu_j| \ell_j )! \leq \left( \sum_{i=1}^s |\mu_i|\,\ell_i \right)! \leq \nu! .\]
Therefore \[ \left(\frac{1}{\nu!}\prod_{j=1}^s \mu_j[\ell_j!]^{|\mu_j|}\right)^{\sigma - 1}  \leq 1, \]
and 
\begin{equation*}
\begin{aligned}
 \frac{1}{t!^{2\sigma}} |u \circ \gamma|^2_{H^t(U)}
&\leq C_0^{2t} \|u\|_{H^t(V) }\sum_{|\nu| =t} \sum_{m=1}^t \sum_{|\lambda| = m } \left(\frac{\lambda!}{m!}\right)^\sigma \left( \sum_{s=1}^t \sum_{p_s(\nu,\lambda)} \left(\frac{\nu! m!}{\sqrt{\nu!t!}}\prod_{j=1}^s \frac{1}{\mu_j!} \right)^\sigma \right)^2
\\
&\leq C_0^{2t} \|u\|_{H^t(V) } \left(\sum_{|\nu| =t} \sum_{m=1}^t \sum_{|\lambda| = m } \frac{\lambda!}{m!}\left( \sum_{s=1}^t \sum_{p_s(\nu,\lambda)} \frac{\nu! m!}{\sqrt{\nu!t!}}\prod_{j=1}^s \frac{1}{\mu_j!} \right)^2\right)^\sigma,
\end{aligned}
\end{equation*}
where in the last line we used \eqref{ineq:convexity_power}.
By \cite[Equation A.5]{galkowski2025numerical} 
there exist $C_1\geq 1$ such that 
\[b_t := \sum_{|\nu| =t} \sum_{m=1}^t \sum_{|\lambda| = m } \frac{\lambda!}{m!}\left( \sum_{s=1}^t \sum_{p_s(\nu,\lambda)} \frac{\nu! m!}{\sqrt{\nu!t!}}\prod_{j=1}^s \frac{1}{\mu_j!} \right)^2 \leq C_1^t\]
and the proof is complete.
\end{proof}

\subsection{Degree- and regularity-explicit polynomial approximation on simplices}
\begin{lemma}[Polynomial approximation on $(-1,1)^d$]
\label{l:cubeApproximate}
Let $d\geq 1$, $\ell\geq 0$. Then there is $C>0$ such that for all $\ell \leq m \leq p+1$ and all $u\in H^m((-1,1)^d)$
\beq\label{e:polyApprox1}
\sum_{|\gamma|\leq \ell}\|T_{\gamma}(I-\Pi_p)u\|^2_{L^2}\leq \frac{C^{2m+1}}{(p+1)^{2(m-\ell)}}
\sum_{|\alpha|=m} 
\|T_{\alpha} u\|^2_{L^2}\leq \frac{C^{2m+1}}{(p+1)^{2(m-\ell)}}
\sum_{|\alpha|=m} 
\big\|\partial^\alpha u\big\|^2_{L^2}.
\eeq
where $\Pi_{p}$ denotes the orthogonal projector onto polynomials of degree less than or equal to $p$ in $L^2((-1,1)^d)$, and 
\beq\label{e:Talpha}
T_{\alpha}u:=\prod_{j=1}^d(1-x_j^{2})^{\alpha_j/2}\partial_{x_j}^{\alpha_j}u.
\eeq
\end{lemma}
\begin{proof}
For $\alpha\in \mathbb{N}^d$, let  $\pi_{\alpha}$ denote that orthogonal projector onto 
$$
\prod_{j=1}^dL_{\alpha_j}(x_j),
$$
where $L_{m}$ is the normalized Legendre polynomial of degree $m$. 
Since each $x_\ell \in [-1,1], (1-x_\ell^2)^{\alpha_\ell/2}\leq 1$, so that 
\begin{align*}
\sum_{|\alpha|=m}\|\partial^\alpha u\|_{L^2}^2\geq\sum_{|\alpha|=m}\Big\|\Big(\prod_{\ell=1}^d(1-x_{\ell}^2)^{\alpha_{\ell}/2}\Big)\partial^\alpha u\Big\|_{L^2}^2=\sum_{|\alpha|=m}\|T_{\alpha}u\|_{L^2}^2.
\end{align*}
Therefore, it is enough to prove the first inequality in \eqref{e:polyApprox1}.

By expanding in Legendre polynomials,
\begin{align*}
\sum_{|\alpha|=m}\|T_{\alpha}u\|_{L^2}^2&= \sum_{|\alpha|=m}\Big\|\sum_{\beta}T_{\alpha} \pi_{\beta}u\Big\|_{L^2}^2.
\end{align*}
We now claim that
\begin{align}
\|T_{\alpha}u\|_{L^2}^2=
\Big\|\sum_{\beta}T_{\alpha} \pi_{\beta}u\Big\|_{L^2}^2
=
\sum_{\beta}\Big\|T_{\alpha}\pi_{\beta}u\Big\|_{L^2}^2.
\label{e:concise1}
\end{align}
Indeed, after integration by parts in the $L^2$ inner product, 
this last equality follows from the fact that
\beq\label{e:magic1}
\pi_{\beta'}\partial^{\alpha} \Big(\prod_{\ell=1}^d(1-x_{\ell}^2)^{\alpha_\ell}\Big)\partial^\alpha\pi_{\beta}=\delta_{\beta\beta'}\pi_{\beta'}\partial^\alpha\Big(\prod_{\ell=1}^d(1-x_\ell^2)^{\alpha_\ell}\Big)\partial^\alpha \pi_{\beta
}
\eeq
(i.e., the operator is zero unless $\beta'=\beta$).
To see \eqref{e:magic1}, recall Rodrigues' formula \cite[Equation 14.7.13]{Di:26}
\[
P_n(x)
=
\frac{1}{2^n n!}
\frac{d^n}{dx^n}(x^2-1)^n
\]
(where $P_n$ are the Legendre polynomials with the normalisation $P_n(1)=1$)
so that 
\begin{equation}
\frac{d^k}{dx^k}\left((1-x^2)^k \frac{d^k}{dx^k}P_n(x)\right)
=
(-1)^k\frac{(n+k)!}{(n-k)!}\,P_n(x),
\qquad n\ge k,
\label{eq:legendre-eigenvalue}
\end{equation}
and thus the differential operator
$$
T_\alpha^*T_\alpha =
\partial^{\alpha}\Big(\prod_{\ell=1}^d(1-x_{\ell}^2)^{\alpha_\ell}\Big)\partial^\alpha$$ 
is diagonalized by the tensor-product Legendre basis; 
therefore, \eqref{e:concise1} holds.
Combining this with  \eqref{eq:legendre-eigenvalue}, we obtain that if $\alpha_\ell\leq \beta_\ell$ for all $\ell=1,\dots,d$ then 
$$
\int_{-1}^1 |T_{\alpha}u|^2dx= \sum_{\alpha\leq \beta}
\frac{(\beta+\alpha)!}{(\beta-\alpha)!}\|\pi_\beta u\|_{L^2}^2
$$
and otherwise is 0. Therefore, for $|\gamma|\leq \ell$, 
\begin{align}\nonumber
\sum_{|\alpha|=m}\|T_{\alpha}u\|_{L^2}^2&=\sum_{\beta}\|\pi_\beta u\|_{L^2}^2\sum_{\substack{|\alpha|=m\\\alpha\leq \beta }}\frac{(\beta+\alpha)!}{(\beta-\alpha)!}\\ \nonumber
&\geq \sum_{|\beta|\geq p+1}\frac{(\beta+\gamma)!}{(\beta-\gamma)!}\|\pi_\beta u\|_{L^2}^2\frac{(\beta-\gamma)!}{(\beta+\gamma)!}\sum_{\substack{|\alpha|=m\\\alpha\leq \beta }}\frac{(\beta+\alpha)!}{(\beta-\alpha)!}\\ 
\nonumber
&= \sum_{|\beta|\geq p+1}\|T_{\gamma}\pi_\beta u\|_{L^2}^2\frac{(\beta-\gamma)!}{(\beta+\gamma)!}\sum_{\substack{|\alpha|=m\\\alpha\leq \beta }}\frac{(\beta+\alpha)!}{(\beta-\alpha)!}\\ \nonumber
&\geq \|T_{\gamma}(I-\pi_p) u\|_{L^2}^2\inf_{|\beta|\geq p+1}\frac{(\beta-\gamma)!}{(\beta+\gamma)!}\sum_{\substack{|\alpha|=m\\\alpha\leq \beta }}\frac{(\beta+\alpha)!}{(\beta-\alpha)!}\\
&\geq \|T_{\gamma}(I-\pi_p) u\|_{L^2}^2\inf_{|\beta|\geq p+1}c|\beta|^{-2|\gamma|}\sum_{\substack{|\alpha|=m\\\alpha\leq \beta }}\frac{(\beta+\alpha)!}{(\beta-\alpha)!}
\label{e:HmEstimate}
\end{align}
where we have used \eqref{e:magic1} in the penultimate step.
Using that
$$
\frac{(\beta+\alpha)!}{\beta!}\geq \frac{\beta!}{(\beta-\alpha)!},\qquad 
\sum_{|\alpha|=m}
\bigg(\frac{m!}{\alpha!}\bigg)^2
\leq 
\bigg(\sum_{|\alpha|=m}
\frac{ m!}{ \alpha!}
\bigg)^2
=d^{2m},
$$
and the Cauchy--Schwarz inequality, we obtain that
\begin{align*}
\sum_{\substack{|\alpha|=m\\\alpha\leq \beta }}\frac{(\beta+\alpha)!}{(\beta-\alpha)!}\geq \sum_{\substack{|\alpha|=m\\ \alpha\leq \beta}}\Big(\frac{\beta!}{(\beta-\alpha)!}\Big)^2&\geq \frac{\Big(\sum_{\substack{|\alpha|=m\\ \alpha\leq \beta}}\frac{m!}{\alpha!}\frac{\beta!}{(\beta-\alpha)!}\Big)^2}{\sum_{\substack{|\alpha|=m\\ \alpha\leq \beta}}\Big(\frac{ m!}{\alpha!}\Big)^2}\geq d^{-2m}\Big(\sum_{\substack{|\alpha|=m\\ \alpha\leq \beta}}\frac{m!}{\alpha!}\frac{\beta!}{(\beta-\alpha)!}\Big)^2.
\end{align*}
We now claim that
\beq\label{e:magic2}
\sum_{\substack{|\alpha|=m\\ \alpha\leq \beta}}m!\frac{\beta!}{(\beta-\alpha)!\alpha!}=\frac{|\beta|!}{(|\beta|-m)!}
\eeq
since both sides count the number of ways to choose $m$ ordered objects from a group of $|\beta|$ elements. Indeed, the right-hand side is the standard formula, and the left-hand side first breaks the $|\beta|$ elements into groups of size $\beta_1,\dots, \beta_d$  and chooses $\alpha_j$ elements from the group of size $\beta_j$. The factor $\beta!/(\beta-\alpha)!\alpha!$ is the number of ways to choose $\alpha_j$ unordered elements from the $j^{th}$ group and the factor $m!$ is the number of orderings of these elements. 

By \eqref{e:magic2},
\begin{align*}
\sum_{\substack{|\alpha|=m\\\alpha\leq \beta }}\frac{(\beta+\alpha)!}{(\beta-\alpha)!}&\geq d^{-2m}\Big(\frac{|\beta|!}{(|\beta|-m)!}\Big)^2\\
&=d^{-2m}|\beta|^{2m}\Big(\prod_{j=0}^{m-1}\Big( 1-\frac{j}{|\beta|}\Big)\Big)^2\\
&\geq d^{-2m}|\beta|^{2m}\Big(\prod_{j=0}^{m-1} \Big(1-\frac{j}{m}\Big)\Big)^2=d^{-2m}|\beta|^{2m}\frac{(m!)^2}{m^{2m}}
\geq d^{-2m}|\beta|^{2m}e^{-2m},
 \end{align*}
where the last inequality follows from Stirling's formula. Hence, the result follows from~\eqref{e:HmEstimate}.
\end{proof}

\begin{lemma}
\label{l:simplexApprox}
    Let $\widehat{K}$ be a $d$-simplex. Then for all $m\geq 0$ there is a map $\mathcal{I}_m^p:H^m(\widehat{K})\to \mathbb{P}_p(\widehat{K})$ such that for all $\ell \geq 0$ there is $C>0$ such that for all $\ell\leq m\leq p+1$, 
    $$
    \|(I-\mathcal{I}_m^p)u\|^2_{H^\ell(\widehat{K})}\leq \frac{C^{2m+2}}{(p+1)^{2m-2\ell}}\sum_{|\alpha|=m}\|\partial^\alpha u\|^2_{L^2(\widehat{K})}.
    $$ 
\end{lemma}
\begin{proof}
Without loss of generality, we assume that $0\in\widehat{K}^{\circ}$ and $\widehat{K}\Subset (-1,1)^d$.
Let 
$$\|u\|_{\tilde{H}^m}:=\|u\|_{L^2}+\sum_{|\alpha|=m}\|\partial^\alpha u\|_{L^2}.
    $$
The plan is to construct an extension $\mathcal{E}_m : H^m(\widehat{K})\to H^m((-1,1)^d)$ satisfying 
   \begin{equation} 
   \label{e:extensionEstimate}
\|\mathcal{E}_m\|_{\tilde{H}^m(\widehat{K})\to \tilde{H}^m((-1,1)^d)}\leq C^m,\qquad 1_{\widehat{K}}\mathcal{E}_m=I.
   \end{equation}
Let 
$$
B(\widehat{K},\delta):=\{ x\in\mathbb{R}^d\,:\, d(x, \widetilde{K})<\delta\}.$$
    By~\cite[Theorem 5]{VIBurenkov_1997}, there are $C_0,\delta>0$ and for all $m\geq 0$ an extension operator $\widetilde{E}_m:H^m(\widehat{K})\to H^m(B(\widehat{K},\delta))$ such that 
    $$
    \|\widetilde{E}_m\|_{\tilde{H}^m\to \tilde{H}^m}\leq C_0^m,\qquad (\widetilde{E}_mu)|_{\widehat{K}}=u.
    $$

    Notice that there is $\delta_1>0$ such that  $B(\widehat{K},\delta)\supset (1+\delta_1)\widehat{K}$. Hence, we may set $E_m:=1_{(1+\delta_1)\widehat{K}}\widetilde{E}_m$ to obtain an extension $E_m$ from $\tilde{H}^m(\widehat{K})$ to $\tilde{H}^m((1+\delta_1)\widehat{K})$ satisfying
    $$
    \|E_m\|_{\tilde{H}^m(\widehat{K})\to H^m((1+\delta_1)\widehat{K})}\leq C_0^m,\qquad (E_mu)|_{\widehat{K}}=u.
    $$
    Define, for $j\geq 1$, $E_m^j:H^m(\widehat{K})\to H^m(\widehat{K})$ by
    $$
    E_m^ju(x):= E_m\Big( (E_m^{j-1}u)\Big)((1+\delta_1)x),\qquad E_m^0:=I.
    $$
    We prove by induction that there is $A>0$ such that for all $m$, 
    \begin{equation} 
    \label{e:extendJ}
    \|E_m^j\|_{\tilde{H}^m(\widehat{K})\to \tilde{H}^m(\widehat{K})}\leq A^{jm},\qquad 1_{(1+\delta_1)^{-j}\widehat{K}}E^j_mu(x)=u((1+\delta_1)^jx).
    \end{equation}

   Suppose that~\eqref{e:extendJ} holds for some $j\geq 0$. Then, we prove it with $j$ replaced by $j+1$.  The second part of~\eqref{e:extendJ} follows easily from the fact that $E_m$ is an extension. For the estimate, observe that
   $$
   \|E_m^{j+1}u\|_{\tilde{H}^m}\leq C_0^m(1+\delta_1)^m\|(E_m^{j}u)\|_{\tilde{H}^m}\leq (1+\delta_1)^m C_0^mA^{jm}\|u\|_{\tilde{H}^m(\widehat{K})}.
   $$
   Therefore, taking $A\geq C_0(1+\delta_1)$ completes the proof.
   Now, fix $J\geq 0$ so that $(-1,1)^d\subset (1+\delta_1)^J\widehat{K}$ and set  $$\mathcal{E}_m(u)(x):=(E_m^Ju)\big((1+\delta_1)^{-J}x\big),\quad x\in(-1,1)^d.$$ Then, there is $C>0$ such that for all $m$ \eqref{e:extensionEstimate} holds.

   Now set $\widetilde{\mathcal{I}}_m^p:=1_{\widehat{K}}\Pi_p\mathcal{E}_m$, where $\Pi_p$ is the $L^2((-1,1)^d)$ orthogonal projector onto polynomials of degree less than or equal to $p$. Then for all $u\in H^m(\widehat{K})$, by the definition of $T_\gamma$ \eqref{e:Talpha},  \eqref{e:polyApprox1}, and \eqref{e:extensionEstimate},
   \begin{align*}
   \sum_{|\gamma|\leq \ell}\|\partial^{\gamma}(I-\widetilde{\mathcal{I}}_m^p)u\|^2_{L^2(\widehat{K})}&\leq C\sum_{|\gamma|\leq \ell} \|T_{\gamma}(I-\Pi_p)\mathcal{E}_mu\|^2_{L^2(-1,1)^d}\\
   &\leq \frac{C^{2m+1}}{(p+1)^{2(m-\ell)}}\sum_{|\alpha|= m} \|\partial^\alpha \mathcal{E}_mu\|^2_{L^2(-1,1)^d}\\
   &\leq \frac{C^{2m+1}}{(p+1)^{2(m-
   \ell)}}\|\mathcal{E}_mu\|^2_{\tilde{H}^m(-1,1)^d}\\
   &\leq \frac{C^{2m+1}}{(p+1)^{2(m-\ell)}}\|u\|^2_{\tilde{H}^m(\widehat{K})}\\
    &\leq \frac{C^{2m+1}}{(p+1)^{2(m-\ell)}}\Big(\|u\|_{L^2(\widehat{K})}+\sum_{|\alpha|=m}\|\partial^\alpha u\|_{L^2(\widehat{K})}\Big)^2.
   \end{align*}

   Now, there is a unique map $\mathcal{Z}_m:H^{m}(\widehat{K})\to \mathbb{P}_{m-1}(\widehat{K})$ such that  
   $$
   \int_{\widehat{K}} \partial^\alpha (u-\mathcal{Z}_mu)=0,\qquad |\alpha|\leq m-1.
   $$
    Set $\mathcal{I}_m^p:=\widetilde{\mathcal{I}}_m^p(u-\mathcal{Z}_mu)+\mathcal{Z}_mu$. 
    (Here, $\mathbb{P}_{-1}=\{0\}$.)
    Then,
   \begin{align*}
   \|(I-\mathcal{I}_m^p)u\|^2_{H^\ell(\widehat{K})}&\leq C\sum_{|\gamma|\leq \ell}\|\partial^{\gamma}(I-\widetilde{\mathcal{I}}_m^p)(u-\mathcal{Z}_mu)\|^2_{L^2(\widehat{K})}\\
   &\leq \frac{C^{2m+1}}{(p+1)^{2(m-\ell)}}\Big(\|u-\mathcal{Z}_mu\|_{L^2(\widehat{K})}+\sum_{|\alpha|=m}\|\partial^\alpha u\|_{L^2(\widehat{K})}\Big)^2.
   \end{align*}
   Now, applying the Poincar\'e inequality $m$ times on $\widehat{K}$, we obtain 
\begin{align*}
   \|(I-\mathcal{I}_m^p)u\|^2_{H^\ell(\widehat{K})}
   &\leq \frac{C^{2m+1}}{(p+1)^{2(m-\ell)}}\Big(\sum_{|\alpha|=m}\|\partial^\alpha u\|_{L^2(\widehat{K})}\Big)^2\\
   &\leq \frac{C^{2m+1}}{(p+1)^{2(m-\ell)}}\sum_{|\alpha|=m}\|\partial^\alpha u\|^2_{L^2(\widehat{K})},
   \end{align*}
   which completes the proof. 

\end{proof}

\subsection{
Proof of Theorem \ref{t:polyApprox}}
\begin{proof}
Let $\widehat{T}$ be the reference element for $\mathcal{T}$, and let $\mathcal{I}_m^p$ the map constructed in Lemma~\ref{l:simplexApprox}.

By \cite[Proposition A.21]{galkowski2025numerical} applied with $\mathcal{J}^p=\mathcal{I}_{m_T}^p$, and Lemma~\ref{l:simplexApprox} with $\ell=0,\dots, d+1$, for each $T\in\mathcal{T}$,  there is a boundary compatible operator 
(in the sense of \cite[Definition A.20]{galkowski2025numerical}) 
$\mathcal{C}_{T}^{p}:H^\ell(\widehat{T})\to \mathbb{P}^{p}$ such that, 
\begin{align}\label{e:run1}
\|(I-\mathcal{C}_{T}^p)v\|_{H_p^1(\widehat{T})}\leq C\|(I-\mathcal{I}_m^p)v\|_{H_p^{\lfloor d/2\rfloor +1}(\widehat{T})}\leq \frac{C^{m_T+1}}{(p+1)^{m_T}}|v|_{H^{m_T}(\widehat{T})},
\end{align}
and 
\beq\label{e:lastDay77}
\mathcal{C}_T^p v|_{\partial\widehat{T}}=\mathcal{C}_{T'}^p v|_{\partial\widehat{T}}
\eeq
(i.e., the different operators agree on the boundary of $\widehat{T}$).

Then, define $\mathcal{C}^p_{\mathcal{T}}:C^0(\Omega)\cap \oplus_{T\in\mathcal{T}}H_k^{m_T}(T)\to \mathcal{P}_{\mathcal{T}}^p$ by
\beq\label{e:CTp}
\mathcal{F}_T^*(\mathcal{C}^p_{\mathcal{T}}u)|_{T}:= \mathcal{C}_{T}^p(\mathcal{F}_T^*u|_{T}).
\eeq
Observe that $\mathcal{C}_{\mathcal{T}}^p$ is triangulation preserving (in the sense of~\cite[Definition A.20]{galkowski2025numerical}
) since $\mathcal{C}_T^p$ are boundary compatible. Furthermore, since $\mathcal{T}$ is a $G^{\sigma}$ simplicial triangulation, by \eqref{e:run1}
and Proposition \ref{prop:projection}, for $\ell=0,1$
\begin{align*}
k^{-2\ell}|(I-\mathcal{C}_{\mathcal{T}}^p)u|^2_{H_{p}^\ell(T)}&\leq C h_{T}^{d-2\ell}k^{-2\ell}\|(I-\mathcal{C}^{p}_T)\mathcal{F}_T^*u\|^2_{H_{p}^\ell(\widehat{T})}\\
&\leq \frac{C^{2m_T+1}}{(p+1)^{2(m_T-\ell)}}k^{-2\ell}h_{T}^{d-2\ell}|A_T^*R_T^*u|^2_{H^{m_T}(\widehat{T})}\\
&\leq \frac{C^{2m_T+1}}{(p+1)^{2(m_T-\ell)}}h_{T}^{2(m_T-\ell)}k^{-2\ell}|R_T^*u|^2_{H^t(A_{T}(\widehat{T}))}\\
&\leq\frac{C^{2m_T+1}(m_T!)^{2\sigma}}{(p+1)^{2(m_T-\ell)}}h_{T}^{2(m_T-\ell)}k^{-2\ell}\frac1{(m_T!)^{2\sigma}}|R_T^*u|^2_{H^{m_T}(A_{T}(\widehat{T}))}\\
    &\leq  C\Big(\frac{h_{T}k}{p+1}\Big)^{2(m_T-\ell)}\Big(\frac{Cm_T^{\sigma}}{k}\Big)^{2m_T}\|u\|^2_{H^{m_T}(T)},
\end{align*}
where, in the last line, we have used Stirling's formula.
Therefore, to establish the result we only need 
to show that $\mathcal{C}_{\mathcal{T}}^pu\in H^1(\Omega)$. 
For this, we refer the reader to the end of the proof of~\cite[Theorem A.11]{galkowski2025numerical} after noticing that $u\in C(\overline{\Omega})$
(with this used in the proof of 
\cite[Theorem A.11]{galkowski2025numerical} just before \cite[Equation A.50]{galkowski2025numerical}). Indeed, 
since $u\in H^1(\Omega)\cap \oplus_{T\in\mathcal{T}}H^{\frac{d}{2}+\e}(T)$, $u|_{T}\in C^{0}(\overline{T})$ for all $T\in \mathcal{T}$ and, on the other hand, since $u\in H^1(\Omega)$, if $F_{1,i_1}(T_1)=F_{1,i_2}(T_2)$, then, $(u|_{T_1})|_{F_{1,i_1}(T_1)}=(u|_{T_2})|_{F_{1,i_2}(T_2)}$. 
\end{proof}

\begin{remark}
    We point out here a small error in the statement of~\cite[Proposition A.21]{galkowski2025numerical}. It is stated there that the operators $Q_{j,i}$ are bounded from $H^{1/2}(F_{j,i})\to \mathbb{P}_{p-1}(F_{j,i})$. However, the proof shows only that they are bounded from $C^0(F_{j,i})\cap H^{1/2}(F_{j,i})\to \mathbb{P}_{p-1}$. 
\end{remark}

\section{Gevrey elliptic regularity case}\label{sec:elliptic_reg_theory}
In this section, we establish elliptic regularity results for the Helmholtz equation when the coefficients and boundaries are Gevrey-$\sigma$ regular. 
Our results apply, more generally, to second-order elliptic semiclassical differential operator (in the sense of Definition~\ref{d:semicl_diff_op}) of the form
\[
P =  \sum_{|\alpha| \leq 2} a_\alpha(x,k)\,(k^{-1} \partial)^\alpha.
\]
To simplify the task of tracking the dependence on $k$ of the Gevrey estimates, we let
$$Lu:=\sum_{|\alpha|=2}a_{\alpha}(x,k)\partial^\alpha. $$
 We then establish estimates for the solution of $Pu =f$, by considering
a boundary value problem 
involving
\begin{equation}\label{eq:L_withrhs}
    L u = k^2(f- k^{-1}b \cdot \nabla u + n u),
\end{equation}
where $b$ and $n$ are spatially-dependent coefficients and are $k$-bounded (see Definition~\ref{d:semicl_diff_op}) for all wavenumbers $k\geq k_0>0$. 

By Definition \ref{d:semicl_diff_op},
$a_\alpha(x)$ are bounded in $k$ and elliptic in the sense that
\beq\label{d:ellipticity}
L^{pr} (\boldsymbol{\xi}) := \sum_{|\alpha|= 2} a_\alpha(x_0) (i \boldsymbol{\xi})^\alpha \neq 0,
\eeq
for all $\boldsymbol{\xi} \in \R^n\setminus\{0\}$ and all $x_0 \in \Omega$. We also assume that $a_{\alpha}$ are Gevrey-$\sigma$ regular for some $\sigma \geq 1$; see Definition \ref{d:Gevrey_class},
We consider boundary conditions of order 1
\beqs
    Tu =g|_{\partial\Omega\setminus\Gamma} \qquad \textrm{on } \partial\Omega\setminus\Gamma,
\eeqs
or of order 0
\beqs
    Du =h|_{\Gamma} \qquad \textrm{on } \Gamma,
\eeqs
where $\Gamma = \emptyset$ or $\Gamma=\partial \Omega$.
The boundary conditions are assumed to satisfy the \emph{complementing condition} (also known as \emph{Shapiro--Lopatinskii conditions}) which ensures the compatibility between the interior differential operator $L$ and the boundary conditions, thus producing unique solvability of the boundary value problem; see, e.g., \cite{agmon1964estimates} for the precise definition.

The boundary conditions are assumed to be given by the restriction to the boundary of functions $g$ and $h$ defined on $\Omega$.

\begin{remark}[Ellipticity for systems]
The definition of ellipticity,~\ref{d:ellipticity}, is given for a single equation ($Lu$ scalar-valued) but extending it to a system of equations is straightforward by choosing the definition of ellipticity of systems in the Petrovsky sense as done in \cite[Definition 1.A]{costabel2010corner}. The Gevrey regularity results given in Theorems~\ref{thm:int_Gev_reg} and~\ref{thm:uttb_Gev_reg} generalize in a natural manner to the case of elliptic $N\times N$ systems of equations.
\end{remark}

\subsection{Gevrey elliptic regularity results}
We provide here the main results on the Gevrey regularity for elliptic problems. The first one deals with the interior regularity, the second extends the Gevrey regularity of the solutions up to the boundary. 
\begin{theorem}[Interior Gevrey regularity]\label{thm:int_Gev_reg}
Let $\Omega\subset \R^n$ be open. 
Let $\sigma\geq 1, k_0>0$, and $L$ be a second-order elliptic partial differential operator with Gevrey-$\sigma$ coefficients. Let $b$ and $n$ be Gevrey-$\sigma$ on $\Omega$. Then for any bounded subdomains $\Omega_1$, $\Omega_2$ such that $\overline{\Omega_1} \subset \Omega_2 \subset \Omega$, there exists a constant $A\geq 1$ such that for all $\kappa\geq 0$, $k\geq k_0$, $f\in H^\kappa(\Omega_2)$,  and $u\in H^2(\Omega_2)$ a solution of \eqref{eq:L_withrhs}, 
\begin{equation}\label{est:main_estimate}
        \begin{aligned}
        | u|_{\kappa+2;\Omega_1} 
        \leq A^{\kappa +2}\max \left(
        k, (\kappa+2)^\sigma\right)^{\kappa+2} \bigg(& \max_{0\leq \ell\leq \kappa}  A^{-\ell-1} \frac{1}{\max (k, (\ell+2)^{\sigma})^{\ell}}  
        |f|_{\ell; \Omega_2} \\
        &
        \quad+ k^{-1} |u|_{1;\Omega_2} +|u|_{0;\Omega_2}\bigg).
        \end{aligned}
\end{equation}
In particular, 
if $f\in G^\sigma(\Omega_2)$ then 
the solution $u \in G^\sigma(\Omega_1)$. 
\end{theorem}

\begin{theorem}[Up-to-the-boundary Gevrey regularity]\label{thm:uttb_Gev_reg}
Let $\sigma\geq 1, k_0>0$. Let $\Omega$ be a Gevrey-$\sigma$ regular domain in $\R^n$. Let $L,\,T,\,D$ be, respectively, second-, first- and zero-order partial differential operator with Gevrey-$\sigma$ coefficients on $\overline \Omega$. Let $\{L,T,D\}$ be elliptic on $\overline \Omega$. Then for any open bounded subdomains $\Omega_1 = \mathcal{U}_1 \cap \Omega$, $\Omega_2 = \mathcal{U}_2 \cap \Omega $ defined through $\mathcal{U}_1,\, \mathcal{U}_2$ bounded opens of $\R^n$ verifying $\overline{\mathcal{U}}_1\subset \mathcal{U}_2$, there exists a constant $A\geq 1$ such that for all $\kappa\geq 0$, $k\geq k_0$,
    $f\in H^\kappa(\Omega_2)$, $h \in H^{\kappa+2}(\Omega_2)$ and $g \in H^{\kappa+1}(\Omega_2)$,
    and $u\in H^2(\Omega_2)$ 
    a solution of \eqref{eq:L_withrhs},
\begin{equation}\label{est:main_estimate_bdry}
       \begin{aligned}
        |u|_{\kappa+2;\Omega_1} 
        \leq A^{\kappa +2}\max \left(
        k, (\kappa+2)^\sigma\right)^{\kappa+2} \bigg(&  \max_{0\leq \ell\leq \kappa}
        A^{-\ell-1} \frac{1}{\max (k, (\ell+2)^{\sigma})^{\ell}}
        |f|_{\ell; \Omega_2} \\&+
         \max_{0\leq \ell\leq \kappa+2}  
         A^{-\ell+1} \frac{1}{\max (k, \ell^{\sigma})^{\ell}}
         |h|_{\ell;\Omega_2}  
        \\ &
        +   \max_{0\leq \ell\leq \kappa+1} A^{-\ell} \frac{1}{\max (k, (\ell+1)^{\sigma})^{\ell+1}}
        |g|_{\ell;\Omega_2} \\
        &+ k^{-1} |u|_{1;\Omega_2} +|u|_{0;\Omega_2} \bigg)
        \end{aligned}
\end{equation}
where $\Gamma_2 := \partial\Omega_2\cap \partial\Omega$.
\end{theorem}

When $\sigma=1$, Theorems \ref{thm:int_Gev_reg} and \ref{thm:uttb_Gev_reg} recovers the anlytic regularity case~\cite[Chapter 5]{melenk2004hp}; see also \cite[Theorems 1.7.1 and 2.7.1]{costabel2010corner}. The proofs of Theorems~\ref{thm:int_Gev_reg} and \ref{thm:uttb_Gev_reg} are the subject of Sections~\ref{sec:int_reg} and \ref{sec:uttb_reg} respectively.

\begin{remark}[Boundary condition defined on $\partial \Omega$ versus $\Omega$]\label{r:boundary_cond}
    We remark here that formulating the result for boundary functions $g$ and $h$ only defined on the boundary or on the whole domain is equivalent thanks to the existence of continuous extension operators for Gevrey-$\sigma$ functions $(\sigma>1)$; see, e.g.,~\cite{bruna1980extension}. For the analytic case ($\sigma =1$), the existence of extensions is generally not true. However, we expect the analytic result holds for boundary functions only defined on the boundary $\partial \Omega$ in light of the result of \cite[Theorem 2.7.1]{costabel2010corner}.
\end{remark}

\subsection{Notation for elliptic regularity results}
In general, we follow here the notation used in \cite{costabel2010corner} with modifications, where necessary, to accommodate the Gevrey regularity setting. 

For $\Omega \subset \R^n$, we denote
\[
\|u\|_{\Omega} : = \|u\|_{L^2(\Omega)}.
\]
The norms and semi-norms associated with the Sobolev space $H^m(\Omega)$ are given by
\[
\|u\|_{m,\Omega} := \left(\sum_{|\alpha| \leq m} \|\freqk{\partial}^\alpha u\|_{\Omega}^2\right)^{\frac12}, \qquad |u|_{m,\Omega} := \left(\sum_{|\alpha| = m} \|\freqk{\partial}^\alpha u\|_{\Omega}^2\right)^{\frac12}.
\]
We denote by $H_0^m(\Omega)$ the closure of the space of smooth functions with compact support in $\Omega$, $C_0^\infty(\Omega)$, with respect to the norm $\|\cdot\|_{m,\Omega}$.

We introduce weighted semi-norms of Sobolev--Morrey type adapted to the present Gevrey context, i.e., they depend on the parameter $\sigma$ in the following way:
\begin{equation}\label{d:Sobolev-Morrey}
\begin{aligned}
    [|u|]_{0;B_R} &:= \|u\|_{B_R},  \\
    [|u|]_{\ell;B_R} &:= \max_{0<\rho \leq \frac{R}{2\ell}} \max_{|\delta| = \ell} \rho^{ \sigma \ell} \|\freqk{\partial}^\delta u\|_{B_{R-\ell\rho}}\qquad \textrm{for } \ell>0, 
     \\
    \rho_*^2[|f|]_{\ell;B_R} &:= \max_{0<\rho \leq \frac{R}{2(\ell+1)}} \max_{|\delta| = \ell} \rho^{ \sigma (2+\ell)} \|\freqk{\partial}^\delta f\|_{B_{R-(\ell+1)\rho}},
\end{aligned}
\end{equation}
for $\ell \in \mathbb{N}$. Similarly to \cite[Notation 1.6.1]{costabel2010corner}, $\rho_*^2$ is not a real number but a symbolic notation.

\subsection{General strategy of the proofs}
We discuss here the general strategy towards proving Theorems~\ref{thm:int_Gev_reg} and \ref{thm:uttb_Gev_reg} which also justifies the notation introduced in the preceding section. First, we discuss the proof of the interior regularity result for a ball of radius $R/2$.

The basic building block is the standard elliptic regularity estimate
\begin{equation}\label{ineq:basic_ell_est_nonscaled}
\sum_{|\alpha| = 2} \|\partial ^\alpha u \|_{B_{R_*}} \leq A_{-1} \left( \|L u \|_{B_{R_*}} + \sum_{|\alpha|\leq 1}  \|\partial^\alpha u \|_{B_{R_*}} \right)
\end{equation}
which holds for some $A_{-1}>0$ for all $u \in H^2_0(B_{R_*})$ ($R_*>0$) and is a consequence of the ellipticity of the operator $L$. Let $k_0>0$, then estimate~\eqref{ineq:basic_ell_est_nonscaled} generalizes to frequency-scaled norms as
\begin{equation}\label{ineq:basic_ell_est}
\sum_{|\alpha| = 2} \|\freqk{\partial}^\alpha u \|_{B_{R_*}} \leq A_0 \left( \|L u \|_{B_{R_*}} + \sum_{|\alpha|\leq 1}  \|\freqk{\partial}^\alpha u \|_{B_{R_*}} \right)
\end{equation}
for all $k\geq k_0$. Above $A_0$ depends on $k_0$.

A first step is to derive a version of estimate \eqref{ineq:basic_ell_est} which does not rely on the vanishing of the solution. This is obtained by first applying a cut-off function (whose definition depends on $\rho$ in \eqref{d:Sobolev-Morrey}) to a function $u \in H^2(B_{R_*})$, this is the object of Lemma~\ref{lem:truncated_basic_ell}. This process leads to an inequality involving nested balls $B_{R-\varrho}$ and $B_{R}$ (with $R\in (0,R_*]$), partly justifying the form of the introduced Sobolev--Morrey semi-norms.

A second step is to successively apply Lemma~\ref{lem:truncated_basic_ell} to the derivatives of the solution $\partial^\beta u$ to improve its regularity. This requires estimating terms of the form $L\partial^\beta u$ on the right-hand side of \eqref{ineq:basic_ell_est}, which we achieve by bounding the difference $[L, \partial^\beta] u $ with
\[
[L, \partial^\beta] := L\partial^\beta - \partial^\beta L
\]
which has the advantage of being of order at most $|\beta|-1$.
The commutator $[L, \partial^\beta]$ measures in some sense how the space-dependent case differs from the constant-coefficient one. Indeed, when the coefficients of the differential operator $L$ are not constant, the commutator is nonzero and we estimate it through Lemma~\ref{lemma::commutator_estimate}. 

Once these ingredients are built, we recursively build the Gevrey estimate in Proposition~\ref{prop:Gevrey_fundamental_est} by iterating over the derivative order $\beta$, using the Sobolev--Morrey norms to more easily track the nesting of the domains (nested balls Figure~\ref{fig:balls}). 
The other key point about these norms is that 
\begin{equation}\label{e:whyNorms1}
[| u|]_{\ell;B_{R}} \geq \left(\frac{R}{2\ell}\right)^{\ell\sigma } | u|_{\ell;B_{R/2}},
\end{equation}
\begin{equation}\label{e:whyNorms2}
\sum\limits_{\ell=0}^{1} [|u|]_{\ell;B_{R}} \leq 2 \max\left\{\left(\frac{R}2\right)^\sigma, 1\right\} \|u\|_{1;B_{R}}
\end{equation}
and 
\begin{equation} \label{e:whyNorms3}    
            \rho_*^2 [|Lu|]_{\ell; B_R} \leq \left(\frac{R}{2(\ell+1)}\right)^{(\ell+2)\sigma} |Lu|_{\ell;B_{R}},
    \end{equation}
with \eqref{e:whyNorms1} giving us the left-hand side of \eqref{est:main_estimate} and \eqref{e:whyNorms2} and \eqref{e:whyNorms3} giving us the right-hand side.
Finally, to prove Theorems~\ref{thm:int_Gev_reg} and \ref{thm:uttb_Gev_reg} for general Gevrey domains, we use Gevrey-regular mappings from neighbourhoods of $x_0 \in \overline\Omega_1$ contained in $\Omega_2$ to $B_R$.

The proof of \ref{thm:uttb_Gev_reg} follows, in spirit, similar steps with special attention paid to the behavior near the boundary. A flattening of the boundary is performed such that nested half-balls are used; see Fig.~\ref{fig:balls}. Higher derivatives are derived, in a first step, in the tangential direction of the boundary; see Proposition~\ref{prop:almost_tangetial}. This is because the tangential derivatives compose more naturally with the boundary operators (as opposed to differentiation in the normal direction which requires knowledge of the solution on the interior of the domain). A second step is to assume that the boundary $x_n =0$ is non-characteristic to obtain isotropic higher-order estimates (i.e., with normal derivatives of order higher than 2) in Proposition~\ref{prop:isotropic}. 

\subsection{Interior regularity}\label{sec:int_reg}

The next lemma generalizes the commutator estimates \cite[Lemmas 1.6.2]{costabel2010corner} to the Gevrey-coefficient case.

\begin{lemma}[Commutator estimate]\label{lemma::commutator_estimate}
Let $M\geq 0$, $\sigma \geq 1$, and $k_0,\, R_*>0$.
There exists two constants $c$ and $K$ such that for all radii $R\in (0,R_*]$, all commutators $[a(x), \freqk{\partial}^\beta]$ with Gevrey coefficients $a(x) \in G^\sigma(B_{R})$ with constant $M$, and all multi-indices $\alpha$, $\beta$ with $|\alpha|+|\beta| = t \geq 2$, for all $\rho \in \left(0,\frac{R}{2t}\right]$ 
\begin{equation}
    \rho^{\sigma t} \|[a(x), \freqk{\partial}^\beta] \freqk{\partial}^\alpha u\|_{B_{R-(t-1)\rho}} \leq c \sum_{d=0}^{t-1} \left( KR\right)^{t-d} [| u|]_{d;B_{R}}.
\end{equation}    
Above, $\alpha$ is a multi-index with $|\alpha|\leq2$. 
\end{lemma}

\begin{proof}
Our proof strategy is adapted from that of \cite[pp. 48--49]{costabel2010corner}. Let \[N: = \rho^{\sigma t} \|[a(x), \freqk{\partial}^\beta] \freqk{\partial}^\alpha u\|_{B_{R-(t-1)\rho}}. \]
    We use the Leibniz formula
    \[\partial^\beta(a \partial^\alpha u) = \sum_{\gamma\leq \beta } \frac{\beta!}{\gamma!(\beta-\gamma)!} \partial^\gamma a \partial^{\beta - \gamma} \partial^\alpha u\]
    together with the combinatorial inequality
    \[
    \frac{\beta!}{\gamma!(\beta-\gamma)!} \leq \frac{|\beta|!}{|\gamma|!(|\beta|-|\gamma|)!}
    \]
    and find that since the contribution of $\gamma = 0$ is absent from $N$:
    \[
    \begin{aligned}
    N \leq &  \rho^{\sigma t} \sum_{1\leq |\gamma|,\, \gamma\leq \beta } M^{|\gamma| + 1}  \frac{|\beta|!\,|\gamma|!^{\sigma-1}}{(|\beta|-|\gamma|)!}  \|\freqk{\partial}^{\beta - \gamma} \freqk{\partial}^\alpha u\|_{B_{R-(t-1)\rho}} \\
    \leq &  \rho^{\sigma t} \sum_{g =1 }^{|\beta|}\sum_{|\gamma|=g,\, \gamma\leq \beta } M^{g + 1}  \frac{|\beta|!\,g!^{\sigma-1}}{(|\beta|-g)!}\max_{|\delta|=|\beta -\gamma + \alpha|}\| \freqk{\partial}^\delta u\|_{B_{R-(t-1)\rho}}.
    \end{aligned}
    \]
    By induction over $n$, we check that $\sum_{|\gamma| =g} 1 \leq (g+1)^{n-1}$. This implies that 
     \[
    \begin{aligned}
    N \leq & \rho^{\sigma t} \sum_{g =1 }^{t-|\alpha|} (g+1)^{n-1} M^{g + 1}  \frac{(t-|\alpha|)!\,g!^{\-1}}{(t-|\alpha|-g)!}\max_{|\delta|=|\beta -\gamma + \alpha|}\| \freqk{\partial}^\delta u\|_{B_{R-(t-1)\rho}}
    \\
    \leq  & \rho^{\sigma t} \sum_{d = |\alpha| }^{t-1} (t-d+1)^{n-1} M^{t-d + 1}  \frac{(t-|\alpha|)!\,(t-d)!^{\sigma-1}}{(d-|\alpha|)!}\max_{|\delta|=|\beta -\gamma + \alpha|}\|  \freqk{\partial}^\delta u\|_{B_{R-(t-1)\rho}}
    \end{aligned}
    \]
    where in the last line, we operated the variable change $d = t-g $.
\\ A contribution for $d=0$ appears only if $|\alpha|=0$. We denote it by $N_0$:
    \[N_0 = \rho^{\sigma t} (t+1)^{n-1} M^{t+1} t!^\sigma \|u\|_{B_{R-(t-1)\rho}}. \]
    
    To deal with the $\sigma$ term $(t-d)!^{\sigma-1}$, which arose from the assumption that the coefficients are Gevrey, we use the inequalities
    \[\frac{(t-|\alpha|)!}{(d-|\alpha|)!} \leq \frac{t!}{d!}, \qquad (t-d)! \leq \frac{t!}{d!},\]
    for all $1 \leq d \leq t-1$ 
    and $0\leq |\alpha| \leq t-1$. We denote the remaining part $N_1 := N-N_0$ and inject the binomial inequality for $(g!)^{\sigma-1} =((t-d)!)^{\sigma-1} $ to obtain
    \[ N_1 \leq \rho^{\sigma t} \sum_{d=1}^{t-1} (t-d+1)^{n-1} M^{t-d+1} \left(\frac{t!}{d!}\right)^\sigma \max_{|\delta| =d}\|\freqk{\partial}^\delta u\|_{B_{R-(t-1)\rho}}. \]
    
    For the purposes of our proof, we want to associate a derivative $\freqk{\partial}^\delta u$ of length $d$ with $\rho^{\sigma d}$. To this end, we write
    \[ N_1 \leq  \sum_{d=1}^{t-1} (t-d+1)^{n-1} \rho^{\sigma(t-d)} M^{t-d+1} \left(\frac{t!}{d!}\right)^\sigma \left(\rho^{\sigma d} \max_{|\delta| =d}\|\freqk{\partial}^\delta u\|_{B_{R-(t-1)\rho}} \right). \]

    Using Stirling's formula, there exists a constant $c_s$, such that for all $t,\,  d>0$:
    \[\left(\frac{t!}{d!}\right)^\sigma \leq  c_s e^{(d-t)\sigma} \left(\frac{t^{(t+1/2)}}{d^{d+1/2}}\right)^\sigma\]

    We now use \[\rho \leq \frac{R}{2t}, \]
    to infer  
    \[N_1 \leq c_s \sum_{d=1}^{t-1} (t-d+1)^{n-1} M \left( \frac{RM}{2e^\sigma}\right)^{t-d} \sqrt{\frac{t}{d}}^\sigma \left(\frac{t^{d}}{d^{d}}\right)^\sigma  \left(\rho^{\sigma d} \max_{|\delta| =d}\|\freqk{\partial}^\delta u\|_{B_{R-(t-1)\rho}} \right).\]

    Similarly to \cite[p. 49]{costabel2010corner}, we define the new distance
    \(\rho_d' := (t-1) \rho/d \). Then $B_{R-(t-1)\rho} = B_{R-d\rho'_d}$.

    We also note the following inequalities 
    \[ \left(\frac{t}{t-1}\right)^{d\sigma}\leq e^\sigma \qquad \mathrm{and}\qquad \sqrt{\frac{t}{d}} \leq t - d +1 \qquad \mathrm{for}\ 1\leq d \leq t-1,\]
    which we use to estimate:
    \begin{equation}
        \label{N1_estimate} 
    N_1 \leq c_0 e^\sigma \sum_{d=1}^{t-1} (t-d+1)^{n-1+\sigma} M \left( \frac{RM}{2e^\sigma}\right)^{t-d}  \left(\rho_d'^{\sigma d} \max_{|\delta| =d}\|\freqk{\partial}^\delta u\|_{B_{R-d\rho'_d}} \right).
    \end{equation}

Concerning $N_0$, following similar calculations as above, we obtain the following inequality:
\[N_0 \leq c_0 M (t+1)^{n-1+\sigma} \left(\frac{RM}{2e^\sigma}\right)^t \|u\|_{B_{R-(t-1)\rho}}\]

Instead of keeping track of $\rho_d'$ in $N_1$, we use the Sobolev-Morrey seminorms introduced in~\eqref{d:Sobolev-Morrey} to write
\begin{equation}
\begin{aligned}
    \rho^{\sigma t} \|[a(x), \partial^\beta] \partial^\alpha u\|_{B_{R-(t-1)\rho}} &\leq c_0 M \sum_{d=0}^{t-1} (t-d+1)^{n-1+\sigma} \left( \frac{RM}{2e^\sigma}\right)^{t-d} [| u|]_{d;B_{R}}\\
    &\leq c_1 \sum_{d=0}^{t-1} \left( KR\right)^{t-d} [| u|]_{d;B_{R}}.
    \end{aligned}
\end{equation}    
with positive constants $c_1$ and $K$, independent of $t$ and $R$ for $R \leq R_*$. In the last line, we have used that polynomial growth is dominated by exponential growth, specifically that there exist constants $c_2$, $J$, such that for all values of $\kappa + 3 - d$
    \[
    (\kappa +3 -d )^{n-1+\sigma} \leq c_2 J^{\kappa + 3 -d} \leq c_2J^2 J^{\kappa+ 1-d}.
    \]
    Consequently we can express $c_1=c_0Mc_2J^2(\frac{R_*M}{2e^\sigma})$, and $K = \frac{JM}{2e^\sigma}.$
    \end{proof}

Next, we adapt \cite[Lemma 1.4.1]{costabel2010corner} to the Gevrey case, i.e., with the desired $\rho^\sigma$ scaling. Lemma~\ref{lem:truncated_basic_ell} will be the central building block of the desired Gevrey estimates.
\begin{lemma}\label{lem:truncated_basic_ell}
    Let $R^*>0$ and $\sigma \geq 1$. Let $L$ be a second-order elliptic operator with $C^0(\overline{\Omega})$ coefficients and assume that \eqref{ineq:basic_ell_est} holds for all functions in $H_0^2(B_{R_*})$. Then there exists $A_1>0$ such that for any $u \in H^2(B_R)$ with $R \leq R_*$, for all $\rho \in (0, R/2)$
    \begin{equation}\label{ineq::basic_elliptic}
        \sum_{|\alpha| \leq 2} \rho^{\sigma|\alpha|} \|\freqk{\partial}^\alpha u\|_{B_{R-|\alpha|\rho}} \leq A_1 \left(\rho^{2\sigma} \| L u\|_{B_{R-\rho}}  +  \sum_{|\alpha| \leq 1} \rho^{\sigma |\alpha|} \|\freqk{\partial}^\alpha u\|_{B_{R-|\alpha|\rho}} \right).
    \end{equation}
\end{lemma}
\begin{proof}
    The idea of the proof is to use a cut-off function $\chi_{R,\rho}$, along the lines of \cite[Eq. (1.40)]{costabel2010corner}. Indeed we define 
    \[\label{truncation_function}
    \chi_{R,\rho}: x \mapsto \chi\left( \frac{|x|-R+\rho}{\rho}\right),
    \]
    where $\chi \in C^\infty(\R)$ is a smooth function which verifies
    \[
    \chi\equiv \left\{ \begin{array}{cl}
         1& \quad \mathrm{on } \qquad (-\infty,0] \\
         0& \quad \mathrm{on } \qquad [1,\infty).
    \end{array}
    \right.
    \] 
    Note that $\chi_{R,\rho}$ equals 1 inside $B_{R-\rho}$ and 0 outside $B_R$. A basic computation shows that for all $R_*>0$ and $\sigma \geq 1$, there exists a constant $D>0$ such that for all $R \in (0,R_*]$, $\rho \in (0,R)$, and $\alpha$ with $|\alpha|\leq 2$
    \begin{equation}\label{ineq:cut_off_Gev_est}
    |\partial^\alpha \chi_{R,\rho}| \leq D \rho^{-\sigma |\alpha|}.
    \end{equation}
    The rest of the proof follows then by applying \eqref{ineq:basic_ell_est} to $\chi_{R,\rho} u$ in a similar manner as in \cite[Lemma 1.4.1]{costabel2010corner}, the details are omitted.
\end{proof}

We now prove an intermediate result towards Proposition~\ref{prop:Gevrey_fundamental_est}, \emph{cf.} \cite[Lemma 5.5.12]{melenk2004hp}
\begin{lemma}\label{lemma:elliptic_reg_kappa}
    Let $R_*>0$. Assume the basic elliptic regularity estimate~\ref{ineq:basic_ell_est} holds.
Then there exists a constant $A_2>1$ such that for all $R \in (0,R_*]$ and for all $\kappa\in \mathbb{N}$,
     \begin{equation}\label{ineq:commutator_4}
        \begin{aligned}
        [| u|]_{\kappa+2;B_{R}} & \leq A_2 \Bigg(\rho_*^{2}  [|Lu|]_{\kappa; B_R}  + \sum_{d=0}^{\kappa} \left(KR\right)^{\kappa-d}     [|u|]_{d;B_{R}} + [| u|]_{\kappa+1;B_{R}} \Bigg).
        \end{aligned}
    \end{equation}    
\end{lemma}
\begin{proof} 
    The proof relies on recursively applying Lemmas~ \ref{lemma::commutator_estimate} and \ref{lem:truncated_basic_ell}. 
    First, let $\beta$ be a multi-index of length $|\beta| = \kappa$. We use  \eqref{ineq::basic_elliptic} with $\freqk{\partial}^\beta u$ and $R = R-|\beta|\rho$, to obtain:
     \begin{equation}\label{ineq::commutator_elliptic}
        \begin{aligned}
        \sum_{|\alpha| = 2} \rho^{\sigma |\alpha|} \|\freqk{\partial}^{\alpha+\beta} u\|_{B_{R-|\beta|\rho-|\alpha| \rho}} \leq A_1 \Bigg(\rho^{2\sigma} & \| L \freqk{\partial}^\beta u\|_{B_{R-|\beta|\rho-\rho}} \\ &+  \sum_{|\alpha| \leq 1} \rho^{\sigma |\alpha|} \|\freqk{\partial}^{\alpha+\beta} u\|_{B_{R-|\beta|\rho -  |\alpha|\rho}} \Bigg),
        \end{aligned}
    \end{equation}
    subsequently, introducing the commutator
        \begin{equation}
        \begin{aligned}
        \sum_{|\alpha| = 2} \rho^{\sigma |\alpha|} \|\freqk{\partial}^{\alpha+\beta}u\|_{B_{R-(|\beta|+|\alpha|) \rho}} \leq A_1 \Bigg(&\rho^{2\sigma} \left(\|\freqk{\partial}^\beta L u\|_{B_{R-|\beta|\rho-\rho}} +\|[\freqk{\partial}^\beta, L] u\|_{B_{R-|\beta|\rho-\rho}} \right) \\& +  \sum_{|\alpha| \leq 1} \rho^{\sigma|\alpha|} \|\freqk{\partial}^{\alpha+\beta} u\|_{B_{R-(|\beta|+  |\alpha|)\rho}} \Bigg).
        \end{aligned}
    \end{equation}

    We multiply by the scaling $\rho^{\sigma|\beta|}$ and sum over $|\beta| = \kappa$ to obtain 
     \begin{equation}\label{ineq:commutator_3}
        \begin{aligned}
        \sum_{|\alpha| = \kappa+2} \rho^{\sigma |\alpha|} \|\freqk{\partial}^\alpha u\|_{B_{R-|\alpha| \rho}} \leq A_1 \Bigg(&\rho^{(\kappa+2)\sigma} \sum_{|\alpha| =\kappa}\left(\|\freqk{\partial}^\alpha L u\|_{B_{R-(\kappa+1)\rho}} +\|[\freqk{\partial}^\alpha, L] u\|_{B_{R-(\kappa+1)\rho}} \right)\\& +  \sum_{|\alpha| = \kappa}^{\kappa+1} \rho^{\sigma|\alpha|} \|\freqk{\partial}^\alpha u\|_{B_{R-|\alpha|\rho}} \Bigg).
        \end{aligned}
    \end{equation}    

    Now using Lemma~\ref{lemma::commutator_estimate}, with $|\beta| = \kappa$ and $|\alpha| \leq 2$ (thus $t \leq \kappa+ 2$), we obtain
\begin{equation}
\begin{aligned} \label{ineq:applying_lem_com_plus}
        \rho^{(\kappa+2)\sigma} \|[\freqk{\partial}^\alpha, L] u\|_{B_{R-(\kappa+1)\rho}}  &\leq 
        c_0        \rho^{\sigma b} \|[\freqk{\partial}^\alpha, L] u\|_{B_{R-(b-1)\rho}} \\
        & \leq c_0c_1  \sum_{d=0}^{b-1} \left( KR\right)^{b-d} [| u|]_{d;B_{R}} \\
     &\leq c_2 \sum_{d=0}^{\kappa+1} \left(KR\right)^{\kappa+2-d}[| u|]_{d;B_{R}}, \\
\end{aligned}
    \end{equation}
    
Putting back in \eqref{ineq:commutator_3}, taking the maximum over $\rho$, and using the previously introduced notation~\eqref{d:Sobolev-Morrey} yields
     \begin{equation}
        \begin{aligned}
        [|u|]_{\kappa+2;B_{R}} & \leq A_1' \Bigg(\rho_*^{2}  [|Lu|]_{\kappa; B_R}  +  c_2\sum_{d=0}^{\kappa+1} \left(KR\right)^{\kappa+2-d}     [| u|]_{d;B_{R}} + \sum_{d = \kappa}^{\kappa+1} [| u|]_{d;B_{R}} \Bigg)\\
        &
        \leq A_2 \Bigg(\rho_*^{2}  [|Lu|]_{\kappa; B_R}  + \sum_{d=0}^{\kappa} \left(KR\right)^{\kappa-d}     [|u|]_{d;B_{R}} + [| u|]_{\kappa+1;B_{R}} \Bigg).
        \end{aligned}
    \end{equation}    
\end{proof}

We recall at this stage that we are interested in $k$-explicit high regularity bounds for the equation: 
\[
L u = k^2(f- k^{-1}b \cdot \nabla u + n u),
\]
which prompts us to bound the quantities $\rho_*^{2}  [|b\cdot \nabla u |]_{\kappa; B_R}$ and $\rho_*^{2}  [|n u|]_{\kappa; B_R} $; see \eqref{ineq:commutator_4}. \emph{Cf.} \cite[Lemma 5.5.13]{melenk2004hp}.
\begin{lemma}\label{lemma:rhs_k_scaling_partial_estimate}
    Let $\sigma \geq 1$ and $R_*>0$
    Let $b$ and $n$ be of Gevrey-$\sigma$ regularity with a Gevrey constant $M$, then there exist constants $c_b$ and $c_n$ such that for all $R \leq R_*$, all $\rho \in (0,R/2)$  
    \begin{equation}
    \begin{aligned}
    \rho_*^{2}  [|b\cdot \nabla u |]_{\kappa; B_R} &\leq  c_b \frac{1}{(\kappa+1)^{2\sigma}} \sum_{d=0}^{\kappa} \left( KR\right)^{\kappa-d} [| \nabla u|]_{d;B_{R}}\\
    \rho_*^{2}  [|n u |]_{\kappa; B_R}  &\leq c_n \frac{1}{(\kappa+1)^{2\sigma}} \sum_{d=0}^{\kappa} \left( KR\right)^{\kappa-d} [| u|]_{d;B_{R}}
    \end{aligned}
    \end{equation}
\end{lemma}
\begin{proof} The proof relies on the commutator estimate from Lemma~\ref{lemma::commutator_estimate}. By definition,
    \begin{equation}\label{eq:cu_commutator}
    \begin{aligned}\rho_*^{2}  [|n u |]_{\kappa; B_R} &= 
    \max_{0<\rho \leq \frac{R}{2(\kappa+1)}} \max_{|\delta| = \kappa} \rho^{ \sigma (2+\kappa)} \|\partial^\delta (nu)\|_{B_{R-(\kappa+1)\rho}} \\
    &=\max_{0<\rho \leq \frac{R}{2(\kappa+1)}} \max_{|\delta| = \kappa} \rho^{ \sigma (2+\kappa)} \left(\|n \partial^\delta u\|_{B_{R-(\kappa+1)\rho}}+ \|[n, \partial^\delta] u\|_{B_{R-(\kappa+1)\rho}} \right),
    \end{aligned}
    \end{equation}

    By Lemma~\ref{lemma::commutator_estimate},  
    \begin{equation}
    \begin{aligned}
        \max_{0<\rho \leq \frac{R}{2(\kappa+1)}} & \max_{|\delta| = \kappa} \rho^{ \sigma (2+\kappa)} \|[n, \partial^\delta] u\|_{B_{R-(\kappa+1)\rho}} 
         \\ & \leq  \max_{0<\rho \leq \frac{R}{2(\kappa+1)}} \rho^{2\sigma} \max_{|\delta| = \kappa} \rho^{ \sigma \kappa} \|[n, \partial^\delta] u\|_{B_{R-\kappa\rho}} 
         \\ &\leq C_{R_*,\sigma} \frac{1}{(\kappa+1)^{2\sigma}} c_1 \sum_{d=0}^{\kappa-1} \left( KR\right)^{\kappa-d} [| u|]_{d;B_{R}}
    \end{aligned}
    \end{equation}
   Taking into account the contribution of $\delta =\kappa$ in \eqref{eq:cu_commutator}, there exist a constant $c_n$ such that
    \[
    \begin{aligned}\rho_*^{2}  [|n u |]_{\kappa; B_R} &\leq c_n \frac{1}{(\kappa+1)^{2\sigma}} \sum_{d=0}^{\kappa}  \left( KR\right)^{\kappa-d} [| u|]_{d;B_{R}}.
    \end{aligned}
    \]
    
    The bound on $b\cdot \nabla u$ follows similarly by using \eqref{N1_estimate} instead of the main estimate of Lemma~\ref{lemma::commutator_estimate} to reflect that we do not need the 0-order term. 
\end{proof}
\begin{proposition}\label{prop:Gevrey_fundamental_est} Let $R_*,k_0>0$. Assume the basic elliptic regularity estimate~\eqref{ineq:basic_ell_est} holds.
Then there exists a constant $B>1$ such that for all $R \in (0,R_*]$, all $k\geq k_0$, and all $\kappa\geq -2$,
\begin{equation}\label{ineq:desired_est}
       \begin{aligned}
        [| u|]_{\kappa+2;B_{R}} 
        \leq  B^{\kappa +2}\max \left\{
        \frac{k}{(\kappa+2)^\sigma},1\right\}^{\kappa+2} \Bigg( &\, \max_{0\leq \ell\leq \kappa}  \Gamma_{B,k}^{\ell} (\ell+2)^{(\ell +2)\sigma}\ \rho^2_*[|f|]_{\ell; B_R} \\
        &+ k^{-1} [|u|]_{1;B_{R}} +[|u|]_{0;B_{R}}\Bigg),
        \end{aligned}
\end{equation} 
    with $\Gamma_{B,k}^{\ell} = B^{-\ell-1} \left(\max \left\{k,(\ell+2)^\sigma\right\}\right)^{-\ell} .$
\end{proposition}
\begin{proof}
Verifying the claim for $\kappa \in\{-2,-1\}$ is straightforward. We proceed by induction on $\kappa \geq 0$, assuming \eqref{ineq:desired_est} holds for all $-1 \leq \kappa'<\kappa$. By Lemma~\ref{lemma:elliptic_reg_kappa}
\begin{equation}
[| u|]_{\kappa+2;B_{R}}  \leq A_2 \Bigg(\rho_*^{2}  [| k^2(f- k^{-1}b \cdot \nabla u + n u)|]_{\kappa; B_R} + \sum_{d=0}^{\kappa} \left(KR\right)^{\kappa-d}     [|u|]_{d;B_{R}} + [| u|]_{\kappa+1;B_{R}} \Bigg)
\end{equation}

For $d \in \{0,\ldots,\kappa+1\}$
\begin{equation}
\begin{aligned}
[| u|]_{d;B_{R}} \leq &  \frac{1}{(\kappa+2)^{\sigma(\kappa+2)}} \frac{(\kappa+2)^{\sigma(\kappa+2)}}{d^{\sigma d} } B^{d} \max(k, d^{\sigma })^d C_{u,d-2} \\
\leq & \frac{1}{(\kappa+2)^{\sigma(\kappa+2)}} (\kappa+2)^{\sigma(\kappa+2-d)} B^{d} \max(k, d^{\sigma })^d C_{u,d-2}
\\
\leq & \frac{1}{(\kappa+2)^{\sigma(\kappa+2)}}  B^{d } \max(k, (\kappa+2)^{\sigma})^ {\kappa+2} C_{u,d-2}\\
\leq &  B^{\kappa+2 } \max\left(\frac{k}{(\kappa+2)^{\sigma}}, 1\right)^ {\kappa+2} C_{u,\kappa } B^{d-\kappa-2},
\end{aligned}
\end{equation}
with \(C_{u,\kappa }= \left( \max\limits_{0\leq \ell\leq \kappa}  \Gamma_{B,k}^{\ell} (\ell+2)^{(\ell +2)\sigma} \ \rho^2_*[|f|]_{\ell; B_R}+ k^{-1} [|u|]_{1;B_{R}} +[|u|]_{0;B_{R}}\right)\).
Above, we used that for all $0 \leq d \leq \kappa+1$
\begin{equation}\label{ineq:max_dichotomy}
(\kappa+2)^{\sigma(\kappa+2-d)} \max(k, d^{\sigma})^d 
\leq   \max(k, (\kappa+2)^{\sigma})^ {\kappa+2}.
\end{equation}

Thus 
\begin{equation}\label{est:of_interest_for_isotropic_later}
\begin{aligned}
    & A_2 \sum_{d=0}^{\kappa} \left(KR\right)^{\kappa-d}  [|u|]_{d;B_{R}} + [| u|]_{\kappa+1;B_{R}} \\ & \leq \left(\frac{1}{(\kappa+2)^{\sigma(\kappa+2)}}  B^{\kappa+2 } \max(k, (\kappa+2)^{\sigma})^ {\kappa+2} C_{u,\kappa}\right) A_2\left(   \sum_{d=0}^{\kappa}  \left(KR\right)^{\kappa-d} B^{d-\kappa-2} + B^{-1} \right).
\end{aligned}
\end{equation}

Similarly, using the estimates of Lemma~\ref{lemma:rhs_k_scaling_partial_estimate},
\begin{equation}
    \begin{aligned}
    &\rho_*^{2}  [|b\cdot \nabla u |]_{\kappa; B_R} \\
    &\leq  c_b \frac{1}{(\kappa+1)^{2\sigma}}\sum_{d=0}^{\kappa} \left( KR\right)^{\kappa-d} [| \nabla u|]_{d;B_{R}}\\
    &\leq c_b \frac{1}{(\kappa+1)^{2\sigma}}\sum_{d=0}^{\kappa} \left( KR\right)^{\kappa-d} [|  u|]_{d+1;B_{R}}
    \\
    & 
    \leq c_b \frac{1}{(\kappa+1)^{2\sigma}}\sum_{d=0}^{\kappa} \left( KR\right)^{\kappa-d}\frac{1}{(\kappa+2)^{\sigma(\kappa+2)}} \frac{(\kappa+2)^{\sigma(\kappa+2)}}{(d+1)^{\sigma(d+1)}} B^{d+1} \max(k, (d+1)^{\sigma})^{d+1} C_{u,d-1}
    \\
     & 
    \leq c_b \frac{1}{(\kappa+1)^{2\sigma}}\sum_{d=0}^{\kappa} \left( KR\right)^{\kappa-d} \frac{1}{(\kappa+2)^{\sigma(\kappa+2)}}  (\kappa+2)^{\sigma(\kappa-d+1)}B^{d+1} \max(k, (d+1)^{\sigma})^{d+1} C_{u,d-1}
    \\
    &
    \leq c_b \frac{(\kappa+2)^\sigma}{(\kappa+1)^{2\sigma}}\sum_{d=0}^{\kappa} \left( KR\right)^{\kappa-d} \frac{1}{(\kappa+2)^{\sigma(\kappa+2)}}  (\kappa+2)^{\sigma(\kappa-d)}B^{d+1} \max(k, (d+1)^{\sigma})^{d+1} C_{u,d-1}.
    \end{aligned}
    \end{equation}
    We now exploit that for $\kappa\geq 0$,  $\frac{(\kappa+2)^\sigma}{(\kappa+1)^{2\sigma}}\leq 2$ and \eqref{ineq:max_dichotomy} to find that 
    \begin{equation}
    \begin{aligned}
    &\rho_*^{2}  [|b\cdot \nabla u |]_{\kappa; B_R} \\
    &\leq \left(\frac{1}{(\kappa+2)^{\sigma(\kappa+2)}}   B^{\kappa+2} \max(k, (\kappa+2)^{\sigma})^{\kappa+1} C_{u,\kappa}\right) \left(c_b \sum_{d=0}^{\kappa} \left( KR\right)^{\kappa-d} B^{d-1-\kappa}\right).
    \end{aligned}
    \end{equation}
    Taking into account the $k$-scaling of the term $b\cdot \nabla u$:
     \begin{equation}
    \begin{aligned}
    &A_2 k \rho_*^{2}  [|b\cdot \nabla u |]_{\kappa; B_R} \\
    &\leq \left( B^{\kappa+2} \max\left(\frac{k}{(\kappa+2)^{\sigma}}, 1\right)^{\kappa+2} C_{u,\kappa}\right) A_2\left( c_b \sum_{d=0}^{\kappa} \left( KR\right)^{\kappa-d} B^{d-1-\kappa}\right).
    \end{aligned}
    \end{equation}
    One treats similarly the term $nu$
    \begin{equation}
    \begin{aligned}
    &\rho_*^{2}  [|n u |]_{\kappa; B_R} \\
    &\leq  c_n \frac{1}{(\kappa+1)^{2\sigma}}\sum_{d=0}^{\kappa} \left( KR\right)^{\kappa-d} [|  u|]_{d;B_{R}}\\
    &\leq  c_n \frac{1}{(\kappa+1)^{2\sigma}}\sum_{d=0}^{\kappa} \left( KR\right)^{\kappa-d}\frac{1}{(\kappa+2)^{\sigma(\kappa+2)}}  (\kappa+2)^{\sigma(\kappa+2-d)} B^{d} \max(k, d^{\sigma})^{d} C_{u,d-2}\\
    &\leq  c_n \frac{(\kappa+2)^{2\sigma}}{(\kappa+1)^{2\sigma}}\sum_{d=0}^{\kappa} \left( KR\right)^{\kappa-d} \frac{1}{(\kappa+2)^{\sigma(\kappa+2)}} (\kappa+2)^{\sigma(\kappa-d)} B^{d} \max(k,  d^{\sigma})^{d} C_{u,d-2}\\
    &\leq  c_{n,\sigma}\sum_{d=0}^{\kappa} \left( KR\right)^{\kappa-d} \frac{1}{(\kappa+2)^{\sigma(\kappa+2)}}  B^{d} \max(k, (\kappa+2)^{\sigma})^{\kappa} C_{u,d-2}\\
    \end{aligned}
    \end{equation}
    where we have used that for $\kappa \geq 0$, \[\frac{(\kappa+2)^{2\sigma}}{(\kappa+1)^{2\sigma}} \leq 4^{\sigma}\]
    and $c_{n,\sigma} = 4^{\sigma} c_n $. Taking into account the $k^2$ scaling of $nu$, we obtain    
     \begin{equation}
    \begin{aligned}
    & A_2 k^2 \rho_*^{2}  [|n u |]_{\kappa; B_R} \\
    &\leq \left(   B^{\kappa+2} \max\left(\frac{k}{ (\kappa+2)^{\sigma}},1\right)^{\kappa+2} C_{u,\kappa}\right) A_2\left(  c_{n,\sigma} \sum_{d=0}^{\kappa} \left( KR\right)^{\kappa-d} B^{d-2-\kappa}\right).
    \end{aligned}
    \end{equation}
    Finally, we turn our attention to the source term $f$, which readily verifies:
    \[k^2 \rho_*^{2} [| f |]_{\kappa; B_R} 
        \leq  B^{\kappa +2} \max\left(\frac{k}{ (\kappa+2)^{\sigma}},1\right)^{\kappa+2} C_{u,\kappa} B^{-1}.
    \]

    The proof is then completed by choosing $B$ large enough such that for all $\kappa\geq0$
    \[\begin{aligned}
A_2\Big(   \sum_{d=0}^{\kappa}  \left(\frac{KR}{B}\right)^{\kappa-d} B^{-2} + B^{-1} &+ c_b \sum_{d=0}^{\kappa} \left(\frac{KR}{B}\right)^{\kappa-d} B^{-1} +c_{n,\sigma} \sum_{d=0}^{\kappa} \left(\frac{KR}{B}\right)^{\kappa-d} B^{-2} + B^{-1}\Big) \leq 1.
    \end{aligned}
    \]
\end{proof}

\begin{proof}[Proof of Theorem \ref{thm:int_Gev_reg}]
    The proof follows similarly to that of \cite[Theorem 1.7.1]{costabel2010corner}. In particular, there exists a neighbourhood $\mathcal{U}_2(x_0) \subset \Omega_2$ of a point $x_0 \in \Omega_1$ which maps to the ball $B_{R_*}$ in $\mathbb{R}^n$ through a Gevrey-$\sigma$ map $\phi$. The equation $L u =f$ in $\mathcal{U}_2(x_0)$ becomes:
    \[\breve L \breve u = \breve f \textrm{ in } B_{R_*} \quad \textrm{with} \quad  \breve u \circ \phi = u, \ \ \breve f \circ \phi = f, \ \textrm{ and } \ \phi^{-1} \circ \breve L \circ \phi = L. \]
    The operator $\breve L$, has Gevrey-$\sigma$ coefficients on the ball $B_{R_*}$, and so the estimate established in Proposition~\ref{prop:Gevrey_fundamental_est} can be applied to $\breve L$.
    
    The only difference with the proof in \cite[p. 53]{costabel2010corner} is that the inequalities \cite[(1.53)]{costabel2010corner} have to be adjusted to our definition of the Sobolev--Morrey seminorms~\eqref{d:Sobolev-Morrey}:
    \begin{equation}\arraycolsep=1.4pt\def\arraystretch{2.2}
        \left\{\begin{array}{rl}
             [|\breve u|]_{\kappa;B_{R}} &\geq \left(\frac{R}{2\kappa}\right)^{\kappa\sigma } |\breve u|_{\kappa;B_{R/2}} \\
            \rho_*^2 [|\breve L\breve u|]_{\ell; B_R} &\leq \left(\frac{R}{2(\ell+1)}\right)^{(\ell+2)\sigma} |\breve L \breve u|_{\ell;B_{R}}.
        \end{array}
        \right. 
    \end{equation}
    The rest of the arguments are similar to the aformentioned proof in \cite{costabel2010corner}. In particular, a finite covering of the domain $\Omega_1$ is used so that a global estimate~\eqref{est:main_estimate} is inferred from a finite number of local estimates over neighbourhoods of $x_0 \in \Omega_1$. The details are omitted.
    
\end{proof}

\subsection{Regularity up to the boundary}\label{sec:uttb_reg}
\begin{figure}
    \centering
    \includegraphics[width=1.1\linewidth,trim=1cm 20.2cm 1cm 2.1cm, clip]{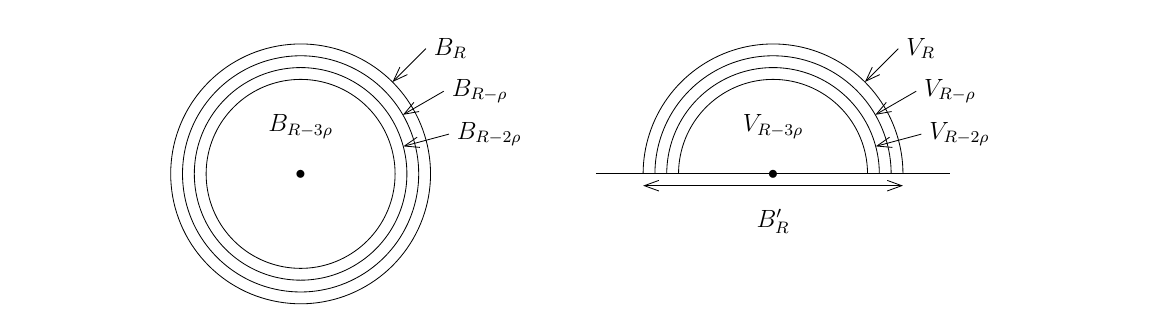}
    \caption{Nested balls and nested half-balls neighbourhoods, from \cite[p. 97]{costabel2010corner}}
    \label{fig:balls}
\end{figure}

In what follows, we establish estimates proving the solution to the elliptic boundary-value problem:
\[
\left\{\begin{array}{rcl}
    Lu &=f \qquad &\textrm{in } \Omega \\
    Tu &=g \qquad &\textrm{on } \partial\Omega\setminus\Gamma\\
    Du &=h \qquad &\textrm{on } \Gamma\\
\end{array}
\right.\]
is Gevrey-$\sigma$ regular whenever the domain, the coefficients, and the problem data are Gevrey-$\sigma$ regular. Above $L$ is a second-order differential operator and $T$, $D$ are differential operator of order 1 and 0, respectively. On the triplet $\{L,T,D\}$, we assume that it is elliptic on $\overline\Omega$ in the sense of \cite[Definition 2.2.33]{costabel2010corner}, and  $\Gamma = \emptyset$ or $\Gamma=\partial \Omega$.

Let us introduce the notation used for the basic domain we consider in this section, i.e., half balls with their flat side corresponding to a flattening of the boundary $\partial \Omega$; see Figure~\ref{fig:balls}. We denote by $B_R$ the ball of radius $R$, by $V_r$ the half ball obtained by intersecting $B_R$ with the half space $\{x_n>0\}$ ($V_R := B_R \cap \{x_n>0\}$), and by $B'_R$ the flat side of the boundary ($B'_R := \partial V_R \cap \{x_n=0\}$).

Let us define the space \[
H^\kappa_0(V,B') =\{ u \in H^\kappa(V), \quad \partial^\ell_r u = 0 \textrm{ on } \partial V \setminus B'  , \quad0\leq \ell \leq \kappa-1\}.
\]

If we assume that $L$ is elliptic and that the boundary operators $T$ and $D$ verify the complementing conditions, then by \cite[Corollary 2.2.16]{costabel2010corner}, the basic elliptic regularity estimate 
\begin{equation}
\label{ineq:basic_elliptic_reg_boundary}
\|u\|_{2,V_{R_*}} \leq \|Lu\|_{0,V_{R_*}} + \|Tu\|_{\frac12,B'_{R_*}} + \|Du\|_{\frac32,B'_{R_*}} + \|u\|_{1,V_{R_*}}
\end{equation}
holds for $u \in H^2_0(V_{R_*},B'_{R_*})$, where the norm $\|\cdot\|_{s,\Omega}$ associated with the fractional Sobolev space $H^s(\Omega)$ is defined as
\[
\|u\|_{s,B'_{R_*}}^2 = \|u\|_{\lfloor s\rfloor,B'_{R_*}}^2 + \sum_{|\alpha| =\lfloor s\rfloor} |\freqk{\partial}^\alpha u |_{s-\lfloor s\rfloor, B'_{R_*}}^2  ,
\]
with
\[
|v |_{s, B'_{R_*}}^2 = \int\int_{B'_{R_*} \times B'_{R_*}} \frac{|u(x)-u(y)|^2}{|x-y|^{(n-1)+2s}}\,\textup{d} x\,\textup{d}y.
\]
We next introduce the analogue of Lemma~\ref{lem:truncated_basic_ell} for points close to the boundary, i.e., we want this estimate to include norms of the boundary operators $T$ and $D$. Along the lines of \cite[Chapter 5]{melenk2004hp}, we state this result for $T u|_{B'_R} =g|_{B'_R} \in H^1(V_R)$ and $Du|_{B'_R} = h|_{B'_R} \in H^2(V_R)$, i.e. for the extension of the boundary condition to the domain of interest $V_R$.

\begin{lemma}\label{lem:truncated_basic_ell_boundary}
    Let $\freqk{k_0}, R^*>0$ and $\sigma \geq 1$. Let $\{L,T,D\}$ be an elliptic triplet with smooth coefficients and assume that \eqref{ineq:basic_elliptic_reg_boundary} holds for all functions in $H^2_0(V_{R_*},B'_{R_*})$. In addition, let $g\in H^1(V_R)$ and $h \in H^2(V_R)$. Then there exists $A_1>0$ such that for any $u \in H^2(V_R)$ with $R \leq R_*$, for all $\rho \in (0, R/2)$ and $k \geq k_0$
    \begin{equation}\label{ineq::basic_elliptic_boundary}
        \begin{aligned}
        \sum_{|\alpha| = 2} \rho^{\sigma|\alpha|} \|\freqk{\partial}^\alpha u\|_{V_{R-|\alpha|\rho}} \leq A_1 \Bigg(\rho^{2\sigma} \| L u\|_{V_{R-\rho}}  &+  \sum_{|\alpha| \leq 1} \rho^{\sigma |\alpha|} \|\freqk{\partial}^\alpha u\|_{V_{R-|\alpha|\rho}} \\& +  \sum_{|\alpha|\leq 2}\rho^{\sigma|\alpha|} \| \partial^\alpha h\|_{V_{R-\rho}}  + \sum_{|\alpha|\leq 1} \rho^{(1+|\alpha|)\sigma} \|\partial^\alpha g
        \|_{V_{R-\rho}} \Bigg).
        \end{aligned}
    \end{equation}
\end{lemma}
\begin{proof}
    The arguments leading to \eqref{ineq:desired_est} in Lemma~\ref{lem:truncated_basic_ell} can be repeated here using the truncation function $\chi_{R,\rho}$ defined in~\eqref{truncation_function}. 
    We recall that $D$ and $T:= \sum\limits_{|\alpha| = 1} t_\alpha \partial^\alpha + t_0 \operatorname{id}$ are zero and first order respectively, thus 
    \[
    D\chi_{R,\rho}  u = \chi_{R,\rho} D u, \qquad T \chi_{R,\rho}  u = \sum\limits_{|\alpha| = 1} t_\alpha u \partial^\alpha  \chi_{R,\rho} +  \chi_{R,\rho} T u.
    \]
    Using the above identity, the slope estimate~\eqref{ineq:cut_off_Gev_est}, and the continuity of the trace operator from $H^s(\R^n)$ into $H^{s-1/2}(\R^{n-1})$ (with $s>1/2$), we obtain
    \[\begin{aligned}
    \|D\chi_{R,\rho}  u\|_{\frac32,B'_{R_*}} \leq& C  \sum_{|\alpha|\leq 2}\rho^{\sigma(|\alpha|-2)} \| \partial^\alpha h\|_{V_{R}}, 
    \\
    \|T \chi_{R,\rho}  u\|_{\frac12,B'_{R_*}} \leq& C \sum_{|\alpha|\leq 1} \rho^{\sigma(|\alpha|-1)} \|\partial^\alpha g\|_{V_{R}} + \sum_{|\alpha|\leq 1} \rho^{\sigma(|\alpha|-2)}\|\partial^\alpha u \|_{V_R},
    \end{aligned}
    \] 
    for some constant $C$ which depends on $R_*$, $T$ and the continuity norm of the aforementioned trace operator.
    The rest of the proof is similar to \ref{lem:truncated_basic_ell} (see also \cite[pp. 100-103]{costabel2010corner}). The details are omitted.
\end{proof}

Our first regularity result will mimic that of \cite[Proposition 2.5.2]{costabel2010corner}. It is an anisotropic estimate, yielding controld over \emph{almost tangential derivatives} of the solution $u$. The motivation behind this approach is the fact that tangential derivatives compose more naturally with the boundary conditions (indeed, taking tangential derivatives $Du =h$ only depend on the boundary function $h$).
To this end let us introduce, the following semi-norms: Let $\mathcal{U}$ be a domain in $\R^n$ and let $\kappa,\,m\in \mathbb{N}$
\[
|u|_{\kappa,m,\mathcal{U}} = \max_{\substack{|\alpha| = \kappa\\ \alpha_n\leq m}} \|\freqk{\partial}^\alpha u\|_{\mathcal{U}},
\]
where $\alpha = (\alpha',\alpha_n)$, $\alpha' \in \mathbb{N}^{n-1}$.
On the boundary, we use the semi-norm
\[
|u|_{\kappa,s,\mathcal{U}'} = \max_{\substack{|\alpha| = \kappa}} \|\freqk{\partial}^\alpha u\|_{s,\mathcal{U}'}.
\]
Accordingly, we also introduce weighted anisotropic semi-norms of the Sobolev-Morrey type: We set
\begin{equation}\label{d:Sobolev-Morrey_anisotropic}
\begin{aligned}
[|u|]_{0,0,V_R} &:= \|u\|_{V_R},
\\
[|u|]_{\kappa,m,V_R} &:= \max_{0\leq \rho\leq \frac{R}{2\kappa}} \rho^{\sigma \kappa}|u|_{\kappa,m,V_{R-\kappa\rho}},\\
\rho^2_*[|f|]_{\kappa,m,V_R} &:= \max_{0\leq \rho\leq \frac{R}{2(\kappa+1)}} \rho^{\sigma(2+ \kappa)}|f|_{\kappa,m,V_{R-(\kappa+1)\rho}},\\
\rho_*[|h|]_{\kappa,\alpha,B'_R} &:= \max_{0\leq \rho\leq \frac{R}{2(\kappa+1)}} \rho^{\sigma\kappa} | h|_{\kappa, \alpha, V_{R-(\kappa+1)\rho}} \\
\rho^{1}_*[|g|]_{\kappa,\alpha,B'_R} &:= \max_{0\leq \rho\leq \frac{R}{2(\kappa+1)}} \rho^{(\kappa+1)\sigma} |g
        |_{\kappa, \alpha, V_{R-(\kappa+1)\rho}}
\end{aligned}
\end{equation}
To fully exploit the dichotomy between normal and tangential derivatives we state the analogue of the commutator estimate given in Lemma~\ref{lemma::commutator_estimate} for the anisotropic case (see also \cite[Lemma 2.6.2]{costabel2010corner}):
\begin{lemma}\label{lemma::comutator_anisotropic}
    Let $M\geq 0$, $\sigma \geq 1$, and $\freqk{k_0},\, R_*>0 $. There exists constants $c_1$ and $K$ such that for all $R\in (0,R_*]$, all commutators $[a(x), \freqk{\partial}^\beta]$ with Gevrey coefficients $a(x) \in G^\sigma(B_{R})$ with constant $M$, and all multi-indices $|\alpha|$, $\beta$ with $|\alpha|+|\beta| = b \geq 2$, for all $\rho \in \left(0,\frac{R}{2b}\right]$ and $k \geq k_0$
\begin{equation}\label{}
    \rho^{\sigma b} \|[a(x), \freqk{\partial}^\beta] \freqk{\partial}^\alpha u\|_{V_{R-(b-1)\rho}} \leq c_1 \sum_{d=0}^{b-1} \left( KR\right)^{b-d} [| u|]_{d,m;V_{R}}.
\end{equation}    
Above, $\alpha$ is a multi-index with $|\alpha|\leq2$, and $m = \alpha_n + \beta_n$ is the total number of normal derivative involved in the commutator. 
\end{lemma}
\begin{proof}
    The proof of this lemma is very similar to that of Lemma~\ref{lemma::commutator_estimate} and is omitted.
\end{proof}

Our goal is to state a proposition establishing the bounds on \emph{almost tangential derivatives} of the solution $u$ (\emph{cf.} \cite[Propostion 2.6.6]{costabel2010corner}). To achieve this goal, we plan to follow the same plan as in Section~\ref{sec:int_reg}. In particular, we establish an intermediate estimate for the $\kappa+2$ semi-norm of the solution as a function of $Lu$, we then substitute $Lu$ for its components and bound them. The following lemma extends Lemma~\ref{lemma:elliptic_reg_kappa} up to the boundary.
\begin{lemma}\label{lemma:elliptic_reg_kappa_uttb}
    Let $R_*>0$. Assume the basic elliptic regularity estimate~\ref{ineq::basic_elliptic_boundary} holds.
Then there exists a constant $A_2>1$ such that for all $R \in (0,R_*]$ and for all $\kappa\in \mathbb{N}$,
     \begin{equation}\label{ineq:commutator_4_uttb}
        \begin{aligned}
        [|u|]_{\kappa+2,2;V_{R}} & \leq A_2 \Bigg(\rho_*^{2}  [|Lu|]_{\kappa,0; V_R}  + c_2 \sum_{d=0}^{\kappa} \left(KR\right)^{\kappa-d} [| u|]_{d,2 ;V_{R}} \\
        & \hphantom{\leq A_1' \Bigg(} +  [|u|]_{\kappa+1,1;V_{R}} + \sum_{|\alpha|\leq 2} \rho_*[|h|]_{\kappa+|\alpha|,\alpha,B'_R} + \sum_{|\alpha|\leq 1} \rho^{1}_*[|g|]_{\kappa+|\alpha|,\alpha,B'_R} \Bigg).
        \end{aligned}
    \end{equation}      
\end{lemma}
\begin{proof} 
    The proof relies on recursively applying Lemmas~\ref{lem:truncated_basic_ell_boundary} and  \ref{lemma::comutator_anisotropic}.  First, let $\beta$ be a multi-index of length $|\beta| = \kappa$ such that its normal component $\beta_n=0$. Because $\beta$ is tangential, the boundary conditions for $\partial^\beta u$ can be written as
    \begin{equation}
        \begin{aligned}
            D \partial^\beta u &= \partial^\beta D u + [\partial^\beta,D] u =  \partial^\beta h + [\partial^\beta,D] u\\
            T \partial^\beta u &
            = \partial^\beta T u + [\partial^\beta,T] u =  \partial^\beta g + [\partial^\beta,T] u.
        \end{aligned}
    \end{equation}
    We use  \eqref{ineq::basic_elliptic_boundary} with $\freqk{\partial}^\beta u$ and $R = R-|\beta|\rho$, to obtain
        \begin{equation}
        \begin{aligned}
        \sum_{|\alpha| = 2} \rho^{\sigma|\alpha|} \|\freqk{\partial}^{\alpha+\beta} u\|_{V_{R-|\beta|\rho-|\alpha|\rho}} \leq A_1' \Bigg(&\rho^{2\sigma}\left( \|\partial^\beta  L u\|_{V_{R-|\beta|\rho-\rho}} + \|[\partial^\beta, L] u\|_{V_{R-|\beta|\rho-\rho}}\right) \\& 
        +  \sum_{|\alpha| \leq 1} \rho^{\sigma |\alpha|} \|\freqk{\partial}^{\alpha+\beta} u\|_{V_{R-|\beta|\rho-|\alpha|\rho}} \\& +  \sum_{|\alpha|\leq 2}\rho^{\sigma|\alpha|} \left(\| \partial^{\alpha+\beta} h\|_{V_{R-|\beta|\rho-\rho}}  + \| [\partial^\beta, \partial^\alpha D]  u\|_{V_{R-|\beta|\rho-\rho}} \right) 
        \\& + \sum_{|\alpha|\leq 1} \rho^{(1+|\alpha|)\sigma} \left( \|\partial^{\alpha+\beta} g
        \|_{V_{R-|\beta|\rho-\rho}} + \| [\partial^\beta, \partial^\alpha T]  u\|_{V_{R-|\beta|\rho-\rho}}\right) \Bigg).
        \end{aligned}
    \end{equation}

    We multiply by the scaling $\rho^{\sigma|\beta|}$ \begin{equation}\label{ineq:commutator_3_uttb}
        \begin{aligned}
     \rho^{\sigma(\kappa+2)} \|\partial^{\alpha+\beta} u\|_{V_{R-|\alpha|\rho}} \leq A_1 \Bigg(&\rho^{(\kappa+2)\sigma}\left( \|\partial^\beta  L u\|_{V_{R-(\kappa+1)\rho}} + \|[\partial^\beta, L] u\|_{V_{R-(\kappa+1)\rho}}\right) \\& 
        +  \sum_{|\alpha|\leq 1} \rho^{\sigma (\kappa+|\alpha|)} \|\partial^{\alpha+\beta} u\|_{V_{R-|\alpha|\rho}}\\& +  \sum_{|\alpha|\leq 2}\rho^{\sigma(\kappa+|\alpha|)} \left(\| \partial^{\alpha+\beta} h\|_{V_{R-(\kappa+1)\rho}}  + \| [\partial^\beta, \partial^\alpha D]  u\|_{V_{R-(\kappa+1)\rho}} \right) 
        \\& + \sum_{|\alpha|\leq 1} \rho^{(\kappa+1+|\alpha|)\sigma} \left( \|\partial^{\alpha+\beta} g
        \|_{V_{R-(\kappa+1)\rho}} + \| [\partial^\beta, \partial^\alpha T]  u\|_{V_{R-(\kappa+1)\rho}}\right) \Bigg).
        \end{aligned}
    \end{equation}    

    Similarly to \cite[Proposition 2.6.3]{costabel2010corner}, we use Lemma~\ref{lemma::comutator_anisotropic}, with $|\beta| = \kappa +2,\ \kappa+1, \kappa$ and with $m=2$ to estimate the commutators
\begin{equation}
\begin{aligned} \label{ineq:applying_lem_com_plus_uttb}
        \rho^{(\kappa+2)\sigma} \|[\freqk{\partial}^\alpha, L]& u\|_{V_{R-(\kappa+1)\rho}}  +  \sum_{|\alpha|\leq 2}\rho^{\sigma(\kappa+|\alpha|)} \| [\partial^\beta, \partial^\alpha D]  u\|_{V_{R-(\kappa+1)\rho}}
       \\ &+ \sum_{|\alpha|\leq 1} \rho^{(\kappa+1+|\alpha|)\sigma}\| [\partial^\beta, \partial^\alpha T]  u\|_{V_{R-(\kappa+1)\rho}}
     \leq c_2 \sum_{d=0}^{\kappa+1} \left(KR\right)^{\kappa+2-d}[| u|]_{d,2;V_{R}},\\
\end{aligned}
    \end{equation}
    
Putting back in \eqref{ineq:commutator_3_uttb}, taking the maximum over $\rho$, and using the previously introduced notation~\eqref{d:Sobolev-Morrey_anisotropic} yields
     \begin{equation}
        \begin{aligned}
        [|u|]_{\kappa+2,2;V_{R}} & \leq A_1' \Bigg(\rho_*^{2}  [|Lu|]_{\kappa,0; V_R}  + c_2 \sum_{d=0}^{\kappa} \left(KR\right)^{\kappa-d} [| u|]_{d,2 ;V_{R}} \\
        & \hphantom{\leq A_1' \Bigg(} +  [|u|]_{\kappa+1,1;V_{R}} + \sum_{|\alpha|\leq 2} \rho_*[|h|]_{\kappa+|\alpha|,\alpha,B'_R} + \sum_{|\alpha|\leq 1} \rho^{1}_*[|g|]_{\kappa+|\alpha|,\alpha,B'_R} \Bigg).
        \end{aligned}
    \end{equation}    
\end{proof}

\begin{proposition}\label{prop:almost_tangetial}
     Let $\sigma\geq 1$ and $k_0, \,R_*>0$. Let the elliptic operator $L$ and the boundary operators $D$, $T$ have Gevrey-$\sigma$ coefficients. Let $b,\, n \in G^\sigma(\Omega)$. We assume that \eqref{ineq:basic_elliptic_reg_boundary} holds for any $u \in H^2_0(V_{R_*},B'_{R_*})$. Then there exists a constant $B\geq1$ such that for all $\kappa \in \mathbb{N}$, all $R \in (0,R_*]$, and for any $u\in H^2(V_R)$ satisfying~\eqref{eq:L_withrhs}
       \begin{equation}\label{bndry_first}
        \begin{aligned}
        [| u|]_{\kappa+2,2;V_{R}} 
        \leq  B^{\kappa +2}\max \left\{
        \frac{k}{(\kappa+2)^\sigma},1\right\}^{\kappa+2} \Bigg(&  \max_{0\leq \ell\leq \kappa}\Gamma_{B,k}^{\ell} (\ell+2)^{(\ell +2)\sigma} \ \rho^2_*[|f|]_{\ell,0; V_R} 
        \\ &+ \max_{0\leq \ell\leq \kappa+2} \Upsilon_{B,k}^{\ell} \ell^{\ell\sigma}\rho_*[|h|]_{\ell,2,B'_R} 
        \\ & +  \max_{0\leq \ell\leq \kappa+1}
        \Upsilon_{B,k}^{\ell+1} (\ell+1)^{(\ell +1)\sigma}\rho^{1}_*[|g|]_{\ell,1,B'_R} \\
        &+ k^{-1} [|u|]_{1;V_{R}} +[|u|]_{0;V_{R}}\Bigg),
     \end{aligned}    
    \end{equation}
    with $\Gamma_{B,k}^{\ell} = B^{-\ell-1} \left(\max \left\{k,(\ell+2)^\sigma\right\}\right)^{-\ell}$
        and 
        $\Upsilon_{B,k}^{\ell} = B^{-\ell+1} \left(\max \left\{k,\ell^\sigma\right\}\right)^{-\ell} .$
\end{proposition}
\begin{proof}
   The proof follows similarly to that of Proposition~\ref{prop:Gevrey_fundamental_est}, by using Lemma~\ref{lem:truncated_basic_ell_boundary} and reasoning by induction. See also \cite[Propositions 2.5.2 \& 2.6.6]{costabel2010corner}. The details are omitted.
\end{proof}

The estimates for general derivatives of order less than $\kappa + 2$ (as opposed to \emph{almost tangential ones}) rely on the assumption that the boundary is non-characteristic, that is,
\begin{equation}\label{nonchar_condition}
a^{(0,\ldots,0,2)} \neq 0.
\end{equation}
along the boundary in the notation of \eqref{d:ellipticity}.
We note that this concept can be extended to $N\times N$ system of equations by adopting the definition given in \cite[Eq. (2.50)]{costabel2010corner}.

\begin{proposition}\label{prop:isotropic}
     Let $\sigma\geq 1$ and $R_*>0$. Let the elliptic operator $L$ and the boundary operators $D$, $T$ have Gevrey-$\sigma$ coefficients. Let $b,\, n \in G^\sigma(\Omega)$. We assume that \eqref{ineq:basic_elliptic_reg_boundary} holds for any $u \in H^2_0(V_{R_*},B'_{R_*})$. Then there exist constants $B_0,\,B\geq1$ such that for all $\kappa,m \in \mathbb{N}$ with $m\leq \kappa$, all $R \in (0,R_*]$, and for any $u\in H^2(V_R)$ satisfying~\eqref{eq:L_withrhs}
    \begin{equation}\label{bndry_isotropic}
        \begin{aligned}
        [| u|]_{\kappa+2,m+2;V_{R}} 
        \leq  B_0^{m} B^{\kappa +2}\max \left\{
        \frac{k}{(\kappa+2)^\sigma},1\right\}^{\kappa+2} \Bigg(& \max_{\substack{0\leq \ell\leq \kappa\\0\leq \nu \leq \min\{\ell,m\}}}  \Gamma_{B,B_0,k}^{m,\ell} (\ell+2)^{(\ell +2)\sigma} \ \rho^2_*[|f|]_{\ell,\nu; V_R}
        \\ &+  \max_{0\leq \ell\leq \kappa+2} \Upsilon_{B,B_0,k}^{m,\ell} \ell^{\ell\sigma}\rho_*[|h|]_{\ell,2,B'_R} 
        \\ & +   \max_{0\leq \ell\leq \kappa+1} \Upsilon_{B,B_0,k}^{m,\ell+1} (\ell+1)^{(\ell +1)\sigma}\rho^{1}_*[|g|]_{\ell,1,B'_R}    \\
        &+ k^{-1} [|u|]_{1;V_{R}} +[|u|]_{0;V_{R}}\Bigg),
     \end{aligned}    
    \end{equation}
    with $$\Gamma_{B,B_0,k}^{m,\ell} = B_0^{-m+1} B^{-\ell-1} \left(\max \left\{k,(\ell+2)^\sigma\right\}\right)^{-\ell}$$ and $$\Upsilon_{B,B_0,k}^{m,\kappa} = B_0^{-m+1}  B^{-\ell+1} \left(\max \left\{k,\ell^\sigma\right\}\right)^{-\ell}.$$
\end{proposition}
\begin{proof} The proof follows along the lines of \cite[Proposition 5.5.2]{melenk2004hp} and reasons by induction. Indeed, for $m=0$, estimate \eqref{bndry_isotropic} is a direct consequence of \eqref{bndry_first}. Assume then that \eqref{bndry_isotropic} holds for $0, \ldots,m-1$:
 \begin{equation}\label{est:hyp_induction}
        \begin{aligned}
        [| u|]_{\kappa+2,m+1;V_{R}} 
        \leq  B_0^{m-1} B^{\kappa +2}\max \left\{
        \frac{k}{(\kappa+2)^\sigma},1\right\}^{\kappa+2} \Bigg(&\max_{\substack{0\leq \ell\leq \kappa\\0\leq \nu \leq \min\{\ell,m-1\}}} \Gamma_{B,B_0,k}^{m-1,\ell} (\ell+2)^{(\ell +2)\sigma} \ \rho^2_*[|f|]_{\ell,\nu; V_R}
        \\ &+ \max_{0\leq \ell\leq \kappa+2}
        \Upsilon_{B,B_0,k}^{m-1,\ell} \ell^{\ell\sigma}\rho_*[|h|]_{\ell,2,B'_R} 
        \\ & +   \max_{0\leq \ell\leq \kappa+1} \Upsilon_{B,B_0,k}^{m-1,\ell+1} (\ell+1)^{(\ell +1)\sigma}\rho^{1}_*[|g|]_{\ell,1,B'_R}    \\
        &+ k^{-1} [|u|]_{1;V_{R}} +[|u|]_{0;V_{R}}\Bigg).
        \end{aligned}
    \end{equation} 
Condition \eqref{nonchar_condition} yields
\[
\partial^2_n u = {a^{(0,\ldots,2)}}^{-1} L u + \sum_{\substack{|\alpha|\leq 2\\ \alpha_n<2}} N^\alpha \partial^\alpha u,
\]
with $N^\alpha = -\frac{a_\alpha}{a^{(0,\ldots,2)}}$.
As a consequence, there exists $B_0>0$ such that for all $R\in (0,R_*)$ and all $u \in H^2(V_R)$ (note that here no boundary conditions are needed)
\begin{equation}\label{est:make_isotropic}
|u|_{2,2,V_R}\leq B_0 \left(|Lu|_{0,0,V_R} + |u|_{2,1,V_R}\right).
\end{equation}
Applying \eqref{est:make_isotropic} in $V_{R-(\kappa+2)\rho}$ for $\freqk{\partial}^{\beta}\freqk{\partial_n}^m u $ with $\beta = (\beta',0)$ and $|\beta| = \kappa-m$ 
\begin{equation}\label{est:make_isotropic_2}
\begin{aligned}
|u|_{\kappa+2,m+2,V_{R-(\kappa+2)\rho}}\leq B_0 \big(|L u|_{\kappa,m,V_{R-(\kappa+2)\rho}} &+ \left| [L, \freqk{\partial}^{\beta}\freqk{\partial_n}^m] u\right|_{0,0,V_{R-(\kappa+2)\rho}}\\&+ |u|_{\kappa+2,m+1,V_{R-(\kappa+2)\rho}}\big).
\end{aligned}
\end{equation}
We multiply by $\rho^{(\kappa+2)\sigma}$ and use Lemma~\ref{lemma::comutator_anisotropic} to obtain
\[
\rho^{(\kappa+2)\sigma} \left| [L, \freqk{\partial}^{\beta}\freqk{\partial_n}^m] u\right|_{0,0,V_{R-(\kappa+2)\rho}} \leq c_1 \sum_{d=0}^{\kappa+1} \left(KR\right)^{\kappa+2-d}[| u|]_{d,\min\{d,m+2\};V_{R}},
\]
for constants $c_1$ and $K$ independent of $R\in(0,R_*]$ and $\kappa$. Plugging \eqref{est:hyp_induction} into \eqref{est:make_isotropic_2}, we obtain
\begin{equation}
        \begin{aligned}
        [| u|]_{\kappa+2,m+2;V_{R}} 
        \leq  B_0^{m} B^{\kappa +2}\max \left\{
        \frac{k}{(\kappa+2)^\sigma},1\right\}^{\kappa+2} \Bigg(& \max_{\substack{0\leq \ell\leq \kappa\\0\leq \nu \leq \min\{\ell,m\}}} \Gamma_{B,B_0,k}^{m,\ell} (\ell+2)^{(\ell +2)\sigma} \ \rho^2_*[|f|]_{\ell,\nu; V_R}
        \\ &+  \max_{0\leq \ell\leq \kappa+2} \Upsilon_{B,B_0,k}^{m,\ell} \ell^{\ell\sigma}\rho_*[|h|]_{\ell,2,B'_R} 
        \\ & +  \max_{0\leq \ell\leq \kappa+1} \Upsilon_{B,B_0,k}^{m,\ell+1}  (\ell+1)^{(\ell +1)\sigma}\rho^{1}_*[|g|]_{\ell,1,B'_R}    \\
        &+ k^{-1} [|u|]_{1;V_{R}} +[|u|]_{0;V_{R}}\Bigg)\\&+
        c_1 \sum_{d=0}^{\kappa+1} \left(KR\right)^{\kappa+2-d}[| u|]_{d,\min\{d,m+2\};V_{R}}.
        \end{aligned}
    \end{equation} 
    To bound the last term, we use similar arguments as in the proof of Proposition~\ref{prop:Gevrey_fundamental_est} (see \eqref{est:of_interest_for_isotropic_later}). By eventually adjusting the value of $B$, we obtain the desired estimate.
\end{proof}

We are now ready to present the proof of the up-to-the-boundary Gevrey regularity estimates.
\begin{proof}[Proof of Theorem~\ref{thm:uttb_Gev_reg}]
    Let $x_0 \in \overline \Omega_1$. We identify two cases:
    \begin{itemize}
        \item $x_0 \in \Omega$, and is thus an interior point. Then, there exist two neighbourhoods of $x_0$:
        $\mathcal{U}_1, \,\mathcal{U}_2$ with $\overline{\mathcal{U}}_1 \subset \mathcal{U}_2\subset \Omega$ such that Theorem~\ref{thm:int_Gev_reg} can be applied, with a constant $A = A_{x_0}$ (where we stress the dependence of the constant on the spatial point $x_0$):
        \begin{equation}\label{est:lego_inter}
       \begin{aligned}
        |u|_{\kappa+2;\mathcal{U}_1} 
        \leq A^{\kappa +2}\max \left(
        k, (\kappa+2)^\sigma\right)^{\kappa+2} \Bigg(&  \max_{0\leq \ell\leq \kappa} A^{-\ell-1} \frac{1}{\max (k, (\ell+2)^{\sigma})^{\ell}} 
        |f|_{\ell; \mathcal{U}_2} \\
        &+ k^{-1} |u|_{1;\mathcal{U}_2} +|u|_{0;\mathcal{U}_2} \Bigg)
        \end{aligned}
    \end{equation}
        \item $x_0 \in \partial \Omega$, and is thus a boundary point. In this case, the proof follows from mapping a neighbourhood $\mathcal{U}_2$ of the the point $x_0$ to the half-ball $V_R$ using a Gevrey-$\sigma$-regular map $\phi$. Applying Proposition~\ref{prop:isotropic} with $m =\kappa$ and $A= B_0B$ (depending again on the point $x_0$), we obtain
        \begin{equation}\label{est:lego_bdry}
       \begin{aligned}
       |u|_{\kappa+2;\mathcal{U}_1} 
        \leq A^{\kappa +2}\max \left(
        k, (\kappa+2)^\sigma\right)^{\kappa+2} \Bigg(& \max_{0\leq \ell\leq \kappa} A^{-\ell-1} \frac{1}{\max (k, (\ell+2)^{\sigma})^{\ell}}|f|_{\ell; \mathcal{U}_2} \\&+\max_{0\leq \ell\leq \kappa+2}
        A^{-\ell+1} \frac{1}{\max (k, \ell^{\sigma})^{\ell}}  |h|_{\ell;\mathcal{U}_2}  
        \\ &
        +   \max_{0\leq \ell\leq \kappa+1} A^{-\ell} \frac{1}{\max (k, (\ell+1)^{\sigma})^{\ell+1}} |g|_{\ell;\mathcal{U}_2} \\
        &+ k^{-1} |u|_{1;\mathcal{U}_2} +|u|_{0;\mathcal{U}_2} \Bigg)
        \end{aligned}
\end{equation}
The neighbourhood $\mathcal{U}_1:=\phi^{-1}(V_{R/2})$. The steps to obtain \eqref{est:lego_bdry} from \eqref{bndry_isotropic} are similar to those taken in the proof of Theorem~\ref{thm:int_Gev_reg} and are omitted.
    \end{itemize}
    Finally, by varying $x_0 \in \overline \Omega_1$, we can extract a finite covering of the compact set $\overline \Omega_1$ by open sets $\mathcal{U}_1$. The proof is then complete by combining a finite number of estimates \eqref{est:lego_inter} and \eqref{est:lego_bdry}.
\end{proof}

\subsection{Transmission/piecewise-regular problems}
In the study of scattering problems or when truncating the domain using PML (Perfectly Matched Layer) for example, one often deals with two (or more) equations simultaneously, involving `similar' operator $L^{\pm}$, one in an interior domain $\Omega^+$ and the other over its complementary domain in $\Omega$, $\Omega^- = \Omega \setminus (\Omega^+ \cup I)$, with $I = \partial \Omega^- \cap \partial\Omega^+$ being the interface between the two domains $\Omega^+$ and $\Omega^-$.

We illustrate in Figure~\ref{fig:transmission_problem} a typical situation of interest, where the exterior boundary is $\partial\Omega$ and
\[
\partial \Omega^+ = I \quad\tand \quad \partial \Omega^- = I \cup \partial \Omega.
\]
Let $\Gamma = \emptyset$ or $\Gamma = \partial\Omega$.

In this section, we present a result dealing with the Gevrey regularity for such transmission problems. We note that Sobolev and analytic regularity results can be found in, e.g.~\cite[Theorem 5.2.2]{costabel2010corner} and \cite[Theorem 4.20]{mclean2000strongly}. 

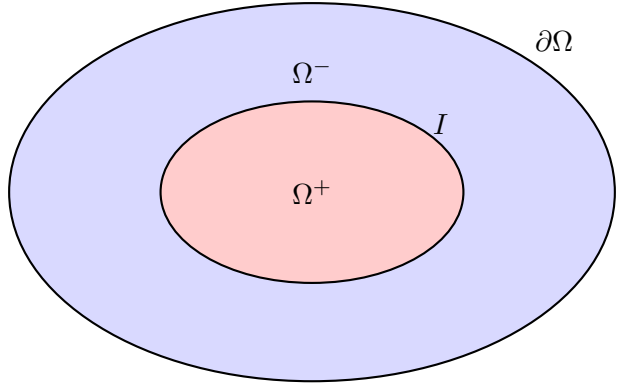
\begin{figure}
    \centering
    \begin{tikzpicture}[scale=1]

  \filldraw[
    fill=blue!15,
    draw=black,
    thick
  ]
  (0,0) ellipse (4cm and 2.5cm);

  \filldraw[
    fill=red!20,
    draw=black,
    thick
  ]
  (0,0) ellipse (2cm and 1.2cm);

  \node at (0,1.6) {$\Omega^-$};
  \node at (0,0) {$\Omega^+$};

  \node[black] at (3.2,2.0) {$\partial \Omega$};
  \node[black]  at (1.7,0.9) {$I$};

\end{tikzpicture}
    \caption{Interface configuration with two subdomains}
    \label{fig:transmission_problem}
\end{figure}

The problem under consideration is of the form:
\begin{equation}
    \left\{
    \begin{array}{rll}
         L^+u^+ = & k^2(f^+- k^{-1}b^+ \cdot \nabla u^+ + n u^+) & \mathrm{in} \quad \Omega^+  \\
         L^-u^-  = & k^2(f^-- k^{-1}b^- \cdot \nabla u^- + n u^-) & \mathrm{in} \quad \Omega^- \\
         Tu^-  = & g & \mathrm{on} \quad \partial \Omega \setminus \Gamma \\
         Du^-  = & h & \mathrm{on} \quad \Gamma \\
         T^+u^+ + T^-u^-  = & g_I & \mathrm{on} \quad I \\
         D^+u^+ + D^-u^-  = & h_I & \mathrm{on} \quad I 
    \end{array}
    \right.\label{e:transmission}
\end{equation}
where the system $\{L^+,L^-,C_\Gamma, C_I\}$ with 
\[C_\Gamma = (T,D) \qquad C_I = (T^+ \ T^-, D^+ \ D^-)\]
is elliptic in the sense of \cite[Def 5.1.2]{costabel2010corner}. Once again, we assume the boundary and interface conditions to be given by the restriction to the boundary of functions $g,\,g_I$ and $h,\,h_I$ defined on $\Omega$; see Remark~\ref{r:boundary_cond}.

To present our piecewise regularity result, it is convenient to introduce the space of piecewise Sobolev spaces
\[
\textup{P}H^\kappa(\Omega) := \{u = (u^-,u^+): u^+\in H^\kappa(\Omega^+) , u^-\in H^\kappa(\Omega^-)\}
\]
for $\kappa \in \mathbb{N}$, equipped with the natural (semi-)norms
\[
\|u\|_{\textup{P}\kappa;\Omega}^2 := \|u^+\|_{\kappa;\Omega^+}^2 +\|u^-\|_{\kappa;\Omega^-}^2, \qquad 
|u|_{\textup{P}\kappa;\Omega}^2 := |u^+|_{\kappa;\Omega^+}^2 +|u^-|_{\kappa;\Omega^-}^2.
\]

\begin{theorem}\label{thm:transmission}
Let $\sigma\geq 1$, $k_0>0$. Let the domain $\Omega$ be Gevrey and let it be subdivided into two subdomains $\Omega^+$ and $\Omega^-$ with a Gevrey interface (as in Figure~\ref{fig:transmission_problem}). Let the system $\{L^+,L^-,C_\Gamma, C_I\}$ be elliptic and let their coefficients be Gevrey-$\sigma$ on $\overline\Omega^{\pm}$. Let $b^\pm$, $n^{\pm}$  be Gevrey-$\sigma$ on $\overline\Omega^{\pm}$. 
Let $\mathcal{U}_1$, $\mathcal{U}_2$ be bounded opens of $\mathbb{R}^n$ verifying  $\overline{\mathcal{U}}_1\subset \mathcal{U}_2$.
We denote 
\[
\Omega_j^\pm = \mathcal{U}_j\cap \Omega^\pm \qquad \mathrm{for } \ j=1,2, \qquad I_2 = \mathcal{U}_2\cap I, \qquad \Gamma_2 = \mathcal{U}_2 \cap \partial \Omega.
\]

There exists a constant $A\geq 1$ such that for all $\kappa\in\mathbb{N}$, $k\ge k_0$, $f \in \textup{P}H^\kappa(\Omega_2)$, $h,h_I \in H^{\kappa+2}(\Omega_2)$, $g,g_I \in H^{\kappa+1}(\Omega_2)$,
and $u\in \textup{P}H^2(\Omega_2)$ satisfying \eqref{e:transmission} in $\Omega_2$,
\begin{equation}\label{ineq:transmission}
\begin{aligned}
        |u|_{\textup{P}\kappa+2;\Omega_1} 
        \leq A^{\kappa +2}\max \left(
        k, (\kappa+2)^\sigma\right)^{\kappa+2} \Bigg(& \max_{0\leq \ell\leq \kappa}  A^{-\ell-1} \frac{1}{\max (k, (\ell+2)^{\sigma})^{\ell}} |f|_{\textup{P}\ell; \Omega_2} \\&+
        \max_{0\leq \ell\leq \kappa+2} A^{-\ell+1} \frac{1}{\max (k, \ell^{\sigma})^{\ell}}  |h|_{\ell;\Omega_2}  
        \\ &
        +   \max_{0\leq \ell\leq \kappa+1} A^{-\ell} \frac{1}{\max (k, (\ell+1)^{\sigma})^{\ell+1}} |g|_{\ell;\Omega_2} \\
         \\&+
        \max_{0\leq \ell\leq \kappa+2}   A^{-\ell+1} \frac{1}{\max (k, \ell^{\sigma})^{\ell}}|h_I|_{\ell;\Omega_2}  
        \\ &
        +  \max_{0\leq \ell\leq \kappa+1} A^{-\ell} \frac{1}{\max (k, (\ell+1)^{\sigma})^{\ell+1}} |g_I|_{\ell;\Omega_2} \\
        &+ k^{-1} |u|_{\textup{P}1;\Omega_2} +|u|_{\textup{P}0;\Omega_2} \Bigg)
\end{aligned}
\end{equation}
\end{theorem}

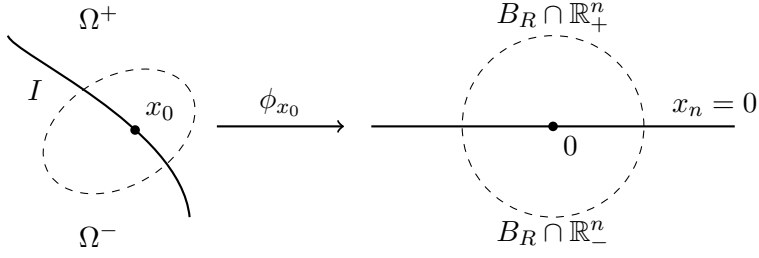
\begin{figure}
    \centering
    \begin{tikzpicture}[scale=1.2]

\begin{scope}

    \draw[thick] (-2,1) .. controls (-1.95,0.8) and (-0.1,0.) .. (0,-1);

    \draw[dashed, rotate=30] (-.7,0.34) ellipse (0.9 and 0.6);

    \fill (-0.6,-0.04) circle (1.5pt);
    \node[above right] at (-0.6,-0.04) {$x_0$};

    \node at (-1,1.2) {$\Omega^+$};
    \node at (-1,-1.2) {$\Omega^-$};
    \node at (-1.7,0.4) {$I$};

    \draw[->, thick] (0.3,0) -- (1.7,0) node[midway, above] {$\phi_{x_0}$};
\end{scope}

\begin{scope}[shift={(4,0)}]


    \draw[thick] (-2,0) -- (2,0);

    \draw[dashed] (0,0) circle (1);

    \fill (0,0) circle (1.5pt);
    \node[below right] at (0,0) {$0$};

    \node at (0,1.2) {$B_R \cap \mathbb{R}^n_+$};
    \node at (0,-1.2) {$B_R \cap \mathbb{R}^n_-$};
    \node[above right] at (1.2,0) {$x_n = 0$};

\end{scope}

\end{tikzpicture}
    \caption{Flattening of interface via local diffeomorphism $ \phi_{x_0}: \mathcal{U} \mapsto \mathcal B_R$.}
    \label{fig:flatten_interface}
\end{figure}

The general strategy of the proof of Theorem~\ref{thm:transmission} follows along the lines of, e.g., \cite[Theorem 4.20]{mclean2000strongly} and \cite[Chapter 5]{costabel2010corner}. It relies on a flattening of the interface $I$ akin to what we did for the study of the regularity up to the boundary in Section~\ref{sec:uttb_reg}; see Fig.~\ref{fig:flatten_interface}. The arguments leading up to Theorem~\ref{thm:uttb_Gev_reg} are then used to obtain the desired piecewise regularity result. In particular, the following intermediate result is established.

Let us define the space \[
\textup{P}H^\kappa_0(B^-\cup B^+, B') =\{ u = (u^-,u^+) \in \textup{P}H^\kappa(B^-\cup B^+), \quad \partial^\ell_r u^\pm = 0 \textrm{ on } \partial B^\pm\setminus B'  , \quad0\leq \ell \leq \kappa-1\}.
\]
We assume the triplet \{L, T, D\} is such that the basic elliptic regularity estimate
   \begin{equation}
\label{ineq:basic_elliptic_reg_boundary_transmission}
\begin{aligned}
\|u^+\|_{2;B_R\cap \R^n_+} + \|u^-\|_{2;B_R\cap \R^n_-} \leq& \|L^+u^+\|_{0,B_R\cap \R^n_+} + \|L^-u^-\|_{0,B_R\cap \R^n_-}\\& + \|T^+u^++T^-u^-\|_{\frac12,B'_{R_*}} + \|D^+u^++D^-u^-\|_{\frac32,B'_{R_*}} \\ 
&+ \|u^+\|_{1,B_R\cap \R^n_+}+\|u^-\|_{1,B_R\cap \R^n_-}.
\end{aligned}
\end{equation}
holds for all $u\in \textup{P}H^2_0(B_{R_*},B'_{R_*})$.

\begin{proposition}\label{prop:transmission_isotropic}
     Let $\sigma\geq 1$ and $R_*>0$. Let the elliptic operators $L^\pm$ and the boundary operators $D^\pm$, $T^\pm$ have Gevrey-$\sigma$ coefficients. Let $b^\pm,\, n^\pm \in G^\sigma(\Omega)$. We assume that \eqref{ineq:basic_elliptic_reg_boundary_transmission} holds for any $u \in \textup{P}H^2_0(B_{R_*},B'_{R_*})$. 
     Then there exist constants $B_0,\,B\geq1$ such that for all $\kappa,m \in \mathbb{N}$ with $m\leq \kappa$, all $R \in (0,R_*]$, and for any 
    $u\in \textup{P}H^2(\Omega_2)$ satisfying \eqref{e:transmission},
    \begin{equation}\label{e:trans_bndry_isotropic}
        \begin{aligned}
        [| u^+|]_{\kappa+2,m+2; B_R\cap \R^n_+} + [| u^-|]_{\kappa+2,m+2; B_R\cap \R^n_-}\\  
        \leq  B_0^{m} B^{\kappa +2}\max \left\{
        \frac{k}{(\kappa+2)^\sigma},1\right\}^{\kappa+2} \Bigg(& \max_{\substack{0\leq \ell\leq \kappa\\0\leq \nu \leq \min\{\ell,m\}}}  \Gamma_{B,B_0,k}^{m,\ell} (\ell+2)^{(\ell +2)\sigma} \ \rho^2_*[|f|]_{\ell,\nu; B_R\cap \R^n_+}\\&
        + \max_{\substack{0\leq \ell\leq \kappa\\0\leq \nu \leq \min\{\ell,m\}}}  \Gamma_{B,B_0,k}^{m,\ell} (\ell+2)^{(\ell +2)\sigma} \  \rho^2_*[|f|]_{\ell,\nu; B_R\cap \R^n_-} \\
         &+  \max_{0\leq \ell\leq \kappa+2} \Upsilon_{B,B_0,k}^{m,\ell} \ell^{\ell\sigma}\rho_*[|h_I|]_{\ell,2,B'_R} 
        \\ & +   \max_{0\leq \ell\leq \kappa+1} \Upsilon_{B,B_0,k}^{m,\ell+1} (\ell+1)^{(\ell +1)\sigma}\rho^{1}_*[|g_I|]_{\ell,1,B'_R}    \\
        &+ k^{-1} ([|u^+|]_{1;B_R\cap \R^n_+} + [|u^-|]_{1;B_R\cap \R^n_-}  ) \\
        &+[|u^+|]_{0;B_R\cap \R^n_+} +[|u^-|]_{0;B_R\cap \R^n_-} \Bigg),
     \end{aligned}    
    \end{equation}
    with $$\Gamma_{B,B_0,k}^{m,\ell} = B_0^{-m+1} B^{-\ell-1} \left(\max \left\{k,(\ell+2)^\sigma\right\}\right)^{-\ell}$$ and $$\Upsilon_{B,B_0,k}^{m,\kappa} = B_0^{-m+1}  B^{-\ell+1} \left(\max \left\{k,\ell^\sigma\right\}\right)^{-\ell}.$$
\end{proposition}
\begin{proof}
    The proof is similar to that leading up to Proposition~\ref{prop:isotropic}, with the difference that the basic elliptic estimate is given by 
\eqref{ineq:basic_elliptic_reg_boundary_transmission}.
To avoid unnecessary repetition, the proof is omitted.
\end{proof}

\begin{proof}[Proof of Theorem~\ref{thm:transmission}]
    Let $x_0 \in \overline \Omega_1$. We identify three cases (the first two are treated as in in the proof of Theorem~\ref{thm:uttb_Gev_reg}):
    \begin{itemize}
        \item $x_0 \in \Omega\setminus I$, and is thus an interior point.
        \item $x_0 \in \partial \Omega$, and is thus a boundary point.
        \item $x_0 \in I$ is an interface point. In this case, the proof follows from mapping a neighbourhood $\mathcal{U}_2$ of the the point $x_0$ to the ball $B_R$ using a Gevrey-$\sigma$-regular map $\phi$. Applying Proposition~\ref{prop:transmission_isotropic} with $m =\kappa$ and $A= B_0B$ (which depends on $x_0$), we obtain
        \begin{equation}\label{est:lego_transmission_bdry}
       \begin{aligned}
           |u|_{\textup{P}\kappa+2;\mathcal{U}_1} 
        \leq A^{\kappa +2}\max \left(
        k, (\kappa+2)^\sigma\right)^{\kappa+2} \Bigg(& \max_{0\leq \ell\leq \kappa}  A^{-\ell-1} \frac{1}{\max (k, (\ell+2)^{\sigma})^{\ell}} |f|_{\textup{P}\ell; \mathcal{U}_2} \\&+
        \max_{0\leq \ell\leq \kappa+2}   A^{-\ell+1} \frac{1}{\max (k, \ell^{\sigma})^{\ell}}|h_I|_{\ell;\mathcal{U}_2}  
        \\ &
        +  \max_{0\leq \ell\leq \kappa+1} A^{-\ell} \frac{1}{\max (k, (\ell+1)^{\sigma})^{\ell+1}} |g_I|_{\ell;\mathcal{U}_2} \\
        &+ k^{-1} |u|_{\textup{P}1;\mathcal{U}_2} +|u|_{\textup{P}0;\mathcal{U}_2} \Bigg)
        \end{aligned}
\end{equation}
        where the neighbourhood $\mathcal{U}_1:=\phi^{-1}(V_{R/2})$. The steps to obtain \eqref{est:lego_transmission_bdry} from \eqref{e:trans_bndry_isotropic} are similar to those taken in the proof of Theorem~\ref{thm:int_Gev_reg} and are omitted.
    \end{itemize}
    Finally, by varying $x_0 \in \overline \Omega_1$, we can extract a finite covering of the compact set $\overline \Omega_1$ by open sets $\mathcal{U}_1$. The proof is then complete by combining a finite number of estimates \eqref{est:lego_inter}, \eqref{est:lego_bdry}, and \eqref{est:lego_transmission_bdry}.
\end{proof}

\section*{Acknowledgements}

 JG, MM, and ES were supported by ERC Synergy Grant ``PSINumScat" 101167139.
JG was supported by EPSRC grants EP/V001760/1 and EP/V051636/1 and  the Leverhulme Trust under Research Project Grant  RPG-2023-325. 

ChatGPT version 5.5 was used (i) to prove some of the combinatorial relations in Lemma \ref{l:cubeApproximate}, (ii) to search the literature, and (iii) to help proofread various parts of the paper.

\footnotesize{
\bibliographystyle{plain}
\bibliography{references}

@book{costabel2010corner,
  title={Corner Singularities and Analytic Regularity for Linear Elliptic Systems. Part I: Smooth domains.},
  author={Costabel, M. and Dauge, M. and Nicaise, S.},
  year={2010},
publisher={}
}

@article{galkowski2025numerical,
  title={{Numerical analysis of the high-frequency Helmholtz equation using semiclassical analysis}},
  author={Galkowski, J. and Spence, E. A.},
     journal={Acta Numerica},
  volume={35},
  pages={1--171},
  year={2026},
  publisher={Cambridge University Press}
}

@article{VIBurenkov_1997,
doi = {10.1070/IM1997v061n01ABEH000103},
url = {https://doi.org/10.1070/IM1997v061n01ABEH000103},
year = {1997},
month = {feb},
publisher = {},
volume = {61},
number = {1},
pages = {1},
author = {V. I. Burenkov and A. L. Gorbunov},
title = {Exact bounds for the minimum norm of extension operators for Sobolev spaces},
journal = {Izvestiya: Mathematics}
}

@article{agmon1964estimates,
	title={Estimates near the boundary for solutions of elliptic partial differential equations satisfying general boundary conditions {II}},
	author={Agmon, S. and Douglis, A. and Nirenberg, L.},
	journal={Communications on Pure and Applied Mathematics},
	volume={17},
	number={1},
	pages={35--92},
	year={1964},
	publisher={Wiley Online Library}
}

@book{mclean2000strongly,
  title={Strongly Elliptic Systems and Boundary Integral Exquations},
  author={McLean, W.},
  year={2000},
  publisher={Cambridge university press}
}

@article{constantine1996multivariate,
  title={A multivariate {F}aa di {B}runo formula with applications},
  author={Constantine, G. and Savits, T.},
  journal={Transactions of the American Mathematical Society},
  volume={348},
  number={2},
  pages={503--520},
  year={1996}
}

@book{melenk2004hp,
  title={{hp}-Finite Element Methods for Singular Perturbations},
  author={Melenk, J. M.},
  year={2004},
  publisher={Springer}
}

@article{lafontaine2021most,
  title={For Most Frequencies, Strong Trapping Has a Weak Effect in Frequency-Domain Scattering},
  author={Lafontaine, D. and Spence, E. A. and Wunsch, J.},
  journal={Communications on Pure and Applied Mathematics},
  volume={74},
  number={10},
  pages={2025--2063},
  year={2021},
  publisher={Wiley Online Library}
}

@article{galkowski2024hp,
  title={{The $ hp $-FEM applied to the Helmholtz equation with PML truncation does not suffer from the pollution effect}},
  author={Galkowski, J. and Lafontaine, D. and Spence, E. A. and Wunsch, J.},
  journal={Communications in Mathematical Sciences},
  volume={22},
  number={7},
  pages={1761--1816},
  year={2024},
  publisher={International Press of Boston}
}

@article{galkowski2023decompositions,
  title={Decompositions of high-frequency Helmholtz solutions via functional calculus, and application to the finite element method},
  author={Galkowski, J. and Lafontaine, D. and Spence, E. A. and Wunsch, J.},
  journal={SIAM Journal on Mathematical Analysis},
  volume={55},
  number={4},
  pages={3903--3958},
  year={2023},
  publisher={SIAM}
}

@article{melenk2010convergence,
  title={Convergence analysis for finite element discretizations of the {H}elmholtz equation with {D}irichlet-to-{N}eumann boundary conditions},
  author={Melenk, J. M. and Sauter, S. A.},
  journal={Math. Comp},
  volume={79},
  number={272},
  pages={1871--1914},
  year={2010}
}

@article{melenk2011wavenumber,
  title={Wavenumber explicit convergence analysis for Galerkin discretizations of the Helmholtz equation},
  author={Melenk, J. M. and Sauter, S. A.},
  journal={SIAM Journal on Numerical Analysis},
  volume={49},
  number={3},
  pages={1210--1243},
  year={2011},
  publisher={SIAM}
}

@article{lafontaine2022wavenumber,
  title={{Wavenumber-explicit convergence of the $hp$-FEM for the full-space heterogeneous Helmholtz equation with smooth coefficients}},
  author={Lafontaine, D. and Spence, E. A. and Wunsch, J.},
  journal={Comp. Math. Appl.},
  volume={113},
  pages={59-69},
  year={2022},
}

@article{bernkopf2025wavenumber,
  title={{Wavenumber-explicit stability and convergence analysis of $hp$ finite element discretizations of Helmholtz problems in piecewise smooth media}},
  author={Bernkopf, M. and Chaumont-Frelet, T. and Melenk, J. M.},
  journal={Mathematics of Computation},
  volume={94},
  number={351}, 
  pages={73--122},
  year={2025}
}

@article{babuska2000pollution,
  title={Is the pollution effect of the {FEM} avoidable for the {H}elmholtz equation considering high wave numbers?},
  author={Babu\v{s}ka, I. M and Sauter, S. A.},
  journal={SIAM Review},
  pages={451--484},
  year={2000},
}

@article{spence2023simple,
  title={A simple proof that the {$hp$}-{FEM} does not suffer from the pollution effect for the constant-coefficient full-space {H}elmholtz equation},
  author={Spence, E. A.},
  journal={Advances in Computational Mathematics},
  volume={49},
  number={2},
  pages={27},
  year={2023},
  publisher={Springer}
}

@article{bruna1980extension,
  title={An Extension Theorem of {W}hitney Type for Non Quasi-Analytic Classes of Functions},
  author={Bruna, J.},
  journal={Journal of the London Mathematical Society},
  volume={2},
  number={3},
  pages={495--505},
  year={1980},
  publisher={Wiley Online Library}
}

@article{FeSc:20,
  title={Exponential convergence in {$H^1$} of $hp$-{FEM} for {G}evrey regularity with isotropic singularities},
  author={Feischl, M. and Schwab, Ch.},
  journal={Numerische Mathematik},
  volume={144},
  number={2},
  pages={323--346},
  year={2020},
  publisher={Springer}
}

@article{AGS2,
  title={Non-uniform finite-element meshes defined by ray dynamics for {H}elmholtz problems},
  author={Averseng, M. and Galkowski, J. and Spence, E. A.},
  journal={arXiv:2506.15630},
  year={2025}
}

@article{ihlenburg1995finite,
  title={Finite element solution of the {H}elmholtz equation with high wave number {P}art {I}: The $h$-version of the {FEM}},
  author={Ihlenburg, F. and Babu{\v{s}}ka, I.},
  journal={Computers \& Mathematics with Applications},
  volume={30},
  number={9},
  pages={9--37},
  year={1995},
  publisher={Elsevier}
}

@article{ihlenburg1997finite,
  title={Finite element solution of the {H}elmholtz equation with high wave number {P}art {II}: the $hp$-version of the {FEM}},
  author={Ihlenburg, F. and Babu\v{s}ka, I.},
  journal={SIAM Journal on Numerical Analysis},
  volume={34},
  number={1},
  pages={315--358},
  year={1997},
  publisher={SIAM}
}

@book{DyZw:19,
  title={Mathematical theory of scattering resonances},
  author={Dyatlov, S. and Zworski, M.},
  publisher={American Mathematical Society},
  series={Graduate Studies in Mathematics},
  volume={200},
  year={2019}
}

@article{GLS2,
  title={Perfectly-matched-layer truncation is exponentially accurate at high frequency},
  author={Galkowski, J. and Lafontaine, D. and Spence, E. A.},
  journal={SIAM J. Math. Anal.},
  volume={55},
  number={4},
 pages={3344-3394},
  year={2023}
}

@article{galkowski2025sharp,
  title={{Sharp preasymptotic error bounds for the Helmholtz $h$-FEM}},
  journal={SIAM Journal on Numerical Analysis},
  volume={63},
  author={Galkowski, J. and Spence, E. A.},
  number={1},
  pages={1--22},
  year={2025},
  publisher={SIAM}
}

@book{schmudgen2012unbounded,
  title={Unbounded self-adjoint operators on Hilbert space},
  author={Schm{\"u}dgen, K.},
  year={2012},
  publisher={Springer Science \& Business Media}
}

@article{Sa:06,
  title={{A refined finite element convergence theory for highly indefinite Helmholtz problems}},
  author={Sauter, S. A.},
  journal={Computing},
  volume={78},
  number={2},
  pages={101--115},
  year={2006},
  publisher={Springer}
}

@article{Sc:74,
  title={{An observation concerning Ritz-Galerkin methods with indefinite bilinear forms}},
  author={Schatz, A. H.},
  journal={Math. Comp.},
  volume={28},
  number={128},
  pages={959--962},
  year={1974}
}

@article{ScWa:96,
  title={{Some new error estimates for Ritz--Galerkin methods with minimal regularity assumptions}},
  author={Schatz, A. H. and Wang, J.},
  journal={Mathematics of Computation},
  volume={65},
  number={213},
  pages={19--27},
  year={1996}
}

@article{ChNi:20,
  title={Wavenumber explicit convergence analysis for finite element discretizations of general wave propagation problems},
  author={Chaumont-Frelet, T. and Nicaise, S.},
  journal={IMA Journal of Numerical Analysis},
  volume={40},
  number={2},
  pages={1503--1543},
  year={2020},
  publisher={Oxford University Press}
}

@article{BabuskaGuo1988RegularityI,
  author  = {I. Babu{\v{s}}ka and B. Guo},
  title   = {Regularity of the Solution of Elliptic Problems with Piecewise Analytic Data. {Part I}: Boundary Value Problems for Linear Elliptic Equation of Second Order},
  journal = {SIAM Journal on Mathematical Analysis},
  volume  = {19},
  number  = {1},
  pages   = {172--203},
  year    = {1988},
  doi     = {10.1137/0519014}
}

@book{Schwab1998,
  author    = {C. Schwab},
  title     = {$p$- and $hp$-Finite Element Methods:
               Theory and Applications in Solid and Fluid Mechanics},
  publisher = {Clarendon Press},
  address   = {Oxford},
  year      = {1998},
  series    = {Numerical Mathematics and Scientific Computation},
  isbn      = {978-0-19-850390-3}
}

@article{GuoBabuska1986Part1,
  author  = {B. Guo and I. Babu{\v{s}}ka},
  title   = {The h-p Version of the Finite Element Method. {Part 1}: The Basic Approximation Results},
  journal = {Computational Mechanics},
  volume  = {1},
  number  = {1},
  pages   = {21--41},
  year    = {1986},
  doi     = {10.1007/BF00298660}
}

@article{GuoBabuska1986Part2,
  author  = {B. Guo and I. Babu{\v{s}}ka},
  title   = {The h-p Version of the Finite Element Method. {Part 2}: General Results and Applications},
  journal = {Computational Mechanics},
  volume  = {1},
  number  = {4},
  pages   = {203--220},
  year    = {1986},
  doi     = {10.1007/BF00270303}
}

@article{BabuskaGuo1989RegularityII,
  author  = {I. Babu{\v{s}}ka and B. Q. Guo},
  title   = {Regularity of the Solution of Elliptic Problems with Piecewise Analytic Data. {II}: The Trace Spaces and Application to the Boundary Value Problems with Nonhomogeneous Boundary Conditions},
  journal = {SIAM Journal on Mathematical Analysis},
  volume  = {20},
  number  = {4},
  pages   = {763--781},
  year    = {1989},
  doi     = {10.1137/0520054}
}

@article{melenk2026wavenumber,
  title={{Wavenumber-explicit $hp$-FEM analysis of Maxwell's equations with impedance boundary conditions in piecewise smooth media}},
  author={Melenk, J. M. and W{\"o}rg{\"o}tter, D.},
  journal={arXiv preprint arXiv:2603.17467},
  year={2026}
}

@article{melenk2024wavenumber,
  title={{Wavenumber-explicit $hp$-FEM analysis for Maxwell’s equations with impedance boundary conditions}},
  author={Melenk, J. M. and Sauter, S. A.},
  journal={Foundations of Computational Mathematics},
  volume={24},
  number={6},
  pages={1871--1939},
  year={2024},
  publisher={Springer}
}

@article{melenk2021wavenumber,
  title={{Wavenumber-explicit $hp$-FEM analysis for Maxwell’s equations with transparent boundary conditions}},
  author={Melenk, J. M. and Sauter, S. A.},
  journal={Foundations of Computational Mathematics},
  volume={21},
  number={1},
  pages={125--241},
  year={2021},
  publisher={Springer}
}

@misc{Di:26,
	Howpublished = {Digital Library of Mathematical Functions, \url{http://dlmf.nist.gov/}},
	Title = {{Digital Library of Mathematical Functions}},
	Author = {{NIST}},
	Year = 2026}

@article{costabel2005exponential,
  title={Exponential convergence of hp-FEM for Maxwell equations with weighted regularization in polygonal domains},
  author={Costabel, M. and Dauge, M. and Schwab, C.},
  journal={Mathematical Models and Methods in Applied Sciences},
  volume={15},
  number={04},
  pages={575--622},
  year={2005},
  publisher={World Scientific}
}
}
\end{document}